\documentclass{amsart}

\usepackage{amsthm, amsmath, amsfonts, amscd, amssymb}
\usepackage{mathrsfs}

\usepackage[breaklinks,colorlinks,citecolor=teal,linkcolor=teal,urlcolor=teal,hyperindex]{hyperref}
\usepackage{color}
\usepackage{tikz}
\usetikzlibrary{arrows} \usetikzlibrary{patterns}
\usepackage[all]{xy}

\usepackage[top=2.5cm,bottom=2.5cm,left=2.6cm,right=2.6cm]{geometry}

\newtheorem{theorem}{Theorem}[section]
\newtheorem{lemma}[theorem]{Lemma}
\newtheorem{prop}[theorem]{Proposition}
\newtheorem{coro}[theorem]{Corollary}

\newtheorem{conj}[theorem]{Conjecture}

\newtheorem{relation}[theorem]{Relation}

\newtheorem{defn}[theorem]{Definition}

\theoremstyle{definition}

\newtheorem{remark}[theorem]{Remark}

\numberwithin{equation}{section}

\def \fk{\mathfrak}
\def \msf{\mathsf}
\def \mb{\mathbb}
\def \mc{\mathcal}
\def \scr{\mathscr}
\def \inv{^{-1}}
\def \0{\infty}
\def \v{\vskip 0.1in}
\def \n{\noindent}
\def \cplane{\mathbb{C}}
\def \integer{\mathbb{Z}}
\def \rone{\mathbb R}
\def \ration{\mathbb{Q}}
\def \qq{\quad}
\def \p{\partial}

\def \P{\mb P}
\def \rto{\rightarrow}
\def \hrto{\hookrightarrow}
\def \rrto{\rightrightarrows}
\def \one{\mathbf 1}
\def \ball{{\mathbb D}}
\def \dimc{\dim_\cplane}
\def \codimc{\text{codim}_{\cplane}}

\def \dimr{\dim_\rone}
\def \co{\colon\thinspace}
\def \C{{\msf C}}
\def \N{{\msf N}}

\def \X{{\msf X}}

\def \W{{\msf W}}
\def \Y{{\msf Y}}
\def \L{{\msf L}}
\def \E{{\msf E}}
\def \I{{\msf I}}
\def \os{\msf S}

\def \U{\msf U}
\def \D{\msf D}
\def \I{{\msf I}}
\def \f{{\msf f}}
\def \x{\msf x}
\def \y{\msf y}
\def \z{\msf z}
\def \T{\scr T}

\def \h{\fk h}

\def \wa{{\fk a}}
\def \wb{{\fk b}}
\def \uX{{\underline{\X}}}
\def \Xa{{\underline{\msf X}_{\fk a}}}
\def \Na{{\underline{\msf N}_{\fk a}}}
\def \Ea{{\underline{\msf E}_{\fk a}}}
\def \PE{{\msf P\E}}
\def \PN{{\msf P\N}}
\def \uE{{\underline{\E}}}
\def \oE{{\overline{\E}}}
\def \uN{{\underline{\N}}}
\def \cN{{\mc N}}
\def \oN{{\overline{\N}}}
\def \M{{\overline{\scr M}}}
\def \F{{\overline{\scr F}}}
\def \po{{\stackrel{\circ}{\prec} }}
\def \hol{{\msf{Hol}}}
\def \mhol{{\mathrm{Hol}}}
\def \aut{{\mathrm{Aut}}}
\def \bl{{\Big\langle}}
\def \br{{\Big\rangle}}
\def \bb{{\Big |}}
\def \simple{{\text{simple}}}
\def \vir{{\text{vir}}}

\def \uh{{\underline{\h}}}

\def \ugamma{{\underline{\Gamma}}}
\def \bgamma{{\overline{\Gamma}}}

\allowdisplaybreaks

\begin{document}

\title[Orbifold Gromov--Witten theory of weighted blowups]
{Orbifold Gromov--Witten theory of weighted blowups}

\author{Bohui Chen}
\address{Department of Mathematics and Yangtz center of Mathematics, Sichuan University, Chengdu 610064, China}
\email{chenbohui@scu.edu.cn}

\author{Cheng-Yong Du}
\address{School of mathematics and V.\thinspace C. \& V.\thinspace R. Key Lab, Sichuan Normal University, Chengdu 610068, China}
\email{cyd9966@hotmail.com}

\author{Rui Wang}
\address{Department of Mathematics, University of California, Berkeley, CA 94720-3840, USA}
\email{ruiwang@berkeley.edu}

\thanks{}

\date{}

\begin{abstract}
Consider a compact symplectic sub-orbifold groupoid $\os$ of a compact symplectic orbifold groupoid $(\X,\omega)$. Let $\Xa$ be the weight-$\wa$ blowup of $\X$ along $\os$, and $\D_\wa=\PN_\wa$ be the exceptional divisor, where $\N$ is the normal bundle of $\os$ in $\X$. In this paper we show that the absolute orbifold Gromov--Witten theory of $\Xa$ can be effectively and uniquely reconstructed from the absolute orbifold Gromov--Witten theories of $\X$, $\os$ and $\D_\wa$, the natural restriction homomorphism $H^*_{\text{CR}}(\X)\rto H^*_{\text{CR}}(\os)$ and the first Chern class of the tautological line bundle over $\D_\wa$. To achieve this we first prove similar results for the relative orbifold Gromov--Witten theories of $(\Xa| \D_\wa)$ and $(\Na|\D_\wa)$. As applications of these results, we prove an orbifold version of a conjecture of Maulik--Pandharipande on the Gromov--Witten theory of blowups along complete intersections, a conjecture on the Gromov--Witten theory of root constructions and a conjecture on Leray--Hirsch result for orbifold Gromov--Witten theory of Tseng--You.
\end{abstract}

\maketitle

\tableofcontents

\section{Introduction}

Symplectic birational geometry, proposed by Li--Ruan \cite{Lit-Ruan2009},  studies symplectic birational cobordism invariants defined via Gromov--Witten theory. Two symplectic manifolds are called symplectic birational cobordant (cf. \cite{Guillemin-Sternberg1989,Hu-Li-Ruan2008}) if one can be obtained from the other by a sequence of Hamiltonian $S^1$ symplectic reductions.  Guillemin--Sternberg \cite{Guillemin-Sternberg1989} proved that every symplectic birational cobordism can be realized by finite times symplectic blowups/blow-downs and $\integer$-linear deformations of symplectic forms\footnote{A $\integer$-linear deformation of a symplectic form $\omega$ on a manifold $X$ is a path of symplectic form $\omega+t\kappa$, $t\in I$, where $\kappa$ is a closed $2$-form representing an integral class and $I$ is an interval. See for example \cite[Definition 2.5]{Hu-Li-Ruan2008}.}. With noticing Gromov--Witten invariants are preserved under smooth deformations of symplectic structures (see for example \cite{Ruan1999b,Chen-Ruan2002}), to understand the relation of Gromov--Witten invariants between two symplectic birational corbordant manifolds, it is enough to take care of the change of Gromov--Witten invariants after a symplectic blowup/blow-down.

The first nontrivial invariant in symplectic birational geometry is the symplectic uniruledness, which was discovered by Hu--Li--Ruan \cite{Hu-Li-Ruan2008}. Since it was proved by Koll\'ar \cite{Kollar98} and Ruan \cite{Ruan1999b} that a smooth projective variety is uniruled if and only if it is symplectically uniruled, the symplectic uniruledness can be regarded as a symplectic generalization of the uniruledness in birational algebraic geometry. Further motivated by Hu--Li--Ruan's work, people conjectured that the symplectic rational connectedness as a symplectic generalization (cf. \cite{Hu-Ruan2013,Lit-Ruan2009}) of rational connectedness for a projective variety is also a symplectic birational invariant. This conjecture is still open and reader can find some recent progress along this topic from \cite{Voisin2008,Hu-Ruan2013,Tian2012,Tian2015}

In fact, the most natural object from symplectic reduction should be symplectic \emph{orbifold} instead of symplectic manifold. Thus it is more natural and also more interesting to study symplectic birational geometry in the category of symplectic orbifolds. At the same time, though symplectic uniruledness and symplectic rational connectedness are defined using genus zero Gromov--Witten invariants only, it is natural to study symplectic geometry using \emph{higher genus} Gromov--Witten invariants and unravel the relation between the Gromov--Witten theories of a symplectic manifold/orbifold and its blowup, as blowups are building blocks of symplectic birational cobordisms.

Further, for a symplectic orbifold, besides the usual symplectic blowup along a symplectic sub-orbifold, one can also consider the \emph{weighted} blowup, and as it is shown in \cite{Chen-Du-Hu2019} by Hu and the first two authors that the weighted blowup instead of the usual one is the proper version for the orbifold version of symplectic birational geometry.

To be concrete, now we assume $(\X,\omega)$ is a compact symplectic orbifold, and $\os$ is a codimension $2n$ compact symplectic sub-orbifold of $\X$.  Denote by $\Xa$ the weight-$\wa$ blowup of $\X$ along $\os$ (cf. \S \ref{subsec weighted-blps}), where $\wa=(a_1,\ldots, a_n)\in (\integer_{\geq 1})^n$ is the blowup weight. In this paper, we focus on  studying the relation between the orbifold Gromov--Witten theory of $\X$ and of $\Xa$.

There are already some work on orbifold Gromov--Witten theory $\Xa$ for the case that $\os$ is a symplectic divisor. When $\X$ is a Deligne--Mumford stack, the weight-$\wa=(r)$ blowup $\Xa=\uX_{(r)}$ of $\X$ along $\os$ is also called a root stack (cf. \cite{Chen-Du-WangR2020a}). In \cite{Tseng-You2016a} Tseng--You studied the orbifold Gromov--Witten theory of $\uX_{(r)}$ and conjectured that the orbifold Gromov--Witten theory of $\uX_{(r)}$ is determined by the orbifold Gromov--Witten theories of $\X$ and $\os$ and the restriction map $H^*_{\text{CR}}(\X)\rto H^*_{\text{CR}}(\os)$. They proved their conjecture for the case that $\os$ is a smooth manifold. The case that $\os$ is an orbifold is still open. We will prove this conjecture in this paper. The exceptional divisor $\D_{(r)}$ of $\uX_{(r)}$ is a $\integer_r$-gerbe over $\os$. It is a natural example of root gerbe. Orbifold Gromov--Witten theory of gerbes is extensively studied by people. See \cite{Andreini-Jiang-Tseng2011,Andreini-Jiang-Tseng2015,Andreini-Jiang-Tseng2016,Tang-Tseng2016}.

We now describe our approaches and results. Denote by $\N:=\N_{\os|\X}$ the normal bundle of $\os$ in $\X$, and by $\D_\wa$ the exceptional divisor of the weighted blowup $\Xa$. Notice that the exceptional divisor $\D_\wa$ can be considered as the weight-$\wa$ projectivization $\PN_\wa$ of $\N$, it is an orbifold fiber bundle over $\os$ with fiber being the weighted projective space $\msf P_\wa$. There is a tautological line bundle over $\D_\wa=\PN_\wa$ coming from the tautological line bundle $\mc O(-1)$ over $\msf P_\wa$. We denote it by $\mc O_{\PN_\wa}(-1)$ or $\mc O_{\D_\wa}(-1)$. It is the normal line bundle of $\D_\wa$ in $\Xa$. On the other hand, let $\oN_\wa$ be the weight-$\wa$ projectification of $\N$. Then $\D_\wa$ is the infinite divisor of $\oN_\wa$. The normal line bundle of $\D_\wa$ in $\oN_\wa$ is $\mc O_{\D_\wa}(1)$, the dual line bundle of $\mc O_{\D_\wa}(-1)$. Then by \cite[Proposition 4.1]{Chen-Li-Sun-Zhao2011}, there is a degeneration of $\X$
\begin{align}\label{E degenerate-X-in-sec1}
\X\xrightarrow{\text{degenerate}} (\Xa|\D_\wa)\wedge_{\D_\wa}(\oN_\wa|\D_\wa),
\end{align}
where ``$\wedge_{\D_\wa}$'' means the gluing is along the exceptional divisor $\D_\wa\subseteq \Xa$ and the infinite divisor $\D_\wa\subseteq\oN_\wa$. Similarly, let $\overline{\mc O_{\D_\wa}(-1)}$ be projectification of $\mc O_{\D_\wa}(-1)$ with trivial weight, that is $\overline{\mc O_{\D_\wa}(-1)}=\P(\mc O_{\D_\wa}(-1)\oplus \mc O_{\D_\wa})$. Let $\D_{\wa,\0}=\P(0\oplus\mc O_{\D_\wa})\cong \D_\wa$ be the infinite divisor of $\overline{\mc O_{\D_\wa}(-1)}$, whose normal line bundle in $\overline{\mc O_{\D_\wa}(-1)}$ is $\mc O_{\D_\wa}(1)$. Then we get a degeneration of $\Xa$ as \eqref{E degenerate-X-in-sec1}
\begin{align}\label{E degenerate-Xa-in-sec1}
\Xa\xrightarrow{\text{degenerate}} (\uX_\wa|\D_\wa)\wedge_{\D_\wa} (\overline{\mc O_{\D_\wa}(-1)}|\D_{\wa,\0}).
\end{align}
As in \eqref{E degenerate-X-in-sec1}, here ``$\wedge_{\D_\wa}$'' means the gluing is along the exceptional divisor $\D_\wa\subseteq \Xa$ and the infinite divisor $\D_\wa\cong \D_{\wa,\0}\in \overline{\mc O_{\D_\wa}(-1)}$. Therefore, by the degeneration formula (cf. \cite{Chen-Li-Sun-Zhao2011,Abramovich-Fantechi2016}), to study the relation between orbifold Gromov--Witten theories of $\Xa$ and $\X$, we only need to compare the relative orbifold Gromov--Witten theories of $(\oN_\wa|\D_\wa)$ and $(\overline{\mc O_{\D_\wa}(-1)}|\D_\wa)$.

On the other hand, in \cite{Chen-Du-Hu2019}, Hu and the first two authors proved that for the degeneration \eqref{E degenerate-X-in-sec1}, there is an invertible lower triangle system which relates relative orbifold Gromov--Witten invariants of $(\Xa|\D_\wa)$ and absolute orbifold Gromov--Witten invariants of $\X$ relative to $\os$, whose entries are relative orbifold Gromov--Witten invariants of $(\oN_\wa|\D_\wa)$. Consequently if we could determine the relative orbifold Gromov--Witten invariants of $(\overline{\mc O_{\D_\wa}(-1)}|\D_{\wa,\infty})$ and of $(\oN_\wa|\D_\wa)$, we could determine the orbifold Gromov--Witten invariants of $\Xa$ from the degeneration \eqref{E degenerate-Xa-in-sec1}.

\subsection{Orbifold Gromov--Witten theories of weighted projectifications}
\label{subsec introduction-GW-of-weight--projectification-blp}

We first study relative orbifold Gromov--Witten invariants of
weighted projectifications of vector bundles, which can apply to both $(\oN_\wa|\D_\wa)$ and $(\overline{\mc O_{\D_\wa}(-1)}|\D_{\wa,\infty})$.

Let $\pi\co \E\rto \os$ be a rank $2n$ symplectic orbifold vector bundle over a compact symplectic orbifold groupoid $\os$. Let $\oE_\wa$ be the weight-$\wa$ projectification of $\E$ and $\PE_\wa$ be the infinite divisor, which is the weight-$\wa$ projectivization of $\E$ (cf. \S \ref{subsec weight--projfi-viza}). The normal bundle $\N_{\PE_\wa|\oE_\wa}$ of $\PE_\wa$ in $\oE_\wa$ is $\mc O_{\PE_\wa}(1)$, the dual line bundle of the tautological line bundle $\mc O_{\PE_\wa}(-1)$. Our first result is

\begin{theorem}\label{thm rel-GW-of-P(E)}
The relative descendent orbifold Gromov--Witten theory of the pair $(\oE_\wa| \PE_\wa)$ can be effectively and uniquely reconstructed from the absolute descendent orbifold Gromov--Witten theories of $\os$ and $\PE_\wa$, the Chern classes of $\E$ and $\mc O_{\PE_\wa}(-1)$.
\end{theorem}

By virtual localization, we could first reduce the determination of relative descendent orbifold Gromov--Witten invariants of the pair $(\oE_\wa| \PE_\wa)$ to orbifold Gromov--Witten invariants of $\os$ twisted by $\E$ and rubber invariants with $\Psi_\0$-integrals associated to the orbifold line bundle $\mc O_{\PE_\wa}(1)\rto\PE_\wa$, i.e. rubber invariants with $\Psi_\0$-integrals of the projectification of $\mc O_{\PE_\wa}(1)$:
\[
(\D_0|\Y|\D_\0):=\big(\P(0\oplus\mc O_{\PE_\wa})\mid \P(\mc O_{\PE_\wa}(1)\oplus\mc O_{\PE_\wa}) \mid \P(\mc O_{\PE_\wa}(1)\oplus 0)\big).
\]
In \S \ref{subsec rubber-calculus} by rubber calculus we could remove those $\Psi_\0$-integrals, and reduce the resulting rubber invariants without $\Psi_\0$-integrals into certain relative descendent orbifold Gromov--Witten invariants of $(\D_0|\Y|\D_\0)$, called {\em distinguished type II invariants}. Then we use an induction algorithm to determine all these distinguished type II invariants of $(\D_0|\Y|\D_\0)$. Among them, fiber class invariants are the initial values of this inductive algorithm. We will determine fiber class invariants for general orbifold $\P^1$-bundles in \S \ref{sec fiber-class-inv}. The inductive algorithm is described in \S \ref{subsec dertermine-dis-type-ii}. This inductive algorithm depends on a partial order over all distinguished type II invariants and the weighted-blowup correspondence in \cite{Chen-Du-Hu2019}. Finally we prove Theorem \ref{thm rel-GW-of-P(E)} in \S \ref{subsec proof-thm-1.1}.

Recently, Janda--Pandharipande--Pixton--Zvonkine computed the double ramification cycles with general smooth algebraic varieties as targets in \cite{Janda-Pandharipande-Pixton-Zvonkine2018}. As a direct consequence, rubber invariants without $\Psi_\0$-integrals associated to a smooth line bundle $L\rto X$ are determined by the Gromov--Witten theory of $X$ and $c_1(L)$. It is natural to expect a formula for the double ramification cycles with orbifold targets. Such a formula would also implies that rubber invariants without $\Psi_\0$-integrals associated to the orbifold line bundle $\mc O_{\PE_\wa}(1)\rto\PE_\wa$ are determined by the orbifold Gromov--Witten theory of $\PE_\wa$ and $c_1(\mc O_{\PE_\wa}(1))$. We will study this in \cite{Chen-Du-WangR2020a}.

\subsection{Orbifold Gromov--Witten theory of weighted blowups}
\label{subsec introduction-orb-GW-of-weight--blp}
Now we can determine the orbifold Gromov--Witten theory of $\Xa$, the weight-$\wa$ blowup of $\X$ along $\os$. There is a natural orbifold morphism
\[
\kappa\co\uX_\wa\rto\X,
\]
which induces a morphism on inertia spaces
\[
\I \kappa=\coprod_{(h)\in \T^{\uX_\wa}} \kappa_{(h)}\co \coprod_{(h)\in \T^{\uX_\wa}}\uX_\wa(h)\rto\coprod_{(h)\in \T^{\uX_\wa}}\X(\kappa_t(h)),
\]
where $\kappa_t$ is the induced map on the index sets of connected components of inertia spaces. In Definition \ref{def admissible-abs-insert-K}, \S \ref{subsec cohomology} we set
\[
\mc K:=\bigoplus_{(h)\in \T^{\uX_\wa}} \kappa_{(h)}^* H^*(\X(\kappa_t(h))),
\]
to be the image of the induced homomorphism on cohomologies. The map $\kappa_{(h)}^*$ is injective for each $(h)\in\T^{\uX_\wa}$. We fix in \S \ref{subsec cohomology} a basis $\Sigma_\star$ of $H^*_{\text{CR}}(\D_\wa)=H^*_{\text{CR}}(\PN_\wa)$ (actually for the Chen--Ruan cohomology of $\PE_\wa$ for a general symplectic orbifold vector bundle $\E\rto\os$) by choosing a basis $\sigma_\star$ of $H^*_{\text{CR}}(\os)$. We denote the dual basis of $\Sigma_\star$ by $\Sigma^\star$.

\begin{defn}\label{def admissible-rel-inv}
We call a relative descendent orbifold Gromov--Witten invariant of ${(\uX_\wa| \D_\wa)}$
\[
\bl \prod_i \tau_{k_i}\gamma_i\Big|\mu\br^{ {(\uX_\wa\mid \D_\wa)}}_{g,\beta}
\]
{\em admissible} if $\gamma_i\in\mc K$ and $\mu$ is a relative insertion weighted by the chosen basis $\Sigma^\star$, i.e.
\[
\check\mu=((\mu_1,\theta_{(h_1)}),\ldots,(\mu_{\ell(\mu)}, \theta_{(h_{\ell(\mu)})}))
\]
with $\theta_{(h_i)}\in\Sigma_\star$.
\end{defn}

Given an admissible relative orbifold Gromov--Witten invariant of $(\uX_\wa\mid \D_\wa)$
\begin{align}\label{E an-admissible-rel-inv}
\bl \prod_i \tau_{k_i}\gamma_i\bb\mu\br^{ {(\uX_\wa\mid \D_\wa)}}_{g,\beta},
\end{align}
we assign it an absolute orbifold Gromvo--Witten invariant of $\X$
\begin{align}\label{E the-corresponding-abs-inv}
\bl \prod_i \tau_{k_i}\bar\gamma_i\cdot \, \mu_\os\br^\X_{g,p_*\beta}
\end{align}
by the weighted-blowup absolute/relative correspondence in \cite[\S 6.2]{Chen-Du-Hu2019}, where $\bar\gamma_i$ is the inverse image of $\gamma_i$ under $\kappa_{(h)}^*$, $\mu_\os$ supports over $\I\os$ and is determined by $\mu$ and the restriction map $H^*_{\text{CR}}(\X)\rto H^*_{\text{CR}}(\os)$ (see \S \ref{sec rel-GW-of-weit-blp} for the explicit expression of $\mu_\os$). There is a partial order over all admissible relative invariants of $(\Xa|\D_\wa)$ given in \cite[\S 6.1]{Chen-Du-Hu2019} (see \S \ref{sec rel-GW-of-weit-blp} for the presentation). It was proved in \cite{Chen-Du-Hu2019} that the degeneration formula \cite{Chen-Li-Sun-Zhao2011,Abramovich-Fantechi2016} gives us an invertible lower triangle system that related all admissible relative invariants of the form \eqref{E an-admissible-rel-inv} and the corresponding absolute invariant \eqref{E the-corresponding-abs-inv}, whose entries are relative invariants of $(\oN_\wa|\D_\wa)$. With Theorem \ref{thm rel-GW-of-P(E)} we prove in \S \ref{sec rel-GW-of-weit-blp} that

\begin{theorem}\label{thm rel-GW-of-weighted-blp}
The admissible relative descendent orbifold Gromov--Witten theory of $ {(\uX_\wa| \D_\wa)}$ can be uniquely and effectively reconstructed from the orbifold Gromov--Witten theories of $\X$, $\os$ and $\D_\wa$, the restriction map $H^*_{\text{\em CR}}(\X)\rto H^*_{\text{\em CR}}(\os)$ and the first Chern class of $\mc O_{\D_\wa}(-1)$.
\end{theorem}

In particular, when $\os$ is a divisor and the blowup weight $\wa=(1)$, the map $\kappa\co {(\uX_\wa| \D_\wa)\rto(\X| \os)}$ is identity, hence $\mc K=H^*_{\text{CR}}(\uX_\wa)=H^*_{\text{CR}}(\X)$. Therefore all relative invariants of $ {(\uX_\wa| \D_\wa)=(\X| \os)}$ are admissible. Since now $c_1(\N_{\os|\X})$ is determined by the restriction map $H^*_{\text{CR}}(\X)\rto H^*_{\text{CR}}(\os)$ (cf. \eqref{eq determine-c(N)}) and $\mc O_{\D_\wa}(-1)=\N_{\os|\X}$, as a direct consequence of Theorem \ref{thm rel-GW-of-weighted-blp} we have

\begin{coro}\label{coro rel-inv-X-S}
When $\os$ is a divisor, the relative descendent orbifold Gromov--Witten theory of $ {(\X| \os)}$ can be uniquely and effectively reconstructed from the orbifold Gromov--Witten theories of $\X$, $\os$, and the restriction map $H^*_{\text{\em CR}}(\X)\rto H^*_{\text{\em CR}}(\os)$.
\end{coro}

This result extends \cite[Theorem 2]{Maulik-Pandharipande2006} to the orbifold case.

After determining all admissible relative invariants of $(\Xa|\D_\wa)$, we could determine the absolute invariants of $\Xa$. The inertia orbifold groupoid $\I\D_\wa$ of $\D_\wa$ is a sub-orbifold groupoid of the inertia orbifold groupoid $\I\Xa$ of $\Xa$. We see in \S \ref{subsec cohomology} that $H^*_{\text{CR}}(\Xa)$ is generated by $\mc K$ and forms that support over $\I\D_\wa$. Consider an absolute invariant of $\uX_\wa$
\begin{align}
\label{eq abs-inv-Xa-in-introduction}
\bl \prod_i\tau_{k_i}\gamma_i\cdot\prod_j \tau_{k'_j}\theta_j\br^{\Xa}_{g,\beta}.
\end{align}
with $\gamma_i$ belonging to $\mc K$, and $\theta_j$ supporting over $\I\D_\wa$. We use degeneration formula to calculate this invariant. We degenerate $\Xa$ as \eqref{E degenerate-Xa-in-sec1}. For $\gamma_i\in \mc K$ we could take the extension over $(\Xa|\D_\wa)$ to be $\gamma_i$ itself and an appropriate extension $\gamma_i^+$ over $(\overline{\mc O_{\PN_\wa}(-1)}|\D_{\wa,\0})$. For $\theta_j$ we could take the extension over $(\Xa|\D_\wa)$ to be $0$, and the extension over $(\overline{\mc O_{\PN_\wa}(-1)}|\D_{\wa,\0})$ to be $\theta_j$. Then by the degeneration formula, the absolute invariant \eqref{eq abs-inv-Xa-in-introduction} is determined by admissible relative invariants of $(\Xa|\D_\wa)$ and relative invariants of $(\overline{\mc O_{\PN_\wa}(-1)}|\D_{\wa,\0})$. By the linearity of Gromov--Witten invariants, all absolute descendent orbifold Gromov--Witten invariants of $\Xa$ are determined by absolute descendent orbifold Gromov--Witten invariants of $\Xa$ of the form \eqref{eq abs-inv-Xa-in-introduction}. Since $\D_{\wa,\0}\cong \D_\wa$, as a direct consequence of Theorem \ref{thm rel-GW-of-P(E)} and Theorem \ref{thm rel-GW-of-weighted-blp} we have proved

\begin{theorem}\label{thm abs-GW-of-weighted-blp}
The absolute descendent orbifold Gromov--Witten theory of $\Xa$ can be uniquely and effectively reconstructed from the absolute descendent orbifold Gromov--Witten theories of $\X$, $\os$ and $\D_\wa$, the restriction map $H^*_{\text{\em CR}}(\X)\rto H^*_{\text{\em CR}}(\os)$ and the first Chern class of $\mc O_{\PN_\wa}(-1)$.
\end{theorem}

Our results have several applications.

\subsubsection{Gromov--Witten theory of blowups along complete intersections}
Here we consider the orbifold version of a conjecture of Maulik--Pandharipande \cite{Maulik-Pandharipande2006} on Gromov--Witten theory of blowups along complete intersections. Let $\X$ be a compact symplectic orbifold groupoid, and $\W_1,\ldots,\W_m$ be $m$ symplectic divisors of $\X$ which intersect transversely such that
\[
\os:=\bigcap_{i=1}^m \W_i
\]
is a symplectic sub-orbifold groupoid of $\X$. Let $\uX$ be the blowup of $\X$ along $\os$ with trivial weight $\wa=(1,\ldots,1)$ and $\D$ be the exceptional divisor. Denote the normal bundle of $\W_i$ in $\X$ by $\N_i$. Then the normal bundle of $\os$ in $\X$ is a direct sum of $m$ line bundles
\[
\N_{\os|\X}=\bigoplus_{i=1}^m \N_i\big|_{\os}.
\]
Therefore
\[
\D=\P(\N_{\os|\X})=\P\Big(\bigoplus_{i=1}^m \N_i\big|_{\os}\Big).
\]
Each line bundle $\N_i|_\os$ gives us a section of $\D\rto\os$
\[
\os_i=\P(0\oplus\ldots\oplus \N_i\big|_\os\oplus\ldots\oplus 0), \qq 1\leq i\leq m.
\]
Moreover, $\os_i\cong\os$ for $1\leq i\leq m$.

\begin{remark}
When $\X=V$ is a smooth nonsingular projective variety, $\os=Z$ is a smooth nonsingular complete intersection of $m$ smooth nonsingular divisors $W_1,\ldots,W_m\subset V$ and $\tilde V$ is the blowup of $V$ along $Z$, Maulik and Pandharipande conjectured that (cf. \cite[Conjecture 2]{Maulik-Pandharipande2006})
\begin{quote}
The Gromov--Witten theory of $\tilde V$ is actually determined by the Gromov--Witten theories of $V$ and $Z$ and the restriction map $H^*(V,\ration)\rto H^*(Z,\ration)$.
\end{quote}
The following Theorem \ref{thm GW-of-blp-along-complete-intersection} is an orbifold version of this conjecture. This conjecture was also proved by Fan \cite{Fan2017} and Du \cite{Du2019}.
\end{remark}

\begin{theorem}\label{thm GW-of-blp-along-complete-intersection}
The orbifold Gromov--Witten theory of $\uX$ is uniquely and effectively determined by the Gromov--Witten theories of $\X$ and $\os$, and the restriction map $H^*_{\text{\em CR}}(\X,\ration)\rto H^*_{\text{\em CR}}(\os,\ration)$.
\end{theorem}

\begin{proof}
By Theorem \ref{thm abs-GW-of-weighted-blp}, the orbifold Gromov--Witten theory of $\uX$ is determined by the orbifold Gromov--Witten theories of $\X$, $\D$ and $\os$, the first Chern class of the normal bundle $\mc O(-1)\rto \D$ of $\D$ in $\uX$ and the restriction map $H^*_{\text{CR}}(\X)\rto H^*_{\text{CR}}(\os)$. We next show that
\begin{enumerate}
\item[(a)] the orbifold Gromov--Witten theory of $\D$ is determined by the orbifold Gromov--Witten theory of $\os$ and,
\item[(b)] $c_1(\mc O(-1)$ is determined by the restriction map $H^*_{\text{CR}}(\X)\rto H^*_{\text{CR}}(\os)$.
\end{enumerate}
There is a $T=(\cplane^\ast)^m$ action on $\N_{\os|\X}$ which descends to a $T$-action on $\D$. The fixed loci of this $T$-action on $\D$ consist of $\os_i\cong\os,\, 1\leq i\leq m$. The fixed lines connecting these $\os_i$ are special lines in the fiber of $\D=\P(\N_{\os|\X})\rto\os$.

The $T$-action on $\D$ induces a $T$-action on the moduli space of orbifold stable maps to $\D$. The fixed loci are determined by those graphes, of which vertices represent moduli spaces of orbifold stable maps to $\os_i$, and edges represent maps from orbifold Riemann spheres to $\D$ that are totally ramified over certain $\os_i,\os_j\subset\D$ with $i\neq j$. By the virtual localization (cf. \cite{Kontsevich1995,Graber-Pandharipande1999,Chen-Li2006,Liu2013}), the orbifold Gromov--Witten invariants of $\D$ are determined by Hodge integrals in the orbifold Gromov--Witten theories of $\os_i\cong \os$ (see \S \ref{subsec localization} for similar localization computations of contributions of simple fixed loci of relative moduli spaces). Hodge integrals in the orbifold Gromov--Witten theories of $\os_i\cong \os$ can be removed by the orbifold quantum Riemann--Roch of Tseng \cite{Tseng2010}. This proves (a).

On the other hand, the normal bundle of $\D$ in $\uX$ is the tautological line bundle $\mc O(-1)$ over $\D=\P(\N_{\os|\X})=\P(\bigoplus_{i=1}^m \N_i|_{\os})$, whose first Chern class $x$ is determined by the classical relation (cf. Bott and Tu \cite[(20.6)]{Bott-Tu82})
\[
(-x)^m+c_1(\N_{\os|\X})(-x)^{m-1}+\ldots +c_{m-1}(\N_{\os|\X})(-x)+c_m(\N_{\os|\X})=0.
\]
The total Chern class $c(\N_{\os|\X})=1+\sum_{i=1}^m c_i(\N_{\os|\X})$ is determined inductively by the restriction map $\iota^*\co H^*_{\text{CR}}(\X)\rto H^*_{\text{CR}}(\os)$ via
\begin{align}\label{eq determine-c(N)}
c(\N_{\os|\X})\cup c(T\os)=\iota^*(c(T\X)).
\end{align}
See also \S \ref{sec rel-GW-of-weit-blp}. This proves (b). The proof is complete.
\end{proof}

\subsubsection{Orbifold Gromov--Witten theory of root constructions}\label{SSS OGW-root-stack}

We next study a conjecture of Tseng--You on orbifold Gromov--Witten theory of root constructions. Let $\X$ be a smooth Deligne--Mumford stack, and $\D$ be a divisor, Tseng--You \cite{Tseng-You2016a} considered the $r$-th root construction $\X_r$ of $\X$ along $\D$. They conjectured \cite[Conjecture 1.1]{Tseng-You2016a} that
\begin{conj}\label{Conj Tseng-You}
The Gromov--Witten theory of $\X_r$ is determined by the Gromov--Witten theories of $\X, \D$, and the restriction map $H^*_{\mathrm{CR}}(\X)\rto H^*_{\mathrm{CR}}(\D)$.
\end{conj}
See \cite{Tseng2019} for more discussions on this conjecture and related conjectures.

The root stack construction corresponds to the weighted blowup along symplectic divisors in the symplectic category. Let $\D\subseteq\X$ be a symplectic divisor with normal line bundle $\L:=\N_{\D|\X}$ and $\uX_{(r)}$ be the weight-$\wa=(r)$ blowup of $\X$ along $\D$. We have a natural projection $\kappa\co\uX_{(r)}\rto \X$ with $\D_{(r)}:=\kappa\inv(\D)=\msf P\L_{(r)}$ being the exceptional divisor. $\D_{(r)}$ is the $r$-th root gerbe $\sqrt[r]{\L/\D}$ of the line bundle $\L$. Its normal line bundle $\N_{\D_{(r)}|\uX_{(r)}}$ in $\uX_{(r)}$ is the $r$-th root $\sqrt[r]{\L}$ of $\L$. As a $\integer_r$-gerbe over $\D$ it is a trivial band gerbe. We denote the restriction of $\kappa$ on $\D_{(r)}$ also by $\kappa\co\D_{(r)}\rto\D$.

\begin{theorem}\label{thm GW-of-weight--blp-along-divisor}
Conjecture \ref{Conj Tseng-You} is true, i.e. the Gromov--Witten theory of $\uX_{(r)}$ is determined by the Gromov--Witten theories of $\X,\D$ and the restriction map $H^*_{\text{\em CR}}(\X)\rto H^*_{\text{\em CR}}(\D)$.
\end{theorem}

\begin{proof}
By Theorem \ref{thm abs-GW-of-weighted-blp}, the Gromov--Witten theory of $\uX_{(r)}$ is determined by the Gromov--Witten theories of $\X,\D_{(r)}$ and $\D$, the restriction map $H^*_{\text{CR}}(\X)\rto H^*_{\text{CR}}(\D)$ and $c_1(\N_{\D_{(r)}|\uX_{(r)}})=c_1(\sqrt[r]\L)$. We have $c_1(\sqrt[r]\L)=\frac{1}{r}\kappa^*c_1(\L)$ as $(\sqrt[r]\L)^{\otimes r}=\kappa^*\L$. Therefore since $c_1(\L)$ is determined by the restriction map $H^*_{\text{CR}}(\X)\rto H^*_{\text{CR}}(\D)$ by \eqref{eq determine-c(N)}, $c_1(\sqrt[r]\L)$ is also determined by the restriction map $H^*_{\text{CR}}(\X)\rto H^*_{\text{CR}}(\D)$.

We next consider the orbifold Gromov--Witten theory of $\D_{(r)}$. The projection $\kappa\co\D_{(r)}\rto\D$ induces a map on inertia spaces (cf. \S\ref{SSS inertia-of-Ea})
\[
\msf I\kappa:= \bigsqcup_{(h)=(g,e^{2\pi \sqrt{-1} R})\in\scr T^{\D_{(r)}}}\kappa_{(h)}\co \bigsqcup_{(h)=(g,e^{2\pi \sqrt{-1} R})\in\scr T^{\D_{(r)}}}H^*(\D_{(r)}(h))\rto \bigsqcup_{(h)=(g,e^{2\pi \sqrt{-1} R})\in\scr T^{\D_{(r)}}} H^*(\D(g))
\]
where for each $\kappa_{(h)}$, the corresponding map $|\kappa_{(h)}|\co |\D_{(r)}(h)|\rto|\D(g)|$\footnote{Here for an orbifold groupoid $\X$, $|\X|$ is its coarse space.} on coarse spaces is a homeomorphism. The $\kappa$ induces an isomorphism of Chen--Ruan cohomology groups between $\D_{(r)}$ and $r$ copies of $\D$.

Consider an orbifold Gromov--Witten invariant of $\D_{(r)}$
\begin{align}\label{E inv-of-Dr}
\bl \tau_{a_1} \kappa_{(h_1)}^*\alpha_1,\ldots, \tau_{a_n}\kappa_{(h_n)}^*\alpha_n\br_{g,\vec h,\beta}^{\D_{(r)}}:=\int_{[\M_{g,\vec h,\beta}(\D_{(r)})]^\vir}\msf{ev}_{\D_{(r)}}^* \left(\prod_{j=1}^n\kappa_{(h_j)}^* \alpha_j\right)\cup \prod_{j=1}^n\bar\psi_j
\end{align}
where
\begin{itemize}
\item $g\geq 0$ is the genus, and $\beta\in H_2(|\D_{(r)}|;\integer)$ is a degree 2 homology class,
\item $\alpha_i\in H^*(\D(g_i))$ with $\D(g_i)$ obtained from $\kappa_{(h_i)}\co\D_{(r)}(h_i)\rto\D(g_i)$,
\item $\vec h=((h_1),\ldots,(h_n))$ indicates the twisted sectors of $\D_{(r)}$ that the images of the evaluation map $\msf{ev}_{\D_{(r)}}\co \M_{g,\vec h,\beta}(\D_{(r)})\rto (\msf I\D_{(r)})^n$ are located,
\item $\bar\psi_j$ is the psi-class of the line bundle over $\M_{g,\vec h,\beta}(\D_{(r)})$ whose fiber over a point $[\f\co(\C,\x_1,\ldots,\x_n)\rto\D_{(r)}]$ is the cotangent space of the coarse space of $\C$ at the $j$-th marked point. See also \S \ref{subsec notation-for-rel-inv-DYD}.
\end{itemize}
From $\kappa\co\D_{(r)}\rto\D$ we get a moduli space $\M_{g,\vec g,\beta}(\D)$ of orbifold stable maps to $\D$ with $\vec g=((g_1),\ldots,(g_n))$, and a natural map $\pi\co \M_{g,\vec h,\beta}(\D_{(r)})\rto\M_{g,\vec g,\beta}(\D)$ which sits in the following commutative diagram
\[
\xymatrix{\M_{g,\vec h,\beta}(\D_{(r)})\ar[r]^-{\msf{ev}_{\D_{(r)}}} \ar[d]_-\pi & (\msf I\D_{(r)})^n\ar[d]^-{(\msf I\kappa)^n}\\ \M_{g,\msf \pi(\vec h),\beta}(\D)\ar[r]^-{\msf{ev}_{\D}}& (\msf I\D)^n,
}
\]
where the horizontal maps are evaluation maps. Recall that $\kappa\co\D_{(r)}\rto\D$ is a trivial band $\integer_r$-gerbe over $\D$. By the computation of Tang--Tseng \cite[Section 5]{Tang-Tseng2016} we have
\[
\pi_*[\M_{g,\vec h,\beta}(\D_{(r)})]^\vir=r^{2g-1}\cdot [\M_{g,\vec g,\beta}(\D)]^\vir.
\]
Therefore from $\msf I\kappa\circ ev_{\D_{(r)}}=ev_\D\circ \pi$ and $\pi^*\bar\psi_j=\bar\psi_j$ for $1\leq j\leq n$, we have
\begin{align*}
\langle \tau_{a_1} \kappa_{(h_1)}^*\alpha_1,\ldots, \tau_{a_n}\kappa_{(h_n)}^*\alpha_n\rangle_{g,\vec h,\beta}^{\D_{(r)}}
&=\left[\msf{ev}_{\D_{(r)}}^*\left(\prod_{j=1}^n \kappa_{(h)}^*\alpha_j\right)\cup \prod_{j=1}^n \bar\psi_j^{a_j}\right]\cap[\M_{g,\vec h,\beta}(\D_{(r)})]^\vir\\
&=\pi^*\left[ \msf{ev}_\D^*\left(\prod_{j=1}^n (\alpha_j)\right)\cup\prod_{j=1}^n\bar\psi_j^{a_j}\right]\cap [\M_{g,\vec h,\beta}(\D_{(r)})]^\vir\\
&=\left[\msf{ev}_\D^*\left(\prod_{j=1}^n (\alpha_j)\right)\cup\prod_{j=1}^n\bar\psi_j^{a_j}\right]\cap \pi_*([\M_{g,\vec h,\beta}(\D_{(r)})]^\vir)\\
&=r^{2g-1}\cdot \left[\msf{ev}_\D^*\left(\prod_{j=1}^n (\alpha_j)\right)\cup\prod_{j=1}^n\bar\psi_j^{a_j}\right]\cap [\M_{g,\vec g,\beta}(\D)]^\vir\\
&=r^{2g-1}\cdot \langle \tau_{a_1}\alpha_1,\ldots,\tau_{a_n}\alpha_n\rangle_{g,\vec g,\beta}^{\D}.
\end{align*}
Therefore, the orbifold Gromov--Witten invariants of $\D_{(r)}$ are determined by the orbifold Gromov--Witten invariants of $\D$. This finishes the proof of this theorem, hence Conjecture \ref{Conj Tseng-You}.
\end{proof}

\subsection{Topological view for orbifold Gromov--Witten theory}

Theorem \ref{thm rel-GW-of-P(E)} is proved by using virtual localization of relative invariants and an analogue of the rubber calculus of Maulik--Pandharipande \cite{Maulik-Pandharipande2006}. As a byproduct, we could generalize the Leray--Hirsch result for line bundles and Mayer--Vietoris result for symplectic cutting in Gromov--Witten theory to orbifold Gromov--Witten theory.

\subsubsection{Orbifold Leray--Hirsch}\label{subsec introcutrion-orb-LH}

Let $\D$ be a compact symplectic orbifold groupoid. Let $\L$ be a symplectic orbifold line bundle over $\D$. Consider the projectification $\Y=\P(\L\oplus \mc O_\D)$. $\Y$ has the zero section $\D_0=\P(0\oplus\mc O_\D)$ and the infinity section $\D_\0=\P(\L\oplus 0).$ We have $\D_0\cong\D_\0\cong\D$. The normal bundles of $\D_0$ and $\D_\0$ in $\Y$ are $\L$ and $\L^\ast$ respectively.

Consider four orbifold Gromov--Witten theories: the absolute descendent orbifold Gromov--Witten theory of $\Y$ and the relative descendent orbifold Gromov--Witten theories of the three pairs
\[
(\Y|\D_0),\ (\Y|\D_\0), \text{   and   }\ (\D_0|\Y|\D_\0).
\]

\begin{theorem} \label{thm orb-LH}
All four theories can be uniquely and effectively reconstructed from the absolute descendent orbifold Gromov--Witten theory of $\D$ and the first Chern class of the line bundle $\L$.
\end{theorem}

We call this the {\em orbifold Leray--Hirsch result}. The relative theories of $(\Y|\D_0)$ and $(\Y|\D_\0)$ are special case of Theorem \ref{thm rel-GW-of-P(E)}. On the other hand
\begin{enumerate}
\item[(1)] when $\D$ is a compact symplectic manifold and $\L$ is a symplectic line bundle over $\D$, this theorem is just the Leray--Hirsch result of Maulik--Pandharipande \cite[Theorem 1]{Maulik-Pandharipande2006}.
\item[(2)] when $\D=BG$ is the classifying groupoid for a finite group $G$, this theorem was proved by Tseng--You \cite{Tseng-You2016b} by an explicit computation of the double ramification cycles on the moduli spaces of admissible covers, following the approach of Janda, Pandharipande, Pixton and Zvonkine \cite{Janda-Pandharipande-Pixton-Zvonkine2017}.
\end{enumerate}

This orbifold Leray--Hirsch result for orbifold Gromov--Witten theory was also conjectured by Tseng and You \cite[Conjecture 2.2]{Tseng-You2016a}. See also Tseng \cite{Tseng2019}.

\subsubsection{Orbifold Mayer--Vietoris}\label{subsec introduction-orb-MV}

Consider a general symplectic cutting on a compact symplectic orbifold groupoid $\X$. We get a family $\epsilon\co\mc D\to\ball$ of symplectic orbifold groupoids (see Chen--Li--Sun--Zhao \cite[\S 4.1]{Chen-Li-Sun-Zhao2011}), where
\begin{enumerate}
\item[(i)] $\ball$ is the unit ball in $\cplane$,
\item[(ii)] for all $t\neq 0$, $\epsilon\inv(t)\cong \X$,
\item[(iii)] $\X_0:=\epsilon\inv(0)= \X^+\wedge_\D\X^-$, with $\X^\pm$ intersect with each other at the common divisor $\D$ normal crossingly.
\end{enumerate}
This is a degeneration of $\X$. There is an induced homomorphism $H_2(|\X|;\integer)\rto H_2(|\X_0|;\integer)$ on homologies, and classes in the kernel of the this induced homomorphism are called {\em vanishing cycles}. There is also an induced homomorphism on Chen--Ruan cohomologies
\[
H^*_{\text{CR}}(\X^+\wedge_\D\X^-)\rto H^*_{\text{CR}}(\X)
\]
with image called the {\em non-vanishing} cohomology of $\X$.

\begin{theorem}\label{thm orb-mv}
If there is no vanishing cycles, then the absolute descendent orbifold Gromov--Witten theory of the non-vanishing cohomology of $\X$ can be uniquely and effectively reconstructed from the absolute descendent orbifold Gromov--Witten theories of $\X^\pm$ and $\D$, and the restriction maps
\[
H^*_{\text{\em CR}}(\X^+)\rto H^*_{\text{\em CR}}(\D),\qq H^*_{\text{\em CR}}(\X^-)\rto H^*_{\text{\em CR}}(\D).
\]
\end{theorem}

\begin{proof}
This is a direct consequence of the degeneration formula \cite{Chen-Li-Sun-Zhao2011,Abramovich-Fantechi2016} and Corollary \ref{coro rel-inv-X-S}.
\end{proof}

It is well-known that when the symplectic cutting is proceeded by (weighted) blowup along a symplectic sub-orbifold, there is no vanishing cycle. Therefore for these cases, this theorem applies.

\subsection{Organization of the paper}
This paper is organized as follows. In \S \ref{sec weitblp-and-chom}, we review the weighted projectification and weighted projectivization of orbifold bundles, and weighted blowup of symplectic orbifold groupoids. Then we describe the Chen--Ruan cohomologies of the projectifications of orbifold bundles, and the Chen--Ruan cohomologies of the weighted blowups and the exceptional divisors. In \S \ref{sec fiber-class-inv} we determine all fiber class relative orbifold Gromov--Witten invariants of the projectification of an orbifold line bundle. The main content of \S \ref{sec rel-GW-of-weight--proj} is to prove Theorem \ref{thm rel-GW-of-P(E)}. In the end of \S \ref{sec rel-GW-of-weight--proj} we give a proof of Theorem \ref{thm orb-LH} for completeness. Then in \S \ref{sec rel-GW-of-weit-blp} we prove Theorem \ref{thm rel-GW-of-weighted-blp}.

\subsection*{Acknowledgements}

This work was supported by the National Natural Science Foundation of China (No. 11890663, No. 11821001, No. 11826102, No. 11501393, No. 12071322), by the Sichuan Science and Technology Program (No. 2019YJ0509), and by a joint research project of Laurent Mathematics Research Center of Sichuan Normal University and V.\thinspace C. \& V.\thinspace R. Key Lab of Sichuan Province.

\section{Weighted blowups and Chen--Ruan cohomology}\label{sec weitblp-and-chom}

In this paper, we study orbifolds (\cite{Satake1956}) via proper \'etale Lie groupoids, which are called orbifold groupoids. There are some nice references on orbifold groupoids. See for example Adem--Leida--Ruan \cite{Adem-Leida-Ruan2007} and Moerdijk--Pronk \cite{Moerdijk-Pronk1997}. One can see also \cite[\S 2]{Chen-Du-Hu2019} for a brief introduction of orbifold groupoids and Chen--Ruan cohomology etc.

We use $\wa=(a_1,\ldots,a_n)$ to denote the blowup weight, where $a_i\in\integer_{\geq 1}$, for $1\leq i\leq n$.

\subsection{Weighted projectification and projectivization}
\label{subsec weight--projfi-viza}

Let $S^1$ act on $\cplane^n$ by
\[
t\cdot(z_1,\ldots,z_n)=(t^{a_1}z_1,\ldots, t^{a_n}z_n).
\]
Denote this action by $S^1(\wa)$. The {\em weight-$\wa$ projective space} is
\[
\msf P_\wa:=\P_\wa(\cplane^n)=S^{2n-1}/{S^1(\wa)}.
\]
The {\em weight-$\wa$ blowup of $\cplane^n$ along the origin} is
\[
[\underline{\cplane^n}]_\wa=S^{2n-1}\times_{S^1(\wa,-1)}\cplane,
\]
it is the total space of the tautological line bundle $\mc O(-1)$ over $\msf P_\wa$.

Now let $\os=(S^1\rrto S^0)$ be an orbifold groupoid. Consider a rank $2n$ symplectic orbifold vector bundle
\[
\pi=(\pi^0,\pi^1) \co\E=(E^1\rrto E^0)\rto\os=(S^1\rrto S^0).
\]
So both $\pi^i\co E^i\rto S^i$ are symplectic vector bundles for $i=0,1$ and $s^*E^0\cong E^1\cong t^*E^0$ for the source and target maps $s,t\co S^1\rto S^0$ of $\os$. This is equivalent to an action of $\os=(S^1\rrto S^0)$ on the symplectic vector bundle $\pi^0\co E^0\rto S^0$, i.e. there is an action map
\[
\mu\co S^1\times_{s,S^0,\pi^0}E^0\rto E^0
\]
that is compatible with $\pi^0\co E^0\rto S^0$, which is called the {\em anchor map} for this action, and satisfies the rule of group actions, for instance
\[
\mu(gh,v)=\mu(h,\mu(g,v))
\]
for two arrows $g,h\in S^1$ and $v\in E^0$. Then one has
\[
E^1\cong S^1\times_{s,S^0,\pi}E^0.
\]
We also write
\[
\E=\os\ltimes E^0.
\]

Take a compatible complex structure. Let $\pi\co \msf P\rto\os$ be the corresponding principle bundle with structure group $K<U(n)$. Then
\[
\E=\msf P\times_K\cplane^n.
\]
The {\em weight-$\wa$ projectivization of $\E$} is
\[
\PE_\wa:=\msf P\times_K \msf P_\wa,
\]
The {\em weight-$\wa$ projectification of $\E$} is
\[
\oE_\wa:=\msf P\times_K \msf P_{\wa,1},
\]
where $\P_{\wa,1}$ is the weight-$(\wa,1)$ projective space.
The {\em weight-$\wa$ blowup of $\E$} is
\[
\uE_\wa:=\msf P \times_K [\underline{\cplane^n}]_\wa.
\]
When the weight $\wa=(1,\ldots,1)$ is trivial, we omit the subscript $\wa$.

$\PE_\wa$ is the exceptional divisor of the blowup $\Ea$. $\Ea$ is the tautological line bundle over $\PE_\wa$, whose restriction on fiber of $\pi\co \PE_\wa\rto\os$ is the tautological line bundle $\mc O(-1)$ over $\msf P_\wa$. So as in the introduction we also write $\Ea$ as $\mc O_{\PE_\wa}(-1)$. $\PE_\wa$ is also the infinite divisor of the weight-$\wa$ projectification $\oE_\wa$, its normal line bundle in $\oE_\wa$ is $\mc O_{\PE_\wa}(1)$, the dual line bundle of $\mc O_{\PE_\wa}(-1)$.

\subsection{Weighted blowups}\label{subsec weighted-blps}

Let $\X$ be a compact symplectic orbifold groupoid with $\os$ being a codimension $\text{codim}_\rone\os=2n$ compact symplectic sub-orbifold groupoids. Denote the normal bundle of $\os$ in $\X$ by $\N$. The symplectic neighborhood theorem holds for $(\X,\os)$ (cf. \cite{Du-Chen-Wang2018,Chen-Du-Hu2019}). Locally, there is an open neighborhood $\U$ of $\os$ in $\X$, which is symplectomorphic to the $\epsilon$-disk bundle
\[
\phi\co\U\rto \mb D_\epsilon(\N)
\]
of the normal bundle $\N$ of $\os$ in $\X$. We call $(\X,\os)$ a symplectic pair.

The {\em weight-$\wa$ blowup} $\Xa$ of $\X$ along $\os$ is obtained by gluing $\X\backslash\U$ with the $\epsilon$-disk bundle of $\Na$, the weight-$\wa$ blowup of $\N$ along the zero section $\os$. The exceptional divisor in $\Xa$ is $\D_\wa:=\PN_\wa$, the weight-$\wa$ projectivization of $\N$. So as \eqref{E degenerate-X-in-sec1}, the weight-$\wa$ blowup of $\X$ along $\os$ gives rise to a degeneration of $\X$
\begin{align*}
\X\xrightarrow{\text{degenerates}}(\Xa| \D_\wa)\wedge_{\D_\wa}(\oN_\wa| \D_{\wa}).
\end{align*}
There is a natural projection
\[
\kappa\co\Xa\rto\X
\]
which restricts to
\[
\kappa\co\D_\wa=\PN_\wa\rto\os.
\]

\subsection{Chen--Ruan Cohomology}\label{subsec cohomology}

We next describe the Chen--Ruan cohomologies of $\oE_\wa$, $\Ea$, $\PE_\wa$, $\Xa$ and $\D_\wa$. For an orbifold groupoid $\W$ we use $\I\W$ to denote its inertia space (see for example \cite[Definition 2.10]{Chen-Du-Hu2019}) and $\T^\W$ to denote the index set of its twisted sectors. For a $\delta \in\T^\W$, we denote the corresponding twisted sector by $\W(\delta)$. When we use local chart, we also treat the conjugate class of isotropy groups as the index of twisted sectors.

\subsubsection{Inertia space of $\Ea$ and $\PE_\wa$}
\label{SSS inertia-of-Ea}

We first focus on the local case. Since $\Ea=\msf P\times_K [\underline{\cplane^n}]_\wa$, we get a projection
\[
\pi\co\uE_\wa\rto\os
\]
which is also an orbifold groupoid morphism and restricts to
\[
\pi\co\PE_\wa\rto\os.
\]
Hence there are induced morphisms between their inertial spaces
\[
\I\pi\co \I\uE_\wa\rto\I\os,\qq \mbox{and}\qq \I\pi\co \I\PE_\wa\rto\I\os.
\]
In particular, there are induced maps on the index sets of twisted sectors
\[
\pi_t \co \T^{\uE_\wa}\rto\T^\os,\qq \mbox{and}\qq \pi_t\co \T^{\PE_\wa}\rto\T^\os.
\]
When restricting on twisted sectors, we write the restriction of $\I\pi$ as
\[
\pi_{(h)}\co \uE_\wa(h)\rto \os(\pi_t(h)),\qq \mbox{and}\qq \pi_{(h)}\co \PE_\wa(h)\rto \os(\pi_t(h)).
\]

\begin{remark}
Note that $\pi\co\uE_\wa\rto\PE_\wa$ is an orbifold line bundle, hence $\T^{\uE_\wa}=\T^{\PE_\wa}$. Moreover, either $\pi_{(h)}\co\uE_\wa(h)\rto\PE_\wa(h)$ is an orbifold line bundle or $\uE_\wa(h)=\PE_\wa(h)$ depending on the action of $(h)$ on the fiber of $\uE_\wa$ is trivial or non-trivial.
\end{remark}

For a point $x\in S^0$, locally near $x_0$, $\os$ is modeled by $G_x\ltimes U_x$ with $G_x$ being the local (or isotropy) group of $x$ in $\os$. Then locally, $\uE_\wa$ and $\PE_\wa$ are of the forms
\[
\frac{U_x\times (S^{2n-1}\times \cplane)}{G_x\times S^1(\wa,-1)},\qq\mbox{and}\qq
\frac{U_x\times S^{2n-1}}{G_x\times S^1(\wa)}
\]
respectively. Now consider a $(g)\in\T^\os$ with representative $g\in G_x$, i.e. $g\cdot x=x$. Suppose the $g$-action on the fiber $E^0|_x$ is given by
\[
g\cdot(z_1,\ldots,z_n)=(e^{2\pi \sqrt{-1} \frac{\fk b(g)^1_\E}{\fk o(g)}} z_1,\ldots,e^{2\pi \sqrt{-1} \frac{\fk b(g)^n_\E}{\fk o(g)}} z_n)
\]
with $\fk o(g)=\text{ord}(g)$ being the order of $g$, and $1\leq \fk b(g)^i_\E\leq \fk o(g)$ being the action weights. The order $\fk o(g)$ and action weights $\fk b(g)_\E^i$ do not depend on the choices of representative $g$ of $(g)$.

If an $(h)\in \T^{\Ea}=\T^{\PE_\wa}$ satisfies $\pi_t(h)=(g)$, then $(h)$ has a representative of the form
\begin{align*}
h=(g,e^{2\pi \sqrt{-1} R})\in G_x\times S^1
\end{align*}
for some $R\in \ration\cap [0,1)$. The $h$-action on $S^{2n-1}\times\cplane$ is given by
\begin{align}\label{E h-action}
h\cdot(z_1,\ldots,z_n,w)=(e^{2\pi \sqrt{-1} (\frac{\fk b(g)^1_\E}{\fk o(g)}+Ra_1)}z_1,\ldots, e^{2\pi \sqrt{-1} (\frac{\fk b(g)^n_\E}{\fk o(g)}+Ra_n)}z_n, e^{-2\pi \sqrt{-1} R}w),
\end{align}
which restricts to the $h$-action on $S^{2n-1}=S^{2n-1}\times\{0\}$.

Then the fiber of $\PE_\wa(h)\rto\os(g)$ is given by
\[
\msf P_{\wa_{I(h)}}/C_{G_x}(g)
\]
with
\[
I(h):=\left\{i\left|\frac{\fk b(g)^i_\E}{\fk o(g)}+Ra_i\in\integer,1\leq i\leq n\right.\right\},
\]
and $\wa_{I(h)}$ being a sub-weight of $\wa$ obtained from the inclusion $I(h)\subseteq\{1,\ldots,n\}$.

It is direct to see from \eqref{E h-action} that
\begin{enumerate}
\item[(a)] $\pi_{(h)}\co\uE_\wa(h)\rto\PE_\wa(h)$ is a line bundle when $R=0$,
\item[(b)] $\uE_\wa(h)=\PE_\wa(h)$ when $R\neq 0$.
\end{enumerate}

Set $\fk r(h):=\# I(h)$. Summarizing above discussions we get the following result.

\begin{lemma}\label{lem tiwsted-sec-of-bl-Ea}
$\pi_{(h)}\co\PE_\wa(h)\rto\os(g)$ is a weight-$\wa_{I(h)}$ projectivization of a rank $\fk r(h)$ sub-bundle of the pull-back bundle of $\E$ over $\os(g)$ via the natural evaluation morphism $e_{(g)}\co \os(g)\subseteq \I\os\rto \os$.

On the other hand, since $\pi\co\E\rto\os$ is a vector bundle, we also have $\T^\E=\T^\os$, and
\begin{enumerate}
\item[(a)] when $(h)=(g,1)$, i.e. $R=0$,
\[
\pi\co\uE_\wa(h)\rto\E(g)
\]
is the weight-$\wa_{I(h)}$ blowup of $\E(g)$ along $\os(g)$,

\item[(b)] when $(h)\neq (g,1)$, $\uE_\wa(h)=\PE_\wa(h)$.
\end{enumerate}
\end{lemma}

\subsubsection{Basis of Chen--Ruan cohomology of $\PE_\wa$}

For each $(g)\in\T^\os$ let
\begin{align*}
\sigma_{(g)}:=\{\delta_{(g)}^1,\ldots,\delta_{(g)}^{\fk k(g)}\}
\end{align*}
be a basis of $H^*(\os{(g)})$, then
\begin{align}\label{eq basis-of-S}
\sigma_\star:=\bigsqcup_{(g)\in\T^\os}\sigma_{(g)}=
\bigsqcup_{(g)\in\T^\os}\{\delta_{(g)}^1,\ldots,\delta_{(g)}^{\fk k(g)}\}
\end{align}
is a basis of
\[
H^*_{\text{CR}}(\os)=\bigoplus\limits_{(g)\in\T^\os} H^{*-2\iota(g)}(\os{(g)}).
\]
Here we assume that each $\delta_{(g)}^i$ is of homogenous degree and $\delta_{(1)}^1$ is the identity elements in $H^*(\os)$, denoted by $\one_\os$, or simply $\one$. Denote the dual basis with respect to the orbifold Poincar\'e duality by
\begin{align}\label{eq dual-basis-of-S}
\sigma^\star:=\bigsqcup_{(g)\in\T^\os}\check\sigma_{(g)}:=
\bigsqcup_{(g)\in\T^\os}\{\check \delta_{(g)}^1,\ldots,\check \delta_{(g)}^{\fk k(g)}\},
\end{align}
i.e. $\check \delta^i_{(g)}$ is the dual of $\delta^i_{(g)}$, hence $\check \delta^i_{(g)}\in H^*(\os(g\inv))$.

Set
\begin{align}\label{eq basis-of-PE}
\Sigma_{(h)}:=\{\delta_{(g)}^j\cup H^m_{(h)}|
\delta_{(g)}^j \in \sigma_{(g)}, 0\leq m\leq \fk r(h)-1\}
\end{align}
with $(g)=\pi_t(h)$. Here $H_{(h)}$ is the hyperplane class of the projective bundle $\pi_{(h)}\co\PE_\wa(h)\rto\os(g)$. Then $\Sigma_{(h)}$ is a basis of $H^*(\PE_\wa(h))$. Set $\Sigma_\star$ to be the union of $\Sigma_{(h)}$ over all $(h)\in\T^{\PE_\wa}$. Denote elements in $\Sigma_{(h)}$ by $\theta_{(h)}^l$, $1\leq l\leq \fk k(g)\fk r(h)$. Let $\Sigma^\star$ denote the dual basis of $\Sigma_\star$ with respect to the orbifold Poincar\'e duality.

\subsubsection{Basis of Chen--Ruan cohomology of $\oE_\wa$}

Similarly as $\PE_\wa$, $\oE_\wa$ is the weight-$(\wa,1)$ projectivization of $\E\oplus\mc O_\os$. Therefore there is a basis of $H^*_{\text{CR}}(\oE_\wa)$ as \eqref{eq basis-of-PE}, which consists of
\begin{align}\label{eq basis-of-oE}
\Xi_{(\bar h)}:=\{\delta_{(g)}^j\cup H^m_{(\bar h)}|
\delta_{(g)}^j \in \sigma_{(g)}, 0\leq m\leq \fk r(\bar h)-1\},
\end{align}
where $(\bar h)$ is the index of twisted sectors of $\oE$ and $\pi_t(\bar h)=(g)$. The definition of $\fk r(\bar h)$ is similar as $\fk r(h)$, just note that the $g$-action on $\mc O_\os$ is trivial. Denote elements in $\Xi_{(\bar h)}$ by $\gamma^{l}_{(\bar h)}$, $1\leq l\leq \fk k(g)\fk r(\bar h)$. Theorem \ref{thm rel-GW-of-P(E)} concerns relative invariants of $(\oE_\wa|\PE_\wa)$ whose  absolute insertions come from the basis \eqref{eq basis-of-oE} and relative insertions come from the basis \eqref{eq basis-of-PE}. We will prove Theorem \ref{thm rel-GW-of-P(E)} in \S \ref{sec rel-GW-of-weight--proj}.

\subsubsection{Chen--Ruan cohomology of $\uX_\wa$}

Now consider a symplectic pair $(\X, \os)$. Let $\N$ be the normal bundle of $\os$, and $(\Xa,\D_\wa=\PN_\wa)$ be its the weight-$\wa$ blowup along $\os$.

The projection $\kappa\co \Xa\rto\X$ also induces morphisms on twisted sectors
\[
\kappa_{(h)}\co \Xa(h)\rto\X(\kappa_t(h)).
\]
The main difference between $\I\Xa$ and $\I\X$ are those twisted sectors of $\Xa$ intersecting with the inertia space $\I\D_\wa$ of the exceptional divisor $\D_\wa=\PN_\wa$. By Lemma \ref{lem tiwsted-sec-of-bl-Ea} there are two kinds of such twisted sectors of $\Xa$.

\begin{lemma}
We have the following description of twisted sectors of $\Xa$. Let $\kappa_t(h)=(g)$.
\begin{enumerate}
\item[(a)] When $\uX_\wa(h)\cap \I\D_\wa=\varnothing$, we have $\X(g)\cap \I\os=\varnothing$ and $\kappa_{(h)}\co\uX_\wa(h)\rto \X(g)$ is identity.

\item[(b)] When $\uX_\wa(h)\cap \I\D_\wa\not= \varnothing$, we have $\X(g)\cap \I\os\not =\varnothing$, and there are two cases:
    \begin{enumerate}
    \item[(b.i)] $\Xa(h)\supseteq \Na(h)$ is the weighted blowup of $\N(g)$ along $\os(g)$, i.e. $(h)=(g,1)$, then $\kappa_{(h)}\co \uX_\wa(h)\rto\X(g)$, the sequence
\[
0\rto H^*(\X(g))\xrightarrow{\kappa_{(h)}^*} H^*(\uX_\wa(h))\rto \mc A(h)\rto 0
\]
is exact, where
\[
\mc A(h)=H^*(\os(g))\{H_{(h)},\ldots,H_{(h)}^{\fk r(h)-1}\},
\]
is a free module over $H^*(\os(g))$ with basis $\{H_{(h)},\ldots,H_{(h)}^{\fk r(h)-1}\}$, $H_{(h)}$ is the hyperplane class of $\PN_\wa(h)=\D_\wa(h)$. Moreover, $H^*(\uX_\wa(h))$ is generated by $\kappa_{(h)}^* H^*(\X(g))$ and
\[
\left(H^*(\os(g))\{1,H_{(h)},\ldots,H_{(h)}^{\fk r(h)-1}\}\right)\cup [\PN_\wa(h)],
\]
where $[\PN_\wa(h)]$ is the Poincar\'e dual of $\PN_\wa(h)$ in $\uX_\wa(h)$.

\item[(b.ii)] $\uX_\wa(h)=\uN_\wa(h)=\PN_\wa(h)$, i.e. $(h)\neq (g,1)$. Then
\[
H^*(\uX_\wa(h))=H^*(\os(g))\{1,H_{(h)},\ldots, H^{\fk r(h)-1}_{(h)}\}.
\]
\end{enumerate}
\end{enumerate}
\end{lemma}

\begin{defn}\label{def admissible-abs-insert-K}
We set
\begin{align}\label{eq admissible-abs-insert-K}
\mc K:=\bigoplus_{(h)\in\T^\Xa}\kappa^*_{(h)} \big(H^*(\X(\kappa_t(h)))\big).
\end{align}
\end{defn}

In this paper we will deal with relative invariants of $(\Xa|\D_\wa)$ with absolute insertions coming from $\mc K$, i.e. admissible relative invariants in Definition \ref{def admissible-rel-inv}.

\section{Fiber class invariants of projectification of orbifold line bundles}\label{sec fiber-class-inv}

In the study of relative orbifold Gromov--Witten theories of weighted blowups or weighted projectifications, we need to study fiber class orbifold Gromov--Witten invariants of projectifications of orbifold line bundles. In this section we study these fiber class invariants.

\subsection{Projectification of line bundles}
\label{subsec porjfi-line--bundle}

Let $\pi\co\L\rto\D$ be a symplectic line bundle. That is if we write $\L=\D\ltimes L^0$ with $\D=(D^1\rrto D^0)$, then $L^0\rto D^0$ is a symplectic line bundle and the linear $\D$-action on $L^0\rto D^0$ preserve the symplectic structure over $L^0$. The projectification of $\L$ is
\[
\Y:=\P(\L\oplus\mc O_\D)=\msf P\times_{U(1)}\msf P_{(1,1)}\rto\D,
\]
where $\sf P$ is the principle $U(1)$-bundle of $\L$, and $\msf P_{(1,1)}=\P^1$ is the one dimensional projective space. In terms of notations in \S \ref{subsec weight--projfi-viza}, $\Y=\overline{\L}_{(1)}$, i.e. $\Y$ is the projectification of $\L$ with trivial weight $\wa=(1)$. $\Y$ has the zero section $\D_0:=\P(0\oplus\mc O_\D)$ and the infinity section $\D_\0:=\P(\L\oplus 0)$. Both are isomorphic to $\D$. The normal bundle of $\D_0$ and $\D_\0$ in $\Y$ are $\L$ and $\L^*$ respectively.

As the description of inertia space and Chen--Ruan cohomology of $\oE_\wa$ in \S \ref{subsec cohomology} we have a similar but simpler description of the inertia space and Chen--Ruan cohomology of $\Y$.

The bundle $\pi\co\L\rto\D$ induces a bundle
\begin{align}\label{eq IL-ID}
\I\pi=\bigsqcup_{(h)\in\T^\D} \pi_{(h)}\co\I\L=\bigsqcup_{(h)\in\T^\D} \L(h)\rto\I\D=\bigsqcup_{(h)\in\T^\D}\D(h)
\end{align}
of inertia spaces, which may have different ranks over different components of $\I\D$. Over each twisted sector $\D(h)$, for every representative $h$ of $(h)$ in local groups, there is an action of $h$ on the fiber of $\L$. As in \S \ref{subsec cohomology}, denote by $\fk o(h)=\text{ord}(h)$ the order of $h$. Then the action is given by
\[
h\cdot z=e^{2\pi \sqrt{-1}\frac{\fk b(h)_\L}{\fk o(h)}}z
\]
with $1\leq \fk b(h)_\L \leq \fk o(h)$ being the action weight of $h$ on $\L$. The order $\fk o(h)$ and action weight $\fk b(h)_\L$ are independent of the choice of the representative $h$ of $(h)$ in local groups, and $\fk b(h)_\L$ is called the weight of $(h)$ on $\L$.

\begin{defn}\label{def L}
For each $(h)\in \T^\D$. Define
\begin{align*}
\fk l_{(h)}:=\left\{\begin{array}{ll}
0 &\textrm {if\,  $1\leq \wb(h)_\L<\fk o(h)$,}\\
1 &\textrm {if\,  $\wb(h)_\L=\fk o(h)$.}
\end{array}\right.
\end{align*}
\end{defn}

Since $\Y=\P(\L\oplus\mc O_\D)$ is the projectification of $\L$ with trivial weight $\wa=(1)$, the analysis in \S \ref{subsec cohomology} shows that
\begin{lemma}\label{lem CR-of-Y}
A component\footnote{Since $\wa=(1)$ we identify $(h,e^{-2\pi \sqrt{-1}\frac{\fk b(h)_\L}{\fk o(h)}})$ with $(h)$.} $\Y(h)=\pi_{(h)}\inv(\D(h))$ of the fiber bundle $\I\pi\co\I\Y\rto\I\D$ is determined as follows:
\begin{enumerate}
\item[(a)] if $\fk l_{(h)}=0$, then $\L(h)=\D(h)$ and $\Y(h)$ is a disjoint union of two zero bundles over $\D(h)$, i.e. $\Y(h)=\D_0(h)\sqcup\D_\0(h)$, and $\D_0(h)\cong\D_\0(h)\cong\D(h)$;
\item[(b)] if $\fk l_{(h)}=1$, then $\L(h)$ is a line bundle over $\D(h)$ and $\Y{(h)}=\P(\L(h)\oplus\mc O_{\D(h)})$.
\end{enumerate}
For the latter case, we have
\[
\L(h)=e_{(h)}^*\L
\]
with $e_{(h)}\co \D(h)\rto\D$ being the evaluation morphism. Moreover, $\Y{(h)}$ also contains the zero and infinity sections $\D_0(h)$ and $\D_\0(h)$, both are isomorphic to $\D(h)$.
\end{lemma}

Take a basis of $H^*_{\text{CR}}(\D)$
\begin{align}\label{eq basis-of-D}
\Sigma_\star:=\bigsqcup_{(h)\in\T^\D}\Sigma_{(h)}:=
\bigsqcup_{(h)\in\T^\D}\{\theta_{(h)}^1,\ldots,\theta_{(h)}^{\fk k(h)}\}.
\end{align}
(Since in this paper in most cases $\D$ is $\PE_\wa$, so for simplicity, we use the same notation as the basis \eqref{eq basis-of-PE} of $H^*_{\text{CR}}(\PE_\wa)$ to denote the basis of $H^*_{\text{CR}}(\D)$).  Then we get a basis of $H^*_{\text{CR}}(\Y)$ as follows. When $\fk l_{(h)}=0$, each $\theta_{(h)}^i$ contributes two elements:
\[
\theta_{(h)}^{0,i}\,\,\,\mbox{for}\,\,\, \D_0(h),\qq \mbox{and} \,\,\,
\theta_{(h)}^{\0,i}\,\,\, \mbox{for}\,\,\, \D_\0(h).
\]
When $\fk l_{(h)}=1$, denote the Poincar\'e dual of $\D_0(h)$ and $\D_\0(h)$ in $\Y(h)=\P(\L(h)\oplus\mc O_{\D(h)})$ by
\[
[\D_0(h)]\qq \text{and}\qq [\D_\0(h)]
\]
respectively. We have
\begin{align}\label{eq D0=D1+c1L}
[\D_0(h)]=[\D_\0(h)]+c_1(\L(h))=[\D_\0(h)]+e_{(h)}^*(c_1(\L)).
\end{align}
Then each $\theta_{(h)}^i$ contributes two elements for $\Y(h)$:
\[
\theta_{(h)}^i \qq\mbox{and}\qq \theta_{(h)}^i\cdot[\D_0(h)].
\]
Combining these together we get a basis of $H^*_{\text{CR}}(\Y)$:
\begin{align}\label{eq basis-of-Y}
\bigsqcup_{(h)\in\T^\D}\left\{\begin{array}{ll}
\theta_{(h)}^{0,1} ,\ldots,\theta^{0,\fk k(h)}_{(h)};
\theta^{\0,1}_{(h)} ,\ldots,\theta^{\0,\fk k(h)}_{(h)}
&\textrm {if\, $\fk l_{(h)}=0$;}\\
\theta_{(h)}^1 ,\ldots,\theta_{(h)}^{\fk k(h)};
\theta_{(h)}^1\cdot [\D_0(h)] ,\ldots,\theta_{(h)}^{\fk k(h)}\cdot[\D_0(h)]
&\textrm {if\, $\fk l_{(h)}=1$.}
\end{array}\right.
\end{align}

\subsection{Notations for relative orbifold Gromov--Witten invariants of $(\D_0|\Y|\D_\0)$}\label{subsec notation-for-rel-inv-DYD}

Now we fix the notation for relative orbifold Gromov--Witten invariants of $(\D_0|\Y|\D_\0)$. A relative invariant of $(\D_0|\Y|\D_\0)$ is of the form
\begin{align}\label{E notation-for-rel-inv-DYD}
&\bl \mu \bb  \prod_{i=1}^m\tau_{k_i} \gamma_{(\bar h_i)}^{l_i}\bb\nu\br_\Gamma^{(\D_0|\Y|\D_\0)}
:=\frac{1}{|\mbox{Aut}(\mu)|\cdot|\mbox{Aut}(\nu)|}\\
&\int_{[\M_\Gamma(\D_0|\Y|\D_\0)]^\vir}
\prod_{i=1}^m\overline\psi_i^{k_i}
{\sf ev}_i^*(\gamma_{(\bar h_i)}^{l_i})
\cup
\prod_{j=1}^{\ell(\mu)} {\sf rev}^{\D_0,*}_j(\theta_{(h_j)}^{s_j})
\cup
\prod_{k=1}^{\ell(\nu)} {\sf rev}^{\D_\0,*}_k(\theta_{(h_k')}^{s_k'})\nonumber
\end{align}
with topological data (or type) $\Gamma=(g,\beta,(\bar\h),\vec\mu,\vec\nu)$, where
\begin{itemize}
\item $g$ is the genus, and $\beta\in H_2(|\Y|;\integer)$ is the homology class,
\item $\bar\h=\big((\bar h_1),\ldots, (\bar h_m)\big)\in (\T^\Y)^m$, and $\gamma_{(\bar h_i)}^{l_i}$ belongs to the basis \eqref{eq basis-of-Y}; denote by
    \[
    \varpi=(\tau_{k_1} \gamma_{(\bar h_1)}^{l_1},\ldots,
    \tau_{k_m} \gamma_{(\bar h_m)}^{l_m})
    \]
    the absolute insertions; denote the number of insertions in $\varpi$ by $\|\varpi\|=m$;

\item $\mu=\left((\mu_1,\theta_{(h_1)}^{s_1}),\ldots, (\mu_{\ell(\mu)}, \theta_{(h_{\ell(\mu)})}^{s_{\ell(\mu)}})\right)$ is a relative weighted partition and is weighted by the chosen basis \eqref{eq basis-of-D} of the Chen--Ruan cohomology of $\D$ with
\[
\vec\mu= \left((\mu_1,(h_1)),\ldots, (\mu_{\ell(\mu)},(h_{\ell(\mu)}))\right),
\]
and
\[
\sum_j \mu_j=\int_\beta^{\mathrm{orb}} [\D_0]\geq 0
\]
is the sum of orbifold contact orders along $\D_0$ of the relative maps;

\item $\nu=\left((\nu_1,\theta_{(h'_1)}^{s_1'}),\ldots, (\nu_{\ell(\nu)}, \theta_{(h_{\ell(\nu)}')}^{s_{\ell(\nu)}'})\right)
$ is a relative weighted partition and is weighted by the chosen basis \eqref{eq basis-of-D} of the Chen--Ruan cohomology of $\D$ with
\[
\vec\nu= \left((\nu_1,(h_1')),\ldots, (\nu_{\ell(\nu)},(h_{\ell(\nu)}'))\right),
\]
and
\[
\sum_j \nu_j=\int_\beta^{\mathrm{orb}} [\D_\0]\geq 0
\]
is the sum of orbifold contact orders along $\D_\0$ of the relative maps;

\item ${\sf ev}_i$, ${\sf rev}_j^{\D_0}$ and ${\sf rev}_j^{\D_\0}$ are the evaluation maps from $\M_\Gamma(\D_0|\Y|\D_\0)$ to $\sf I\Y$, $\sf I\D_0$ and $\sf I\D_\0$ at absolute marked points and relative marked points respectively;

\item $\overline\psi_i$ is the first Chern class of the $i$-th cotangent line bundle $\overline{\mc L_i}$ over the moduli space $\M_\Gamma(\D_0|\Y|\D_\0)$ of relative stable orbifold maps, whose fiber over a relative stable map is the cotangent line of the coarse moduli space of the domain curve at the $i$-th absolute marked point. There is another $i$-th cotangent line bundle $\mc L_i$ over $\M_\Gamma(\D_0|\Y|\D_\0)$, whose fiber over a relative stable map is the cotangent line of the $i$-th absolute marked point on the domain curve. The first Chern class of $\mc L_i$ is denoted by $\psi_i$ usually; we have
\[
\overline{\mc L_i}=\mc L_i^{\otimes r_i},
\qq \mbox{hence}\qq \overline\psi_i=r_i\psi_i
\]
where $r_i$ is the order of the local group of the $i$-th absolute marked point.
\end{itemize}

For the relative weighted partition $\mu=\left((\mu_1,\theta_{(h_1)}^{s_1}),\ldots, (\mu_{\ell(\mu)},\theta_{(h_{\ell(\mu)})}^{s_{\ell(\mu)}})\right)$ above we set
\begin{itemize}
\item[(i)] $\check\mu$ to be the dual relative weighted partition, which is
\[
\check\mu:=\left((\mu_1,\check\theta_{(h_1)}^{s_1}),\ldots,
(\mu_{\ell(\mu)}, \check\theta_{(h_{\ell(\mu)})}^{s_{\ell(\mu)}})\right),
\]
therefore
\[
\vec{\check\mu}=
\left((\mu_1,(h_1\inv)),\ldots, (\mu_{\ell(\mu)},(h_{\ell(\mu)}\inv))\right),
\]
we will also denote $\vec{\check\mu}$ by $\check{\vec\mu}$;
\item[(ii)] $\fk z(\mu):=|\aut(\mu)|\cdot\prod_i\mu_i$;
\item[(iii)] $\deg_{\text{CR}}\mu:=\sum\limits_{i=1}^{\ell(\mu)}(\deg \theta_{(h_i)}+2\iota^{\D}(h_i))$ with $\iota^{\D}(h_i)$ being the degree shifting number of the twisted sector $\D(h_i)$ in $\D$.
\end{itemize}
Similar notations apply to $\nu$.

For the invariant \eqref{E notation-for-rel-inv-DYD} we also set
\[
(\h)=\left((h_1),\ldots,(h_{\ell(\mu)})\right)\qq \mbox{and}\qq
(\h')=\left((h'_1),\ldots,(h'_{\ell(\mu)})\right).
\]
Suppose
\[
\I\pi\co \Y(\bar h_i)\rto\D(\underline{h_i})
\]
i.e. $\pi_t(\bar h_i)=(\underline {h_i})$ (cf. Lemma \ref{lem CR-of-Y}). Then we set $(\uh)=\left((\underline {h_1}),\ldots,(\underline {h_m})\right)$ and
\[
 \ugamma=\left((\uh),(\h),(\h')\right).
\]
Denote by
\[
n=\#(\bar h)+\#(\h)+\#(\h')=m+\ell(\mu)+\ell(\nu)
\]
the number of absolute marked points and relative marked points. Finally we set
\begin{align*}
\vec r =(r_1,\ldots,r_n):=\big(\mathrm{ord}(\underline {h_1}),\ldots,  \mathrm{ord}(\underline {h_m}),\mathrm{ord}(h_1), \ldots,\mathrm{ord}(h_{\ell(\mu)}), \mathrm{ord}(h'_1),\ldots,\mathrm{ord}(h'_{\ell(\nu)}) \big)
\end{align*}
where for example the order $\mathrm{ord}(\underline {h_1})$ is the order of a representative of $(\underline {h_1})$ in local groups and does not depend on the choices of representatives.

Similar notations also apply to relative invariants of $(\Y|\D_0)$ and $(\Y|\D_\0)$. We will also deal with rubber invariants of $(\D_0|\Y|\D_\0)$, where we use a superscript ``$\sim$'' to indicate rubber invariants.

The orbifold fibration $\pi\co\Y=\P(\L\oplus \mc O_\D)\rto \D$ induces a fibration of coarse space $|\pi|\co|\Y|\rto |\D|$, which is a topological $\P^1$-fiber bundle. A class $\beta\in H_2(|\Y|;\integer)$ is called a fiber class if
\[
|\pi|_*(\beta)=0.
\]

For a relative invariant of $(\D_0|\Y|\D_\0)$ as \eqref{E notation-for-rel-inv-DYD}, we call it a {\em fiber class invariant} if the homology class $\beta$ is a fiber class in $H_2(|\Y|;\integer)$. We next study fiber class invariants of $(\D_0|\Y|\D_\0)$.

\subsection{The moduli spaces of fiber class invariants}
\label{subsec 3.3}

Now consider a fiber class invariant of the form as \eqref{E notation-for-rel-inv-DYD}
\[
\bl \mu \bb  \prod_{i=1}^m\tau_{k_i} \gamma_{(\bar h_i)}^{l_i}\bb\nu\br_\Gamma^{(\D_0|\Y|\D_\0)}.
\]
So the homology class $\beta$ in the topological data $\Gamma=(g,\beta,(\bar\h),\vec\mu,\vec\nu)$ is a fiber class.

We next analyze the structure of the corresponding moduli space $\M_{\Gamma}(\D_0|\Y|\D_\0)$, which we denote simply by $\M_\Gamma$. It consists of equivalence classes of stable representable pseudo-holomorphic morphisms from (nodal) orbifold Riemann surfaces to $\Y$ of topological type indicated by $\Gamma$.

We next describe a typical element in $\M_\Gamma$. We first recall the construction of expanded (or degenerate) targets. Given two nonnegative integers $l_0,l_\0$, let $\Y[l_0,l_\0]$ be the degenerate orbifold groupoid (cf. \cite{Chen-Li-Sun-Zhao2011})
\begin{align}\label{E expanded-Y}
\Y[l_0,l_\0]:=\underbrace{\Y\wedge_{\D_\0}\ldots\wedge_{\D_\0}\Y}_{l_0} \wedge_{\D_\0}\Y\wedge_{\D_0}\underbrace{\Y \wedge_{\D_0} \ldots\wedge_{\D_0}\Y}_{l_\0},
\end{align}
i.e. we glue $l_0$ copies of $\Y$ to the original $\Y$ along the infinite section and $l_\0$ copies of $\Y$ to the original $\Y$ along the zero section. To distinguish them we number them as follows
\[
\Y[l_0,l_\0]=\Y_{-l_0}\wedge_{\D_\0}\ldots\wedge_{\D_\0}\Y_{-1} \wedge_{\D_\0}\Y_0\wedge_{\D_0}\Y_1\wedge_{\D_0} \ldots\wedge_{\D_0}\Y_{l_\0}.
\]
We call $\Y[l_0,l_\0]$ {\em an expanded target} when $l_0+l_\0>0$. The $\Y_0$ in $\Y[l_0,l_\0]$ is called the {\em root}, and the rest part is called the {\em rubber}. When $l_0=l_\0=0$, $\Y[0,0]=\Y_0=\Y$ is called {\em the unexpanded target}. Denote the zero section of $\Y_i$ by $\D_{0,i}$ and the infinite section of $\Y_i$ by $\D_{\0,i}$ for $-l_0\leq i\leq l_\0$. So from \eqref{E expanded-Y} we see that $\Y[l_0,l_\0]$ is obtained by gluing $\D_{0,i}$ in $\Y_i$ with $\D_{\0,i-1}$ in $\Y_{i-1}$ for $-l_0+1\leq i\leq l_\0$. So the singular set of $\Y[l_0,l_\0]$ is
\[
\text{Sing}\Y[l_0,l_\0]=\sqcup_{i=-l_0}^{l_\0-1} \D_{\0,i}=\sqcup_{i=-l_0+1}^{l_\0} \D_{0,i}.
\]
Let $\aut^{\text{rel}}_{l_0,l_\0}:= \aut(\Y[l_0,l_\0],\text{Sing}\Y[l_0,l_\0] \sqcup\Y_0)$ be the group of automorphisms of $\Y[l_0,l_\0]$ preserving the singular set $\text{Sing}\Y[l_0,l_\0]$ and the root $\Y_0$. Then $\aut^{\text{rel}}_{l_0,l_\0}\cong (\cplane^*)^{l_0+l_\0}$, where each factor of $(\cplane^*)^{l_0+l_\0}$ dilates the fibers of the $i$-th orbifold $\P^1$-bundle for $i\neq 0$. We have a natural map from $\Y[l_0,l_\0]$ to the root, $\Y[l_0,l_\0]\rto \Y_0=\Y$, which contracts all $\Y_i$ to $\D_{0,0}$ for $i<0$ and all $\Y_i$ to $\D_{\0,0}$ for $i>0$.

A map to $(\D_0|\Y|\D_\0)$ of topological type $\Gamma$ is a triple
\[
(\epsilon,\C',\f)\co\C\xleftarrow{\epsilon}\C'\xrightarrow{\f} \Y[l_0,l_\0]
\]
consists of the following ingredients:
\begin{enumerate}
\item[(1)] $\C$ is a genus $g$ (nodal) orbifold Riemann surface with (possible orbifold) absolute marked points $\x=(\x_1,\ldots,\x_m)$ and (possible orbifold) relative marked points $\y=(\y_1,\ldots,\y_{\ell(\mu)})$ and $\z=(\z_1,\ldots,\z_{\ell(\nu)})$,
\item[(2)] $\epsilon$ is a refinement of orbifold groupoid by an open cover of the object space,
\item[(3)] $\f$ is pseudo-holomorphic and the induced map $|\f|\co|\C|\rto|\Y[l_0,l_\0]|\rto|\Y_0|=|\Y|$ on coarse spaces satisfies $|\f|_*[|\C|]=\beta\in H_2(|\Y|;\integer)$,
\item[(4)] the marked points (resp. nodal points) in the coarse space $|\C|$ are divided into absolute marked points (resp. nodal points) and relative marked points (resp. nodal points) as follows,
    \begin{enumerate}
\item[(4.1)] the absolute marked points $|\x|$ and absolute nodal points are mapped into the nonsingular part of $|\Y[l_0,l_\0]|$, i.e. $|\Y[l_0,l_\0]|-|\text{Sing}(\Y[l_0,l_\0])|$,
\item[(4.2)] the relative marked points $|\y|$ are mapped into $|\D_{0,-l_0}|$ and $|\f|\inv(|\D_{0,-l_0}|)$ consists of only relative marked points $|\y|$, and the intersection multiplicities, i.e. contact orders, are given by $\mu$ such that the sum of all contact orders equals to $\D_0\cdot \beta$,
\item[(4.3)] the relative marked points $|\z|$ are mapped into $|\D_{\0,l_\0}|$ and $|\f|\inv(|\D_{\0,l_\0}|)$ consists of only relative marked points $|\z|$, and the intersection multiplicities, i.e. contact orders, are given by $\nu$ such that the sum of all contact orders equals to $\D_\0\cdot \beta$,
\item[(4.4)] the relative nodal points are mapped into $|\text{Sing}(\Y[l_0,l_\0])|$ and $|\f|\inv(|\text{Sing}(\Y[l_0,l_\0])|)$ consists of only relative nodal points,
\item[(4.5)] the relative nodal points in $|\f|\inv(|\text{Sing}(\Y[l_0,l_\0])|)$ satisfy the balanced condition that for each node $q\in |\f|\inv(|\D_{0,i}|=|\D_{\0,i-1}|)$, $i=-l_0+1,\ldots, l_\0$, the two branches of the domain curve $|\C|$ at the nodal point $q$ are mapped to different irreducible components of $|\Y[l_0,l_\0]|$ and the contact orders to $\D_{0,i}=\D_{\0,i-1}$ are equal,
    \end{enumerate}

\item[(5)] $\f$ is representable, i.e. the induced maps on local groups are injective.
\end{enumerate}

The equivalence relation between such maps are generated by the following relations:
\begin{enumerate}
\item[(i)] We say $\C\xleftarrow{\epsilon}\C'\xrightarrow{\f} \Y[l_0,l_\0]$ is equivalent to $\C\xleftarrow{\tilde \epsilon}\tilde{\C'}\xrightarrow{\tilde{\f}} \Y[l_0,l_\0]$ if there is a natural transformation (see for example \cite{Chen-Du-WangR2019})
    \[
    \alpha\co\f\circ\pi_1\Rightarrow \tilde{\f}\circ\pi_2 \co \C'\times_{\epsilon,\C,\tilde\epsilon}\tilde{\C'}\rto \Y[l_0,l_\0].
    \]
\item[(ii)] For an automorphism $\phi\in \aut_{l_0,l_\0}^{\text{rel}}$, we say $\C\xleftarrow{\epsilon}\C'\xrightarrow{\f} \Y[l_0,l_\0]$ is equivalent to $\C\xleftarrow{\epsilon}\C'\xrightarrow{\phi\circ\f} \Y[l_0,l_\0]$.
\end{enumerate}

A map $\C\xleftarrow{\epsilon}\C'\xrightarrow{\f} \Y[l_0,l_\0]$ is called a \emph{stable} map if its self-equivalences are finite. For simplicity, we will also denote such a map by $(\epsilon,\C',\f)\co(\C,\x,\y,\z)\rto (\D_0|\Y|\D_\0)$ or $(\epsilon,\C',\f)$. $\M_\Gamma(\D_0|\Y|\D_\0)$ is the space (in fact groupoid) of stable maps of topological type $\Gamma$. The arrow space of $\M_\Gamma(\D_\0|\Y|\D_\0)$ consists of equivalences between stable maps.

\begin{theorem}\label{thm structure-fiber-cls-moduli}
The moduli space $\M_\Gamma(\D_0|\Y|\D_\0)$ is a fibration over a certain multi-sector $\D_\ugamma'$ of $\D$ that determined by $\Gamma$, whose fiber is the relative moduli space of stable maps into
\[
([0\rtimes K_\ugamma]| [\P^1\rtimes K_\ugamma]|[\0\rtimes K_\ugamma])
\]
for certain $\P^1$-orbifold $[\P^1\rtimes K_\ugamma]=(K_\ugamma\times \P^1\rrto \P^1)$. The finite group $K_\ugamma$, the $\P^1$-orbifold $[\P^1\rtimes K_\ugamma]$ and the topological data of the relative moduli space of $([0\rtimes K_\ugamma]| [\P^1\rtimes K_\ugamma]|[\0\rtimes K_\ugamma])$ are all determined by $\Gamma$.
\end{theorem}

The rest of this subsection is devoted to explain and prove this theorem.

\subsubsection{Universal curve of orbifold Riemann surfaces}

Let $\scr T_{g,n}, 2g-3+n\geq 0$ be the Techim\"uller space of genus $g$, $n$ marked Riemann surfaces (without nodal points). Let $\pi\co \scr C_{g,n}\rto \scr T_{g,n}$ be the corresponding universal curve, equipped with $n$ canonical sections $\sigma_i,1\leq i\leq n$ corresponding to those $n$ marked points.

Take a point $b\in \scr T_{g,n}$, then we have a stable curve
\[
C_b:=(\pi\inv(b),\sigma_1(b),\ldots,\sigma_n(b)).
\]
Consider the punctured surface $C_b^\circ:=C_b\setminus\{\sigma_1(b),\ldots,\sigma_n(b)\}$. It has a canonical hyperbolic metric of constant curvature $-1$. Then around every puncture we have a series of horocycles. The mapping class group $MP_{g,n}$ acts on $\scr T_{g,n}$ and also on $\scr C_{g,n}$. We could take $n$ $MP_{g,n}$-invariant positive functions $\delta_i$ over $\scr T_{g,n}$, and use them as the radius of the horocycles at the $n$ punctures. Then for each $C_b^\circ$, at the $i$-th puncture we remove a horodisc whose horocycle is of length $\delta_i$.

In other words, for each $C_b$ we pick out $n$ discs around the $n$ marked points such that these discs are $MP_{g,n}$-invariant. These discs together give rise to a tubular neighborhood $U_i$ of $\sigma_i(\scr T_{g,n})$ in $\scr C_{g,n}$ for $1\leq i\leq n$. See for example \cite{Chen-Li-Wang2016}.

Now we construct a family of orbifold Riemann surfaces. We first cut out the sections of marked points
\[
\bigsqcup_{i=1}^n \sigma_i(\scr T_{g,n})
\]
from $\scr C_{g,n}$ to get $\scr C_{g,n}^\circ$, and then patch the
\[
\integer_{r_i}\ltimes U_i
\]
back to $\scr C_{g,n}^\circ$. Here we identify
\begin{align}\label{E Z-r-i-zeta-i}
\integer_{r_i}=\langle\zeta_i\rangle\qq \mathrm{where} \qq \zeta_i=e^{\frac{2\pi\sqrt{-1}}{r_i}},
\end{align}
and $\integer_{r_i}$ acts on $U_i$ by rotating the horodiscs. Then we get an effective orbifold $\C_{g,n}$. Moreover, from $\scr C_{g,n}^\circ$ and $\integer_{r_i}\ltimes U_i$ we get an orbifold groupoid that represents $\C_{g,n}$, which we still denote by $\C_{g,n}$. The object space of this groupoid is
\[
C^0_{g,n}:=\scr C_{g,n}^\circ \sqcup \bigsqcup_{i=1}^n U_i.
\]
The arrow space $C^1_{g,n}$ is obtained accordingly, since $\C_{g,n}$ is effective. So this orbifold groupoid is $\C_{g,n}=(C^1_{g,n}\rrto C^0_{g,n})$.

We have a natural projection
\[
\pi\co \C_{g,n}\rto \scr T_{g,n}.
\]
which on $C^0_{g,n}$ is obtained from the natural projections $\scr C^\circ_{g,n}\hrto \scr C_{g,n}\rto \scr T_{g,n}$ and $U_i\rto\scr T_{g,n}$. The fiber of this projection $\pi\co \C_{g,n}\rto \scr T_{g,n}$ is an orbifold Riemann surface with $n$ orbifold points whose local groups are $\integer_{r_i},1\leq i\leq n$. This is a family of orbifold Riemann surfaces over $\scr T_{g,n}$.

Moreover by the choices of $U_i$, $MP_{g,n}$ acts on $\C_{g,n}$. First of all, it acts on $C^0_{g,n}=\scr C^\circ_{g,n}\sqcup \bigsqcup_i U_i$. Secondly, since the horodiscs are $MP_{g,n}$-equivariant, and the $\integer_{r_i}$-action also commutes with the $MP_{g,n}$-action on horodiscs, $MP_{g,n}$ acts on $\C_{g,n}$. Again, by the choices of horodiscs, the projection $\C_{g,n}\rto\scr T_{g,n}$ is $MP_{g,n}$-equivariant. Note that $MP_{g,n}$-action on $C^0_{g,n}$ and $C^1_{g,n}$ are both free. So we get a family of orbifold Riemann surface
\begin{align}\label{E family-of-orb-sur-over-mgn}
\C_{g,n}/MP_{g,n}\rto \scr T_{g,n}/MP_{g,n}.
\end{align}
Note that the quotient $\scr T_{g,n}/MP_{g,n}$ is the top strata $\scr M_{g,n}$ of $\M_{g,n}$, the Deligne--Mumford moduli space of Riemann surfaces. The groupoid corresponding to this top strata is
\[
\msf M_{g,n}=\scr T_{g,n}\rtimes MP_{g,n}=(MP_{g,n}\times \scr T_{g,n}\rrto\scr T_{g,n}).
\]
In terms of groupoid, we could write \eqref{E family-of-orb-sur-over-mgn} as
\[
(C^1_{g,n}\times MP_{g,n}\rrto C^0_{g,n})\rto(MP_{g,n}\times \scr T_{g,n}\rrto\scr T_{g,n})=\msf M_{g,n}.
\]
That is $\C_{\scr M}:=(C^1_{g,n}\times MP_{g,n}\rrto C^0_{g,n})$ corresponds to $\C_{g,n}/MP_{g,n}$. So $\C_{\scr M}$ is a family of orbifold Riemann surfaces over $\scr M_{g,n}$. We can view it as a universal curve of orbifold Riemann surfaces over $\scr M_{g,n}$. In the same way we could construct the universal curve of orbifold Riemann surfaces over lower strata of $\M_{g,n}$.

\subsubsection{The top strata $\scr M_\Gamma$}\label{subsubsec top-strata}

We next first study the top strata $\scr M_\Gamma$ of $\M_\Gamma(\D_0|\Y|\D_\0)$. By top strata we mean that for each stable maps in $\scr M_\Gamma$ its domain curve has no nodal points. Hence the target is the unexpanded target, i.e. $\Y$ itself. Take a fiber $\C_b=\pi\inv(b)$ of $\C_{g,n}$ over a point $b\in\scr T_{g,n}$. So $\C_b$ is a genus $g$ orbifold Riemann surface with $n$ (possible orbifold) marked points. As above we denote these $n=m+\ell(\mu)+\ell(\nu)$ marked points orderly by $(\x,\y,\z)$ with $\x=(\x_1,\ldots,\x_m)$, $\y=(\y_1,\ldots,\y_{\ell(\mu)})$ and $\z=(\z_1,\ldots,\z_{\ell(\nu)})$.

Now consider the groupoid of stable morphisms from $(\C_b,\x,\y,\z)$ to $(\D_0|\Y|\D_\0)$ of topological type $\Gamma$
\[
\hol_\Gamma(\C_b,\Y)=(\mhol^1_\Gamma(\C_b,\Y)\rrto \mhol^0_\Gamma(\C_b,\Y)),
\]
where
\begin{enumerate}
\item[(i)]
$\mhol^0_\Gamma(\C_b,\Y)$ is the space of stable morphisms from $\C_b$ to $\Y$ of topological type $\Gamma$,

\item[(ii)]
$\mhol^1_\Gamma(\C_b,\Y)$ is the space of natural transformations of stable morphisms in $\mhol^0_\Gamma(\C_b,\Y)$.
\end{enumerate}

As a groupoid, the object space of the top strata $\scr M_\Gamma$ is
\[
\bigcup_{b\in\scr T_{g,n}}\mhol^0_\Gamma(\C_b,\Y).
\]
The arrow space consists of two parts. The first part is
\[
\bigcup_{b\in\scr T_{g,n}}\mhol^1_\Gamma(\C_b,\Y).
\]
The second part comes from the action of mapping class group. The action of mapping class group commutes with the action of $\bigcup_{b\in\scr T_{g,n}}\mhol^1_\Gamma(\C_b,\Y)$. This is similar to the arrow space of $\C_{\scr M}$.

We next study the structure of holomorphic morphisms in $\bigcup_{b\in\scr T_{g,n}}\mhol^0_\Gamma(\C_b,\Y)$.

In the next, we omit the subscript $b$ of $\C_b$ to simplify notations. Since the curve class $\beta$ in $\Gamma$ is a fiber class of the fiber bundle $\pi\co\Y=(Y^1\rrto Y^0)=\D\ltimes Y^0\rto \D=(D^1\rrto D^0)$, by \cite[Theorem 7.4]{Chen-Du-WangY2020} we could assume that every morphism
\[
\C\xleftarrow{\epsilon}\C'=(C^{'1}\rrto C^{'0})\xrightarrow{\f=(f^0,f^1)}\Y
\]
in $\mhol^0_\Gamma(\C,\Y)$ satisfies that $\pi(f^0(C^{'0}))$ is a point in $D^0$, i.e. the image $f^0(C^{'0})$ lies in a single fiber of $Y^0$.

Now take a morphism $\C\xleftarrow{\epsilon}\C'=(C^{'1}\rrto C^{'0})\xrightarrow{\f=(f^0,f^1)}\Y\in \mhol^0_\Gamma(\C,\Y)$. Suppose $\pi\circ f^0(C^{'0})=p_{\f}\in D^0$. Let $G_{p_\f}$ be the local group of $p_\f$ in $\D$. The fiber of $Y^0\rto D^0$ over $p_\f$ is $\P(L^0_{p_\f}\oplus \cplane)$, and the fiber of $\Y\rto\D$ over $p_\f$ is modeled by $[\P(L^0_{p_\f}\oplus\cplane)\rtimes G_{p_\f}]$ where $G_{p_\f}$ acts on $\P(L^0_{p_\f}\oplus\cplane)$ via acting on $L^0_{p_\f}$. So $\C\xleftarrow{\epsilon}\C'=(C^{'1}\rrto C^{'0})\xrightarrow{\f=(f^0,f^1)}\Y$ factors through
\[
\xymatrix{
\C& \C'\ar[l]_-{\epsilon}\ar[r]^-{\f} & [\P(L^0_{p_\f}\oplus\cplane)\rtimes G_{p_\f}] \ar@{^(->}[r] & \Y.}
\]

By a result of Chen--Ruan (cf. \cite[Theorem 2.54]{Adem-Leida-Ruan2007}), the representable pseudo-holomorphic map $\f\co \C'\rto [\P(L^0_{p_\f}\oplus\cplane)\rtimes G_{p_\f}] \hrto\Y$ is determined by the homomorphism
\[
\rho_\f:\pi^{\text{orb}}_1(\C')\rto G_{p_\f},
\]
where $\pi^{\text{orb}}_1(\C')=\pi_1^{\text{orb}}(\C)$ is the orbifold fundamental group of $\C'$ and has a representation
\[
\pi^{\text{orb}}_1(\C')=\left\langle \lambda_1,\ldots,\lambda_n, \alpha_1,\ldots,\alpha_{2g}\left | \prod_{i=1}^n\lambda_i \prod_{j=1}^{g}(\alpha_{2j-1}\alpha_{2j}\alpha_{2j-1}\inv \alpha_{2j}\inv)=1,\lambda_i^{r_i}=1 \right. \right\rangle,
\]
with $\lambda_i$ corresponding to those marked points, i.e. to the generators $\zeta_i$ of local groups of those marked points (cf. \eqref{E Z-r-i-zeta-i}), and $\alpha_j$ corresponding to the generators of the fundamental group of the coarse space $|\C'|=|\C|$ (a smooth genus $g$ Riemann surface). Suppose the images of the generators of $\pi^{\text{orb}}_1(\C')$ w.r.t $\rho_\f$ are
\[
\rho_\f(\lambda_i)=h_i\in G_{p_\f},\qq \rho_\f(\alpha_j)=\tilde h_j\in G_{p_\f}.
\]
Then $\pi\circ\f$ is determined by the $(n+2g)$-tuple\footnote{As the morphism is representable, the orders of $h_i$ is the same as the orders of $\lambda_i$, i.e. $r_i$ for $1\leq i\leq n$.}
\[
\h_\f:=(h_1,\ldots,h_n,\tilde h_1,\ldots, \tilde h_{2g}).
\]
This $(n+2g)$-tuple gives rise to a point in the object space of the $(n+2g)$-th multiple-sector\footnote{The multiple-sector $\D^{[n+2g]}$ is defined as follows (cf. \cite{Adem-Leida-Ruan2007}). The object space is
\[
(\D^{[n+2g]})^0:=\{(g_1,\ldots, g_{n+2g})\in (D^1)^{n+2g}\mid s(g_i)=t(g_j),1\leq i,j\leq n+2g\}.
\]
The $\D$-action on $(\D^{[n+2g]})^0$ is given by the anchor map $e\co(\D^{[n+2g]})^0\rto D^0, e(g_1,\ldots, g_{n+2g})=s(g_1)$ and the action map
\[
(\D^{[n+2g]})^0\times_{e,s} D^1\rto (\D^{[n+2g]})^0,\qq (g_1,\ldots,g_{n+2g};h)\mapsto (h\inv g_1 h,\ldots, h\inv g_{n+2g}h).
\]
So $(\D^{[n+2g]})^1=(\D^{[n+2g]})^0\times_{e,s} D^1$.
} $\D^{[n+2g]}=((\D^{[n+2g]})^1\rrto (\D^{[n+2g]})^0)=\D\ltimes (\D^{[n+2g]})^0$ of $\D$. So $p_\f\in e((\D^{[n+2g]})^0)\subseteq D^0$.

So we have a map
\[
\rho\co\mhol^0_\Gamma(\C,\Y)\rto (\D^{[n+2g]})^0,\qq (\epsilon,\C',\f)\mapsto \h_\f.
\]
Moreover, an $\C\xleftarrow{\epsilon}\C'=(C^{'1}\rrto C^{'0})\xrightarrow{\f=(f^0,f^1)}\Y$ in $\mhol^0_\Gamma(\C,\Y)$ further factors through
\begin{align}\label{eq the-actual-fiber-rel-maps}
\xymatrix{
\C &\C'\ar[l]_-\epsilon     \ar[r]^-{\bar{\f}}&
[\P(L^0_{p_\f}\oplus\cplane)\rtimes \langle \h_\f\rangle ]  \ar@{^{(}->}[r] &
[\P(L^0_{p_\f}\oplus\cplane)\rtimes G_{p_\f}  ]       \ar@{^{(}->}[r] &
\Y.}
\end{align}
Meanwhile, $\Gamma$ determines a topological data $\bgamma$ such that the induced morphism $\C\xleftarrow{\epsilon}\C'\xrightarrow{\bar\f} [\P(L^0_{p_\f}\oplus\cplane)\rtimes \langle \h_\f\rangle ]$ to the $\P^1$-orbifold $[\P(L^0_{p_\f}\oplus\cplane)\rtimes \langle \h_\f\rangle ]$ is of topological type $\bgamma$. Explicitly, the genus, homology class and marked points of $\bgamma$ are the same as the $\Gamma$; the only difference is that the twisted sectors of $\bgamma$ are induced from those of $\Gamma$ via viewing those $h_i$ as elements in $\langle \h_\f\rangle$ instead of $G_{p_\f}$. By this further factorization \eqref{eq the-actual-fiber-rel-maps} we see

\begin{lemma}\label{lem fiber-of-rho}
$\rho\co\mhol^0_\Gamma(\C,\Y)\rto (\D^{[n+2g]})^0$ is a fibration over its image whose fiber over a point $\h=(h_1,\ldots,h_n,\tilde h_1,\ldots, \tilde h_{2g})$ is
\begin{align*}
&\widetilde\mhol^0_\bgamma(\C,[\P(L^0_{e(\h)}\oplus\cplane)\rtimes \langle \h\rangle])\\
:=&\left\{
\C\xleftarrow{\epsilon}\C'\xrightarrow{\f} [\P(L^0_{e(\h)}\oplus\cplane)\rtimes \langle \h\rangle]\left|(\epsilon,\C',\f)\text{ is of type }\bgamma \text{ and } \rho_\f(\lambda_i)=h_i, \rho_\f(\alpha_j)=\tilde h_j\right.
\right\}.
\end{align*}
\end{lemma}

Denote the image of $\rho$ by $D^0_\ugamma\subseteq (\D^{[n+2g]})^0$. Then we get a sub-groupoid $\D^{[n+2g]}|_{D^0_{\ugamma}}$, which we denoted by $\D_{\ugamma}=(D_\ugamma^1\rrto D_\ugamma^0)$. So $D_\ugamma^1\subseteq (\D^{[n+2g]})^0\times_{e,s}D^1$.

\begin{lemma}\label{lem D-Gamma-action-on-Hol}
With $\rho\co\mhol^0_\Gamma(\C,\Y)\rto D^0_\ugamma$ as the anchor map, there is a $\D_\ugamma$-action on $\mhol^0_\Gamma(\C,\Y)$ given as follows. For an arrow $(\h,h)\in D_\ugamma^1$ and a morphism $\C\xleftarrow{\epsilon}\C'\xrightarrow{\f} \Y$ with $\h_\f=\h$ (i.e. $\C\xleftarrow{\epsilon}\C'\xrightarrow{\f} \Y$ belongs to the fiber of $\mhol^0_\Gamma(\C,\Y)$ over $\h$), $h$ acts on $\C\xleftarrow{\epsilon}\C'\xrightarrow{\f} \Y$ by transforming the image of $f^0$ from the fiber of $Y^0$ over $s(h)$ to the fiber of $Y^0$ over $t(h)$ and conjugate $f^1$ by $h$, in particular, it transfer $\h_\f=\h$ into $h\inv\cdot \h\cdot h$.
\end{lemma}

Denote the resulting action groupoid by
\[
\D_\ugamma\ltimes \mhol^0_\Gamma(\C,\Y).
\]
Then again by the factorization \eqref{eq the-actual-fiber-rel-maps} we have
\begin{lemma}\label{L 3.6}
\[
\hol_\Gamma(\C,\Y)\cong \D_\ugamma\ltimes \mhol^0_\Gamma(\C,\Y)
\]
\end{lemma}

Therefore we have a groupoid projection
\[
\hol_\Gamma(\C,\Y)\rto\D_\ugamma.
\]
Note that $\D_\ugamma$ may have several connected components (i.e. connected components of its coarse space). In a single component, the group $\langle \h\rangle$ is invariant up to conjugation. {\em For simplicity, in the following we assume that $\D_\ugamma$ has only one component, hence the group $\langle \h\rangle $ is invariant up to conjugation. Otherwise, we only need to deal with components separately.}

Now we fix an $\h_\ugamma\in D_\ugamma^0$ and set $K_\ugamma=\langle\h_\ugamma\rangle$. Then for every $\h\in D_\ugamma^0$, $K_\ugamma\cong \langle \h\rangle$ via conjugation. Moreover, this isomorphism also identifies the $K_\ugamma$-action  on $L^0_{e(\h_\ugamma)}$ with the $\langle \h\rangle$-action on $L^0_{e(\h)}$. Then by Lemma \ref{lem fiber-of-rho} and Lemma \ref{lem D-Gamma-action-on-Hol} we see that the fiber of $\rho \co\mhol^0_\Gamma(\C,\Y)\rto D^0_\ugamma$ are all isomorphic to
\begin{align}\label{E unified-fiber-of-stable-maps}
&\widetilde\mhol^0_\bgamma(\C,[\P(L^0_{e(\h_\ugamma)}\oplus\cplane) \rtimes K_\ugamma])\\
:=&\left\{
\C\xleftarrow{\epsilon}\C'\xrightarrow{\f} [\P(L^0_{e(\h_\ugamma)}\oplus\cplane)\rtimes K_\ugamma] \left|
(\epsilon,\C',\f)\text{ is of type }\bgamma \text{ and }
(\ldots,\rho_\f(\lambda_i),\ldots, \rho_\f(\alpha_j), \ldots)=\h_\ugamma
\right.
\right\}.\nonumber
\end{align}
Moreover, this fibration $\rho \co\mhol^0_\Gamma(\C,\Y)\rto D^0_\ugamma$ is locally trivial.

By viewing $\hol_\Gamma(\C,\Y)\rto\D_\ugamma$ as a groupoid fibration (cf. \cite{Chen-Du-WangY2020}), we see that $\hol_\Gamma(\C,\Y)$ is a fibration over $\D_\ugamma$ with fiber being the unitary/trivial groupoid (cf. \cite{Moerdijk-Mrcun2003})
\[
\widetilde\mhol^0_\bgamma(\C,[\P(L^0_{e(\h_\ugamma)}\oplus\cplane) \rtimes K_\ugamma])\rrto \widetilde\mhol^0_\bgamma(\C,[\P(L^0_{e(\h_\ugamma)}\oplus\cplane) \rtimes K_\ugamma])
\]
associated to the space $\widetilde\mhol^0_\bgamma(\C,[\P(L^0_{e(\h_\ugamma)}\oplus\cplane) \rtimes K_\ugamma])=\widetilde\mhol^0_\bgamma(\C,[\P^1\rtimes K_\ugamma])$.

We next interpret $\hol_\Gamma(\C,\Y)$ as a groupoid fibration over another groupoid $\D'_\ugamma$ with fiber being groupoids that correspond to stable maps to the $\P^1$-orbifold $[\P^1\rtimes K_\ugamma]$. We first construct the groupoid $\D'_\ugamma$. It is similar to the construction of the effective orbifold groupoid for an ineffective orbifold groupoid (cf. \cite[Definition 2.33]{Adem-Leida-Ruan2007}).

Consider the following subspace of $D_\ugamma^1$:
\[
\ker D_\ugamma^1:=\{(\h,h)\in D_\ugamma^1\mid h\in C_{\langle\h\rangle}(\h)\subseteq C_{G_{e(\h)}}(\h)\},
\]
where, as above, $G_{e(\h)}$ is the local (or isotropy) group of $e(\h)\in D^0$ in $\D$ and $\langle\h\rangle$ is the subgroup generated by $\h$, $C_{G_{e(\h)}}(\h)$ is the centralizers of $\h$ in $G_{e(\h)}$, and $C_{\langle\h\rangle}(\h)$ is the centralizers of $\h$ in $\langle \h\rangle$, hence the center of $\langle \h\rangle$ as $\langle\h\rangle$ is generated by $\h$. We define a relation on $D_\ugamma^1$ by
\begin{align}\label{E equivalence-relation-on-D-ugamma1}
(\h,h)\sim (\h,kh)
\end{align}
for all $k\in C_{\langle\h\rangle}(\h)$. This is obvious an equivalence relation.

\begin{remark}
In fact, the restrictions of structure maps of $D^1_\ugamma$ over $\ker D_\ugamma^1$ gives rise to an orbifold groupoid $\ker \D_\ugamma:=(\ker D^1_\ugamma\rrto D_\ugamma^0)$.
Then \eqref{E equivalence-relation-on-D-ugamma1} gives rise to an action of $\ker \D_\ugamma^1$ on $D^1_\ugamma$, whose anchor map is the source map $s\co D^1_\ugamma\rto D^0_\ugamma$. So \eqref{E equivalence-relation-on-D-ugamma1} is an equivalence relation.
\end{remark}

We set
\[
D^{',1}_\ugamma:=D_\ugamma^1/\sim.
\]
Then one see that all structure maps of $\D_\ugamma$ descends to $D^{',1}_\ugamma$ and $D_\ugamma^0$. Moreover,
\[
\D'_\ugamma:=(\D^{',1}_\ugamma\rrto D^0_\ugamma)
\]
is an orbifold groupoid.

\begin{remark}
As all $\h\in D^0_\ugamma$ are conjugate to each other, $\ker D_\ugamma$ is a trivial bundle over $D^0_\ugamma$ with fiber being $C_{K_\ugamma}(\h_\ugamma)$. So $\D_\ugamma\rto \D'_\ugamma$ is a $C_{K_\ugamma}(\h_\ugamma)$-gerbe over $\D'_\ugamma$.
\end{remark}

On the other hand, for a fixed $\h\in D^0_\ugamma$, $C_{\langle \h\rangle}(\h)$ is a normal subgroup of $C_{G_{e(\h)}}(\h)$ and
\[
C_{\langle \h\rangle}(\h)\ltimes \widetilde\mhol^0_\bgamma(\C,\P(L^0_{e(\h)}\oplus\cplane)\rtimes \langle \h\rangle)
=\hol_\bgamma(\C,\P(L^0_{e(\h)}\oplus\cplane)\rtimes \langle \h\rangle)
\]
is the groupoid of morphisms of type $\bgamma$ from $\C$ to $\P(L^0_{e(\h)}\oplus\cplane)\rtimes \langle \h\rangle$. As every $\h$ is conjugate to $\h_\ugamma$ and $K_\ugamma=\langle\h_\ugamma\rangle$, we have
\[
\hol_\bgamma(\C,\P(L^0_{e(\h)}\oplus\cplane)\rtimes \langle\h\rangle )\cong
\hol_\bgamma(\C,\P^1\rtimes K_\ugamma).
\]
\begin{lemma}\label{lem fiber-moduli-top-strata}
The composition of projections
\[
\hol_\Gamma(\C,\Y)\rto \D_\ugamma\rto \D_\ugamma'
\]
makes $\hol_\Gamma(\C,\Y)$ a fibration over $\D'_\ugamma$ with fiber $\hol_\bgamma(\C,\P^1\rtimes K_\ugamma)$. We denote this groupoid fibration by
\[
\hol_\bgamma(\C,\P^1\rtimes K_\ugamma)\hrto \hol_\Gamma(\C,\Y)\rto\D_\ugamma'.
\]
\end{lemma}
See for \cite[\S 3]{Chen-Du-WangY2020} for the definition of groupoid fibration. One can think it as a groupoid fiber bundle with fiber being groupoids.
\begin{proof}
By Lemma \ref{L 3.6} we only have to show that the composed projection
\[
\hol_\Gamma(\C,\Y)\rto \D_\ugamma\rto \D_\ugamma'
\]
has fiber being isomorphic to $\hol_\bgamma(\C,\P^1\rtimes K_\ugamma)$.

First of all, the projection on object space is
\[
\rho\co\mhol^0_\Gamma(\C,\Y)\rto D_\ugamma^0.
\]
Secondly, on arrows, the projection is the composition
\[
\rho^1\co D_\ugamma^1\times_{s,D_\ugamma^0,\rho} \mhol^0_\Gamma(\C,\Y)\rto D_\ugamma\rto D_\ugamma^{',1}
\]
which is
\[
((\h,h);\C\xleftarrow{\epsilon}\C'\xrightarrow{\f}\Y)\mapsto [\h,h],
\]
where $[\h,h]$ is the equivalence class of $(\h,\h)$ in $D_\ugamma^{',1}=D_\ugamma^1/\sim$.

Now we consider the fiber of this projection. Take a point $\h\in D_\ugamma^0$ and consider the identity arrow $[\h,1_{p_\h}]\in D_\ugamma^{',1}$, where $p_\h=e(\h)\in D^0$. Then the inverse images of $\h$ and $[\h,1_{p_\h}]$ are
\[
(\rho^0)\inv(\h)= \widetilde{\mhol}^0_\bgamma(\C,\P(L^0_{p_\h}\oplus\cplane)\rtimes \langle \h\rangle),
\]
and
\[
(\rho^1)\inv([\h,1_{p_\h}])=C_{\langle \h\rangle}(\h)\times \widetilde{\mhol}^0_\bgamma(\C,\P(L^0_{p_\h}\oplus\cplane)\rtimes \langle \h\rangle)
\]
respectively. Therefore, the fiber of $\rho$ over $([\h,1_{p_\h}]\rrto \h)$ is
\[
C_{\langle \h\rangle}(\h)\ltimes \widetilde{\mhol}^0_\bgamma(\C,\P(L^0_{p_\h}\oplus\cplane)\rtimes \langle \h\rangle)\cong\hol_\bgamma(\C,\P^1\rtimes K_\ugamma).
\]
This finishes the proof.
\end{proof}

Now we vary $\C$ in $\C_{g,n}$. Note that the $MP_{g,n}$-action on $\bigcup_{b\in\scr T_{g,n}}\mhol^0_\Gamma(\C_b,\Y)$ commutes with the $\D_\ugamma$-action, we see that the top strata $\scr M_\Gamma(\D_0|\Y|\D_\0)$ is a fibration over $\D'_\ugamma$ with fiber $\scr M_\bgamma([0\rtimes K_\ugamma]|[\P^1\rtimes K_\ugamma]|[\0\rtimes K_\ugamma])$.

\subsubsection{The moduli space $\M_\Gamma$}
Now for the whole moduli space $\M_\Gamma$ we have
\begin{lemma}\label{lem fiber-moduli-extension}
The moduli space $\M_\Gamma(\D_0|\Y|\D_\0)$ is a groupoid fibration over $\D_\ugamma'$ with fiber being $\M_\bgamma([0\rtimes K_\ugamma]|[\P^1\rtimes K_\ugamma]|[\0\rtimes K_\ugamma])$.
\end{lemma}
\begin{proof}
Consider a general stable map $\C\xleftarrow{\epsilon}\C'\xrightarrow{\f} \Y[l_0,l_\0]$ of topological type $\Gamma$. As above by \cite[Theorem 7.4]{Chen-Du-WangY2020} we could assume that the image of objects $f^0(C^{'0})$ is a point $p_\f$ in $D^0$. Then as the analysis in \S \ref{subsubsec top-strata} the restriction of $\f$ over each irreducible component of $\C'$ is determined by the corresponding homomorphism between orbifold fundamental group of the irreducible component and the local group $G_{p_\f}$ of $p_\f$ in $\D$. Each nodal point $q$ in $\C$ contributes a generator, say $\lambda_{q,+}$ or $\lambda_{q,-}$ of finite order (determined by the twisted sectors of $\D_0,\D_\0$ or $\Y$ that this nodal point mapped into), to the orbifold fundamental groups of the two irreducible branches of $\C'$ at $q$ respectively. Then one see that the induced homomorphisms on orbifold fundamental groups of irreducible components of $\C'$ must satisfy that the images of $\lambda_{q,+},\lambda_{q,-}$ are inverse to each other for each nodal point $q$. Conversely, if the homomorphisms on orbifold fundamental groups of irreducible components of $\C'$ satisfy that the images of $\lambda_{q,+},\lambda_{q,-}$ are inverse to each other for all nodal points, then they determine an $\f$. Then one see that similar results as Lemma \ref{lem fiber-of-rho}, Lemma \ref{lem D-Gamma-action-on-Hol} and Lemma \ref{L 3.6} hold for general stable maps. Therefore we have the analogue of Lemma \ref{lem fiber-moduli-top-strata}.
\end{proof}

This also finishes the proofs of Theorem \ref{thm structure-fiber-cls-moduli}.

\begin{remark}\label{R 3.11}
The moduli spaces for fiber class relative invariants of $(\Y|\D_0)$ and $(\Y|\D_\0)$ have similar structures.

More generally, one can consider the weighted projectification $\oE_\wa$ of general orbifold vector bundles $\E\rto\os$ and the fiber class relative Gromov--Witten invariants of $(\oE_\wa|\PE_\wa)$. The corresponding moduli spaces have similar structure as Theorem \ref{thm structure-fiber-cls-moduli} and Lemma \ref{lem fiber-moduli-extension} state, that is they are fibrations over certain multi-sectors of $\os$ with fibers being relative moduli spaces of quotients of $(\msf P_{\wa,1}|\msf P_\wa)$ by certain finite groups.
\end{remark}

\subsection{The invariants}
We next compute the invariants. Consider the deformation-obstruction theory of $\M_\Gamma(\D_0|\Y|\D_\0)$, which at a point $(\epsilon,\C',\f)\co(\C,\x,\y,\z)\rto(\D_0|\Y|\D_\0)$ is
\begin{align}\label{E def-obs-complex-1}
0&\rto \aut(\C,\x,\y,\z) \rto \text{Def}(\f) \rto T^1_{(\C,\x,\y,\z,\f)}\rto\\
&
\rto \text{Def}(\C,\x,\y,\z) \rto \text{Obs}(\f) \rto T^2_{(\C,\x,\y,\z,\f)}\rto 0,\nonumber
\end{align}
where $\aut(\C,\x,\y,\z)$ is the space of infinitesimal automorphism of the domain $(\C,\x,\y,\z)$, $\text{Def}(\C,\x,\y,\z)$ is the space of infinitesimal deformation of the domain $(\C,\x,\y,\z)$,  $\text{Def}(\f)=H^0(\C,\f^* T\Y(-\D_0-\D_\0))$ is the space of infinitesimal deformation of the map $\f$, and $\text{Obs}(\f)=H^1(\C,\f^* T\Y(-\D_0-\D_\0))$ is the space of obstruction to deforming $\f$.

According to the splitting of $T\Y(-\D_0-\D_\0)$ into fiber part and base part, the \eqref{E def-obs-complex-1} splits into
\begin{align}\label{E fiber-def-obs-complex}
0&\rto \aut(\C,\x,\y,\z) \rto H^0(\C,\f^* T_{\mathrm{fiber}}\Y(-\D_0-\D_\0)) \rto \overline T^1_{(\C,\x,\y,\z,\f)}\rto\\
&\rto \text{Def}(\C,\x,\y,\z) \rto H^1(\C,\f^* T_{\mathrm{fiber}}\Y(-\D_0-\D_\0)) \rto \overline T^2_{(\C,\x,\y,\z,\f)}\rto 0,\nonumber
\end{align}
and
\begin{align}\label{E base-def-obs-complex}
0&\rto 0\rto H^0(\C,(\pi\circ\f)^*T\D) \rto \underline T^1_{(\C,\x,\y,\z,\f)}\rto\\
&\rto 0 \rto H^1(\C,(\pi\circ\f)^*T\D) \rto \underline T^2_{(\C,\x,\y,\z,\f)}\rto 0,\nonumber
\end{align}
where $\pi:\Y\rto\D$.

Fiberwisely, $\overline T^1_{(\C,\x,\y,\z,\f)}-\overline T^2_{(\C,\x,\y,\z,\f)}$ is the deformation-obstruction theory of $\M_\bgamma([0\rtimes K_\ugamma]\mid[\P^1\rtimes K_\ugamma]\mid[\0\rtimes K_\ugamma])$. On the other hand from \eqref{E base-def-obs-complex} we have
\[
\underline T^1_{(\C,\x,\y,\z,\f)}- \underline T^2_{(\C,\x,\y,\z,\f)} =H^0(\C,(\pi\circ\f)^*T\D)-H^1(\C,(\pi\circ\f)^*T\D).
\]
It gives rise to the obstruction theory of stable maps of genus $g$, degree zero and type $\ugamma$ in $\D$. By the index theorem of Chen--Ruan \cite[Theorem 4.2.2]{Chen-Ruan2004} and the same proof of \cite[Theorem 3.2]{Hu-Wang2013} (see also \cite[Proposition 1]{Chen-Hu2006}) we have
\begin{align*}
\text{rank}\,(\underline T^1_{(\C,\x,\y,\z,\f)}- \underline T^2_{(\C,\x,\y,\z,\f)} )
&=\dim\D(1-g)-\iota^\D(\h)-\iota^\D(\h')-\iota^\D(\uh)\\
&=-\dim\D\cdot g+(\dim\D-\iota^\D(\h)-\iota^\D(\h')-\iota^\D(\uh))\\
&=-\dim\D\cdot g+\dim\D_\ugamma-\mbox{rank}\ E_\ugamma
\end{align*}
where $E_\ugamma$ is the obstruction bundle over $\D_\ugamma$ defined in \cite[\S 4.2]{Chen-Ruan2004}, whose fiber is the co-kernel of $\bar\p$ operator.

The fiberwise deformation-obstruction theory gives rise to the $\pi$-relative virtual fundamental class
\[
[\M_\Gamma(\D_0|\Y|\D_\0)]^{\vir_\pi},
\]
and we have
\[
[\M_\Gamma(\D_0|\Y|\D_\0)]^\vir
=(c_{\mathrm{top}}(\mathbb E\boxtimes T\D)\cup
c_{\mathrm{top}}(E_{\ugamma})) \cap[\M_\Gamma(\D_0|\Y|\D_\0)]^{\vir_\pi}
\]
where $\mb E$ is the Hodge bundle over
$\M_\bgamma([0\rtimes K_\ugamma]|[\P^1\rtimes K_\ugamma]|[\0\rtimes K_\ugamma])$.

Note that
\[
c_{\mathrm{top}}(\mathbb E\boxtimes T\D)
=\sum_q h_q(c_1(\mathbb E),c_2(\mathbb E),\ldots)
t_q(c_1(T\D),c_2(T\D),\ldots).
\]
where $h_q$ and $t_q$ are polynomials.

Therefore the invariant \eqref{E notation-for-rel-inv-DYD} is
\begin{align*}
&\int_{[\M_\Gamma(\D_0|\Y|\D_\0)]^\vir} \prod_{i=1}^m\overline\psi_i^{k_i} {\sf ev}_i^*(\gamma_{(\bar h_i)}^{l_i}) \cup \prod_{j=1}^{\ell(\mu)} {\sf rev}^{\D_0,*}_j(\theta_{(h_j)}^{s_j})
\cup \prod_{k=1}^{\ell(\nu)} {\sf rev}^{\D_\0,*}_k(\theta_{(h_k')}^{s_k'})\\
=&\int_{[\M_\Gamma(\D_0|\Y|\D_\0)]^{\vir_\pi}} \prod_{i=1}^m\overline\psi_i^{k_i} {\sf ev}_i^*(\gamma_{(\bar h_i)}^{l_i}) \cup \prod_{j=1}^{\ell(\mu)} {\sf rev}^{\D_0,*}_j(\theta_{(h_j)}^{s_j})
\cup \prod_{k=1}^{\ell(\nu)} {\sf rev}^{\D_\0,*}_k(\theta_{(h_k')}^{s_k'})\\
&\qq\qq\qq\qq\qq\qq\qq\qq\qq \cup c_{\mathrm{top}}(\mathbb E\boxtimes T\D)\cup c_{\mathrm{top}}(E_{\ugamma})
\end{align*}
Denote the integrand
\[
\prod_{i=1}^m
{\sf ev}_i^*(\gamma_{(\bar h_i)}^{l_i})
\cup
\prod_{j=1}^{\ell(\mu)} {\sf rev}^{\D_0,*}_j(\theta_{(h_j)}^{s_j})
\cup
\prod_{k=1}^{\ell(\nu)} {\sf rev}^{\D_\0,*}_k(\theta_{(h_k')}^{s_k'})
\]
by $\Xi$. As $\gamma_{\bar h_i}^{l_i}$ belongs to the basis \eqref{eq basis-of-Y}, so it is of the form $\theta_{(h)}^l$ or $\theta_{(h)}^l\cdot[\D_0(h)]$. Let $\Xi^F$ denote those possible factor $[\D_0(h_i)]$ coming from $\gamma_{\bar h_i}^{l_i}$, and $\Xi^\D$ denote the rest part. So all classes in $\Xi^\D$ are pullback classes from $H^*_{\text{CR}}(\D)$. Then the invariant \eqref{E notation-for-rel-inv-DYD} is
\begin{align}\label{eq cput-fiber-inv}
&\int_{[\M_\Gamma(\D_0|\Y|\D_\0)]^\vir}
\prod_{i=1}^m\overline\psi_i^{k_i}\wedge \Xi\\
=&\int_{[\M_\Gamma(\D_0|\Y|\D_\0)]^{\vir_\pi}}
\prod_{i=1}^m\overline\psi_i^{k_i}
\wedge \Xi \wedge
c_{\mathrm{top}}(\mathbb E\boxtimes T\D)\cup c_{\mathrm{top}}(E_{\ugamma})\nonumber\\
=&\sum_q \int_{\D_\ugamma'}\Big\{ \Xi^\D \cup c_{\mathrm{top}}(E_{\ugamma})\cup t_q
\int_{[\M_\bgamma([0\rtimes K_\ugamma]|[\P^1\rtimes K_\ugamma]|[\0\rtimes K_\ugamma])]^{\mathrm{vir}}}
\Xi^{F}\cup \prod_{i=1}^n\overline\psi_i^{k_i}\cup h_q\Big\}
\nonumber
\end{align}
The Hodge integrals in the relative Gromov--Witten invariants of $([0\rtimes K_\ugamma]|[\P^1\rtimes K_\ugamma]|[\0\rtimes K_\ugamma])$ can be computed via virtual localizations, and the computation of the double ramification cycles on the moduli spaces of admissible covers of Tseng and You \cite{Tseng-You2016b}. The Hodge integrals reduces to the Hodge integrals over $\M(BK_\ugamma)$, which can be removed by the orbifold quantum Riemann--Roch of Tseng \cite{Tseng2010}. Finally the descendent integrations over the moduli spaces of stable curves are determined by Witten's conjecture \cite{Witten1990}, equivalently Kontsevich's theorem \cite{Kontsevich1992}.

\begin{remark}
Similar analysis applies to fiber class relative invariants of $(\Y|\D_0)$ and $(\Y|\D_\0)$, and an analogue of \eqref{eq cput-fiber-inv} holds for fiber class relative invariants of $(\Y|\D_0)$ and $(\Y|\D_\0)$. Therefore, every fiber class relative invariant of $(\Y|\D_0)$ or $(\Y|\D_\0)$ reduces respectively to relative invariants of $([\P^1\rtimes G]|[0\rtimes G])$ or $([\P^1\rtimes G]|[\0\rtimes G])$, where $G$ is a finite group determined by the invariant and acts on $\P^1=\P(\cplane\oplus\cplane)$ by acting on the first $\cplane$ linearly and on the second $\cplane$ trivially. All relative invariants of such $([\P^1\rtimes G]|[0\rtimes G])$ and $([\P^1\rtimes G]|[\0\rtimes G])$ were determined by Tseng and You \cite{Tseng-You2016b}.
\end{remark}

\section{Relative orbifold Gromov--Witten theory of weighted projectification}
\label{sec rel-GW-of-weight--proj}

In this section we determine relative orbifold Gromov--Witten theory of $(\oE_\wa|\PE_\wa)$. As in \S \ref{subsec notation-for-rel-inv-DYD}, \eqref{E notation-for-rel-inv-DYD} consider a relative invariant of $(\oE_\wa|\PE_\wa)$
\begin{align}\label{eq a-rel-inv-of-oE-PE}
\bl \prod_{i=1}^m \tau_{k_i} \gamma_{(\bar h_i)}^{l_i}
\bb \mu\br^{(\oE_\wa|\PE_\wa)}_\Gamma:=&\frac{1}{|\mbox{Aut}(\mu)|}
\int_{[\M_\Gamma(\oE_\wa|\PE_\wa)]^\vir}
\prod_{i=1}^m\overline\psi_i^{k_i}
{\sf ev}_i^*(\gamma_{(\bar h_i)}^{l_i})
\cup\prod_{j=1}^{\ell(\mu)} {\sf rev}^*_j(\theta_{(h_j)}^{r_j})
\end{align}
with topological data  $\Gamma=(g,\beta,(\bar\h),\vec\mu)$, where $\vec\mu=\left((\mu_1,(h_1)),\ldots,(\mu_{\ell(\mu)}, (h_{\ell(\mu)}))\right)$, and $\sum_j \mu_j=\int_\beta^{\mathrm{orb}} [\PE_\wa]\geq 0$ is the sum of orbifold contact orders along $\PE_\wa$ of the relative maps. The absolute insertions $\gamma_{(\bar h_i)}^{l_i}$ belong to the basis \eqref{eq basis-of-oE} of $H^*_{\text{CR}}(\oE_\wa)$, and the relative insertions $\theta_{(h_j)}^{r_j}$ belong to the basis \eqref{eq basis-of-PE} of $H^*_{\text{CR}}(\PE_\wa)$. As in \S \ref{subsec notation-for-rel-inv-DYD}, \eqref{E notation-for-rel-inv-DYD} we denote by
\[
\varpi=(\tau_{k_1} \gamma_{(\bar h_1)}^{l_1},\ldots,
\tau_{k_m} \gamma_{(\bar h_m)}^{l_m})
\]
the absolute insertions and set $\|\varpi\|=m$ to be the number of insertions in $\varpi$.

\subsection{Localization}\label{subsec localization}

There is a fiberwise $\cplane^*$-action on $\oE_\wa$ coming from the fiberwise $\cplane^*$-dilation on $\E$. We next apply the relative virtual localization with respect to this $\cplane^*$-action to compute the relative invariant \eqref{eq a-rel-inv-of-oE-PE}. The fixed loci in $\oE_\wa$ consist of the zero section $\os$, and the infinity section $\PE_\wa$. The fixed lines connecting them are lines in the fiber of $\oE_\wa\rto\os$ that connect a point in $\os$ and points in $\PE_\wa$, which correspond to lines in $\msf P_{\wa,1}$ connecting $[0,\ldots,0,1]$ with $[z_1,\ldots,z_n,0]$.

Consider the relative invariant \eqref{eq a-rel-inv-of-oE-PE}. Denote the moduli space simply by $\M_\Gamma$, with $\Gamma=(g,\beta,(\bar\h),\mu)$ denoting the topological data. As in \S \ref{subsec 3.3}, stable maps in $\M_\Gamma$ consists of two types, those mapped to the unexpanded target $(\oE_\wa|\PE_\wa)$, and those mapped to an expanded target $(\oE_\wa[l]|\PE_\wa)$. Here $\oE_\wa[l]$ is obtained in a parallel way of \eqref{E expanded-Y} as follows.

Denote the normal line bundle of $\PE_\wa$ in the weight-$\wa$ projectification $\oE_\wa$ of $\E$ by $\L$ (which is $\mc O_{\PE_\wa}(1)$, the dual line bundle of $\mc O_{\PE_\wa}(-1)$). Then we have the projectification $\Y=\P(\L\oplus\mc O_{\PE_\wa})$ (cf. \S \ref{subsec porjfi-line--bundle}). It has the zero section $\D_0$ and the infinity section $\D_\0$. Both are isomorphic to $\PE_\wa$. Take $l$ copies of $\Y$. Denote the zero section and infinity section of the $i$-th copy of $\Y$ by $\D_{0,i}$ and $\D_{\0,i}, 1\leq i\leq l$. We first glue the $l$ copies of $\Y$ together to get $\Y[l]$ via identifying $\D_{0,i}$ with $\D_{\0,i+1}$ for $1\leq i\leq l-1$. Then $\oE_\wa[l]$ is obtained by gluing $\oE_\wa$ with $\Y[l]$ by identifying $\PE_\wa$ with $\D_{\0,1}\in\Y[l]$. We also denote the $\PE_\wa$ in $\oE_\wa$ by $\D_{0,0}$. Then
\[
\mbox{Sing}(\oE_\wa[l])= \bigsqcup_{i=0}^{l-1}\D_{0,i} =\bigsqcup_{i=1}^l \D_{\0,i}.
\]
As for $\Y[l_0,l_\0]$, the $\oE_\wa$ in $\oE_\wa[l]$ is called the {\em root}, and the rest $\Y[l]$ is called the {\em rubber}.

Therefore there are two types of fixed loci of the induced $\cplane^*$-action on $\M_\Gamma$. A component of the fixed loci consisting of stable maps with target $\oE_\wa$ is call a {\bf simple} fixed locus. Otherwise, it is called a {\bf composite} fixed locus. We denote the simple fixed locus by $\M_\Gamma^{\simple}$. Denote the virtual normal bundle of $\M_\Gamma^\simple$ in $\M_\Gamma$ by $\cN_\Gamma$.

Every element of a composite fixed locus is of the form $\f\co \C'\cup\C''\rto \oE_\wa[l]~(l\geq 1)$\footnote{Here and in the following for simplicity we omit the refinement of domain curve by open covers of object spaces.}, such that the restrictions $\f'\co \C'\rto\oE_\wa$ and $\f''\co \C''\rto\Y[l]$ agree over the nodal points $\{\msf n_1,\cdots,\msf n_k\}=\C'\cap\C''$.
Suppose $\msf n_i$ is mapped into $\PE(h'_i)$, and the contact order of $\f'$ at $\msf n_i$, i.e. at $\D_{\0,1}(=\oE_\wa\cap \Y[l])$ in $\Y[l]$, is $\eta_i$ for $1\leq i\leq k$. Let $\Gamma'$ be the topological data corresponding to $\f'$ and $\Gamma''$ the topological data corresponding to $\f''$. (Here $\Gamma''$ denote the genus, absolute marked points, degree and contact orders of relative marked points relative to both $\D_{0,l}$ in $\Y[l]$ and contact orders at those nodes $\msf n_i, 1\leq i \leq k$.) Any two of $\{\Gamma,\Gamma',\Gamma''\}$ determine the third, and $\Gamma'$ gives us a simple fixed locus $\M_{\Gamma'}^\simple$.

Denote by $\M_{\Gamma''}^\sim$ the moduli space of relative stable maps to the rubber (see for example \cite{Graber-Vakil2005,Maulik-Pandharipande2006}). Then the composite fixed locus $\F_{\Gamma',\Gamma''}$ corresponding to a given $\Gamma'$ and $\Gamma''$ is canonically isomorphic to the quotient of the moduli space
\[
\M_{\Gamma',\Gamma''}:= \M_{\Gamma'}^\simple\times_{(\I\PE)^\ell}\M_{\Gamma''}^\sim
\]
by the finite group $\aut(\vec\eta)$, which consists of permutations of $\{1,\ldots,k\}$ preserving
\[
\vec\eta=\left((\eta_1,(h'_1)),\ldots, (\eta_k,(h'_k))\right).
\]
Denote the quotient map by $gl$:
\[
gl\co \M_{\Gamma',\Gamma''}\rto \F_{\Gamma',\Gamma''}.
\]
Set
\[
[\M_{\Gamma',\Gamma''}]^\vir :=\Delta^!([\M_{\Gamma'}^\simple]^\vir\times [\M_{\Gamma',\Gamma''}^\sim]^\vir)
\]
where
\[
\Delta\co (\I\PE_\wa)^k\rto (\I\PE_\wa)^k\times (\I\PE_\wa)^k
\]
is the diagonal map. Then we have
\[
[\F_{\Gamma',\Gamma''}]^\vir= \frac{1}{|\aut(\vec\eta)|}gl_\ast[\M^{\Gamma',\Gamma''}]^\vir.
\]

The virtual normal bundle of composite locus $\F_{\Gamma',\Gamma''}$ consists of two parts. The first part is the virtual normal bundle $\cN_{\Gamma'}$ of $\M_{\Gamma'}^\simple$ in $\M_{\Gamma'}$. The second part is a line bundle $\mc L$ corresponding to the deformation (i.e. smoothing) of the singularity $D_{\0,1}(=\oE_\wa\cap\Y[l])$ in $\Y[l]$. The fiber of this line bundle over a point in the fixed locus is canonically isomorphic to $H^0(\PE_\wa,\N_{\PE_\wa|\oE}\otimes\N_{\D_{\0,1}|\Y[l]})$. The line bundle
$\N_{\PE_\wa|\oE_\wa}\otimes\N_{\D_{\0,1}|\Y[l]}=\L\otimes \L^*=\mc O_{\PE_\wa}$
is trivial over $\PE_\wa$, so its space of global sections is one-dimensional, and we can canonically identify this space of sections with the fiber of the line bundle at a generic point $pt$ of $\PE_\wa$. Thus we can write the bundle $\mc L$ as a tensor product of bundles pulled back from the two factors separately. The one coming from $\M_{\Gamma'}$ is trivial, since it is globally identified with $H^0(pt,\N_{\PE_\wa|\oE_\wa}\big|_{pt})$, but it has a nontrivial torus action; we denote this weight by $t$. The line bundle coming from $\M_{\Gamma''}^\sim$ is a nontrivial line bundle, which has fiber $H^0(pt,\N_{\D_{\0,1}|\Y[l]}\big|_{pt})$, but has trivial torus action. We denote its first Chern class by $\Psi_\0$\footnote{In some literatures, $-\Psi_\0$ is referred as ``target psi class''. $-\Psi_\0$ correspond to the $\psi$ in \cite[\S 2.5 and \S3.3]{Graber-Vakil2005} and the $\Psi_0$ in \cite[\S 1.5.2]{Maulik-Pandharipande2006}. There is also another one $\Psi_0$ out of $\N_{\D_{0,l}|\Y[l]}|_{pt}$, corresponding to the $\Psi_\0$ in \cite[\S 1.5.2]{Maulik-Pandharipande2006}.}.

The relative virtual localization for relative orbifold Gromov--Witten invariants reads
\begin{align}\label{eq vir-local-formula}
[\M_\Gamma]^\vir=\frac{[\M_\Gamma^\simple]^\vir}{e(\cN_\Gamma)}+
\sum_{\substack{\M_{\Gamma',\Gamma''} \text{ composite}}}
\frac{(\prod_i \eta_i) gl_*[\M_{\Gamma',\Gamma''}]^\vir}
{|\aut(\vec\eta)|e(\cN_{\Gamma'})(t+\Psi_\0)}.
\end{align}
Here $\prod_i \eta_i$ is the ``mapping degree'' of the gluing maps (cf. \cite[\S 5.3.2]{Chen-Li-Sun-Zhao2011}), and $t$ is the equivariant weight of the $\cplane^*$--action.

We next describe explicitly the fixed loci of $\M_\Gamma$. We first consider the simple fixed locus $\M_\Gamma^\simple$. A map $\f\co (\C,\x,\y)\rto(\oE_\wa|\PE_\wa)$ in the simple fixed locus must have the following form.
\begin{itemize}
\item[(1)]
       $(\C,\x,\y)$ with absolute marked points $\x=(\x_1,\ldots,\x_m)$ and relative marked points $\y=(\y_1,\ldots,\y_{\ell(\mu)})$ is of the form
       \[
       \C=\C_0\cup\C_1\cup\ldots\cup \C_{\ell(\mu)},
       \]
       where $\C_0\cap \C_i=\{\msf n_i\}$ is a nodal point, called a {\em distinguished nodal point}, and $\C_i\cap \C_j=\varnothing$ for
       $1\leq i<j\leq \ell(\mu)$. Moreover $\x_i\in\C_0, 1\leq i\leq m$,
       and $\y_j\in\C_j,1\leq j\leq \ell(\mu)$.
\item[(2)]
       $\C_0$ is a genus $g$ pre-stable curves with $m$ marked points $\x$ and $\ell(\mu)$ marked points $\msf n=(\msf n_1,\cdots,\msf n_{\ell(\mu)})$ corresponding to the $\ell(\mu)$ distinguished nodes.
\item[(3)]
       For $1\leq i\leq \ell(\mu)$, each $\C_i$ is an (orbifold) Riemann sphere with a marked point $\y_i$ and a marked point $\msf n_i$ corresponding to the $i$-th distinguished nodal point.
\item[(4)]
      $\f\co (\C_0,\x\sqcup\msf n)\rto\os$ is a genus $g$ degree $\pi_*(\beta)$ stable maps to $\os$, and belongs to $\M_{g,\pi_*(\beta),\pi_t(\bar\h\sqcup\h)}(\os)$.
\item[(5)]
      For $1\leq i\leq \ell(\mu)$, $\f\co (\C_i,\y_i,\msf n_i)\rto\oE_\wa$ is a total
      ramified covering of a line in the fiber of $\oE_\wa$ that connects a point in $\PE_\wa(h_i)$ and a point in $\os(\pi_t(h_i\inv))$, the degree is determined by the contact order at $\PE_\wa$. Hence it is in the simple fixed loci of the moduli space
      \[
      \M_{0,\mu_i[F],\pi_t(h_i\inv), (\mu_i,(h_i))}(\oE_\wa|\PE_\wa),
      \]
      the moduli space of fiber class $\mu_i[F]$ stable maps from orbifold Riemann spheres with exactly one absolute marked point mapped to $\oE_\wa(\pi_t(h_i\inv))\supseteq\os(\pi_t(h_i\inv))$ and one relative marked point mapped to $\PE_\wa(h_i)$ with contact order $\mu_i$. For simplicity, we denote this simple fixed locus by $\M_{\mu_i,(h_i)}^\simple$. Denote the disconnected union of these $\M_{0,\mu_i[F],\pi_t(h_i\inv), (\mu_i,(h_i))}(\oE_\wa|\PE_\wa)$ by $\M^\bullet_{\vec \mu}$, and the disconnected union of simple fixed locus of them by $\M^{\bullet,\simple}_{\vec \mu}$.
\end{itemize}

Therefore, the simple fixed locus $\M_\Gamma^\simple$ is obtained by gluing stable maps in $\M_{g,\pi_*(\beta),\pi_t(\bar\h\sqcup\h)}(\os)$, and stable maps in $\M^{\bullet,\simple}_{\vec \mu}$ along the absolute marked points of $\C_0$ and $\C_i$ corresponding to those distinguished nodal points. Moreover
\[
gl\co \M_{\vec \mu}^{\bullet,\simple}\times_{(\I\os)^{\ell(\mu)}}
\M_{g,\pi_*\beta,\pi_t(\bar\h\sqcup\h)}(\os)\rto \M_\Gamma^\simple
\]
is a degree $|\aut(\vec\mu)|$ cover. The fiber product is taken with respect to the evaluation maps at the marked points out of the $\ell(\mu)$ distinguished nodal points of $\C$.

The tangent space $T^1_{(\C,\x,\y,\f)}$ and the obstruction space $T^2_{(\C,\x,\y,\f)}$ at a point $\f\co(\C,\x,\y)\rto(\oE_\wa|\PE_\wa)$ in $\M_\Gamma^\simple$ fit in the following long exact sequence of $\cplane^*$--representations:
\begin{align*}
0&\rto \aut(\C,\x,\y) \rto \text{Def}(\f) \rto T^1_{(\C,\x,\y,\f)}\rto\\
&\rto \text{Def}(\C,\x,\y) \rto \text{Obs}(\f) \rto T^2_{(\C,\x,\y,\f)}\rto 0,
\end{align*}
where
\begin{enumerate}
\item[(a)]  $\aut(\C,\x,\y)=\text{Ext}^0(\Omega_\C(\sum_i\x_i+\sum_j\y_j),
       \mc O_\C)$ is the space of infinitesimal automorphism of the domain
       $(\C,\x,\y)$. We have
       \[
       \aut(\C,\x,\y)=\aut(\C_0,\x,\msf n)\oplus\bigoplus_{i=1}^{\ell(\mu)}\aut(\C_i,\y_i,\msf n_i),
       \]
\item[(b)] $\text{Def}(\C,\x,\y)=\text{Ext}^1(\Omega_\C(\sum_i \x_i+\sum_j\y_j),\mc O_\C)$ is the space of infinitesimal deformation of the domain $(\C,\x,\y)$. We have a short exact sequence of $\cplane^*$-representations:
\[
0 \rto \text{Def}(\C_0,\x,\msf n)\rto \text{Def}(\C,\x,\y)\rto
\bigoplus_{i=1}^{\ell(\mu)}T_{\msf n_i}\C_0\otimes T_{\msf n_i}\C_i\rto 0,
\]
\item[(c)] $\text{Def}(\f)=H^0(\C,\f^*(T\oE_\wa(-\PE_\wa)))$ is the space of infinitesimal deformation of the map $\f$, and
\item[(d)] $\text{Obs}(\f)=H^1(\C,\f^*(T\oE_\wa(-\PE_\wa)))$ is the space of obstruction to deforming $\f$.
\end{enumerate}

For $i=1,2,$ let $T^{i,f}$ and $T^{i,m}$ be the fixed and moving parts of $T^i|_{\M_\Gamma^\simple}$. Then
\[
T^1=T^{1,f}+T^{1,m},\qq T^2=T^{2,f}+T^{2,m}.
\]
The virtual normal bundle of $\M_\Gamma^\simple$ in $\M_\Gamma$ is
\[
\cN_\Gamma=T^{1,m}-T^{2,m}.
\]
Let
\[
B_1=\aut(\C,\x,\y),\, B_2=\text{Def}(\f),\, B_4=\text{Def}(\C,\x,\y),\, B_5= \text{Obs}(\f)
\]
and let $B^f_i$ and $B^m_i$ be the fixed and moving parts of $B_i$. Then
\[
\frac{1}{e_{\cplane^*}(\cN_\Gamma)}=\frac{e_{\cplane^*}(B^m_5) e_{\cplane^*}(B^m_1)} {e_{\cplane^*}(B^m_2)e_{\cplane^*}(B^m_4)}.
\]

On the other hand we have the following exact sequence
\begin{align}\label{eq exact-sequence-O}
0\rto \mc O_\C\rto\bigoplus_{0\leq i\leq \ell(\mu)} \mc O_{\C_i} \rto \bigoplus_{1\leq i\leq \ell(\mu)} \mc O_{\C_{\msf n_i}} \rto 0.
\end{align}
Denote by $V=\f^*(T\oE_\wa(-\PE_\wa))$. Then the exact sequence \eqref{eq exact-sequence-O} gives us the following exact sequence
\begin{align*}
0&\rto H^0(\C,V)\rto \bigoplus_{0\leq i\leq \ell(\mu)} H^0(\C_i,V) \rto \bigoplus_{1\leq i\leq \ell(\mu)} V_{\msf n_i}\\
&\rto H^1(\C,V)\rto \bigoplus_{0\leq i\leq \ell(\mu)} H^1(\C_i,V) \rto0.
\end{align*}

\n Then the virtual normal bundle $\cN_\Gamma$ of $\M_\Gamma^\simple$ in $\M_\Gamma$ consists of
\begin{itemize}
\item[(i)] the normal bundle $\cN_{\vec \mu}$ of $\M_{\vec \mu}^{\bullet,\simple}$ in $\M_{\vec \mu}^\bullet$;
\item[(ii)] the contribution from deforming maps into $\os$, i.e.  $H^0(\C_0,\f^*\E)- H^1(\C_0,\f^*\E)$;
\item[(iii)] the contribution from deforming the distinguished nodal points: $T_{\msf n_i}\C_0\otimes T_{\msf n_i}\C_i-\E|_{\f(\msf n_i)}, 1\leq i\leq\ell(\mu)$, (note that only when $\msf n_i$ is a smooth nodal point, i.e. is not an orbifold nodal point, we have the term $-\E|_{\f(\msf n_i)}$).
\end{itemize}

Denote the last two contribution by $\Theta_\Gamma$, it contains psi-class out of $T_{\msf n_i}\C_0$ and Chern class of $\E$. We could write it in the form
\[
\Theta_\Gamma=\sum_{d\geq 0} \sum_{j+k=d} \left(\Theta_{\Gamma,\os}^j \prod_{i=1}^{\ell(\mu)} \Theta_{\Gamma,i}^k\right)
\]
with $\Theta_{\Gamma,\os}^j$ living over $\M_{g,\pi_*(\beta),\pi_t(\bar\h\sqcup \h)}(\os)$ and $\Theta_{\Gamma,i}^k$ living over each $\M_{\mu_i,(h_i)}^{\simple}$, where the contribution of $H^0(\C_0,\f^*\E)- H^1(\C_0,\f^*\E)$ is contained in $\Theta_{\Gamma,\os}^j, j\geq 0$, the contribution of $-\E|_{\f(\msf n_i)}$ is contained in $\Theta_{\Gamma,i}^k, k\geq 0$, and the contribution of $T_{\msf n_i}\C_0\otimes T_{\msf n_i}\C_i$ splits into both $\Theta_{\Gamma,\os}^j, j\geq0$ and $\Theta_{\Gamma,i}^k, k\geq 0$.

Then for the invariant \eqref{eq a-rel-inv-of-oE-PE}, the contribution from the simple fixed loci is
\begin{align*}
&\int_{[\M_\Gamma^\simple]^\vir} \frac{ev_\x^*\varpi\cup ev_\y^*\mu}{e_{\cplane^*}(\cN_\Gamma)} =\frac{1}{|\aut(\vec\mu)|}\cdot
\\&\sum_{ \substack{ \delta=(\delta_{\pi_t(h_1)},\ldots,\delta_{\pi_t(h_{\ell(\mu)})})\\
\text{in the chosen basis } \\ \eqref{eq basis-of-S} \text{ of } H^*_{\mathrm{CR}}(\os)\\ d\geq 0,\,j+k=d}}
\left( \int_{[\M_{\vec \mu}^{\bullet,\simple}]^\vir}
\frac{ev_\y^*\mu\cup ev_{\msf n}^*\check\delta\cup \prod_{i=1}^{\ell(\mu)}\Theta_{\Gamma,i}^k}{e_{\cplane^*}(\cN_{\vec \mu})}
\cdot \int_{[\M_{g,\pi_*(\beta),\pi_t(\bar\h\sqcup \h)}(\os)]^{\vir}}
ev_\x^*\varpi\cup ev_{\msf n}^*\delta\cup \Theta_{\Gamma,\os}^j \right)
\end{align*}
where the sum is taken over all possible $\ell(\mu)$-tuple of the chosen basis $\sigma_\star$ of $H^*_{\mathrm{CR}}(\os)$ in \eqref{eq basis-of-S}. The integration
\[
\int_{[\M_{\vec \mu}^{\bullet,\simple}]^\vir} \frac{ev_\y^*\mu\cup ev_{\msf n}^*\check\delta\cup \prod_{i=1}^{\ell(\mu)} \Theta_{\Gamma,i}^k}{e_{\cplane^*}(\cN_{\vec \mu})}
=\prod_{i=1}^{\ell(\mu)} \int_{[\M_{\mu_i,(h_i)}^\simple]^\vir}
\frac{ev_{\msf y_i}^*(\theta_{(h_i)}^{s_i})\cup ev_{\msf n_i}^*(\check\delta_{\pi_t(h_i)})\cup\Theta_{\Gamma,i}^k}
{e_{\cplane^*}(\cN_{\mu_i})}
\]
is a product of fiber class (1+1)-point relative invariants of $(\oE_\wa|\PE_\wa)$ and was computed by Hu and the first two authors  \cite[\S 5]{Chen-Du-Hu2019} (see also Remark \ref{R 3.11}). See \eqref{eq 1+1-point-rel-inv} in \S \ref{sec rel-GW-of-weit-blp} for explicit expressions of (1+1)-point fiber class relative invariants of weighted projectification of orbifold vector bundles related to the infinity divisors. These invariants are related to (1+1)-point relative invariants of weighted projective spaces.

The invariant
\[
\int_{[\M_{g,\pi_*\beta,\pi_t(\bar\h\sqcup \h)}(\os)]^{\vir}} ev_\x^*\varpi\cup ev_{\msf n}^*\delta\cup\Theta_{\Gamma,\os}^j
\]
is a Hodge integral in the twisted Gromov--Witten invariant of $\os$ with twisting coming from the bundle $\E$, which by the orbifold quantum Riemann--Roch of Tseng \cite{Tseng2010} is determined by the orbifold Gromov--Witten theory of $\os$
and the total Chern class of $\E\rto \os$.

We next consider the composite fixed loci. Recall that a composite fixed locus is of the form
\[
\F_{\Gamma',\Gamma''} =gl(\M_{\Gamma'}^\simple\times_{(\I\PE_\wa)^{\ell(\vec\eta)}} \M_{\Gamma''}^\sim).
\]
with contribution being
\[
\frac{\prod_i \eta_i}{|\aut(\vec\eta)|}\cdot \frac{gl_*\Delta^!([\M_{\Gamma'}^\simple\times
\M_{\Gamma''}^\sim]^{\vir})}{e(\cN_{\Gamma'})(t+\Psi_\0)}.
\]
By the analysis for simple fixed locus, the contribution from $\M_{\Gamma'}^\simple$ reduces to Gromov--Witten theory of $\os$. Therefore, the contribution of $\F_{\Gamma',\Gamma''}$ reduces to rubber invariants corresponding to $\M_{\Gamma''}^\sim$, which are rubber invariants with $\Psi^k_\0$-integrals of
\[
(\D_0| \Y| \D_\0)=(\PE_{\wa,0}| \P(\L\oplus\mc O_{\PE_\wa})| \PE_{\wa,\0}).
\]
We denote these rubber invariants by
\begin{align}\label{eq notation-of-rubber-invariants}
&\bl\mu\bb\varpi\cdot \Psi_\0^k\bb\nu\br^{ {(\D_0\mid\Y\mid \D_\0)},\sim}_{g,\beta}\\
&:= \frac{1}{|\aut(\mu)||\aut(\nu)|}
\int_{[\M_{g,(\bar\h),\beta,\vec\mu,\vec\nu}^\sim (\D_0|\Y|\D_\0)]^{\vir}}
\msf{ev}^*\varpi\cup \msf{rev}^{\D_0,*}\mu\cup \msf{rev}^{\D_\0,*} \nu\cup\Psi_\0^k.
\nonumber
\end{align}
Note that here the $\varpi$ is different from the absolute insertions in \eqref{eq a-rel-inv-of-oE-PE}, as here $\varpi$ consists of cohomology classes in $H^*_{\text{CR}}(\Y)$, not $H^*_{\text{CR}}(\oE_\wa)$. We will determine these rubber invariants in the following subsections.

\subsection{Rubber calculus}\label{subsec rubber-calculus}

As above denote $\PE_\wa$ by $\D$. Then in $\Y=\P(\L\oplus\mc O_{\PE_\wa})=\P(\L\oplus\mc O_\D)$ we have $\D\cong \D_0\cong\D_\0$. The projection $\pi\co \Y\rto\D$ induces a topological fiber bundle $|\pi|\co |\Y|\rto|\D|$ over the coarse spaces, whose fiber is $\P^1$.

In this subsection we relate rubber invariants \eqref{eq notation-of-rubber-invariants} of $ {(\D_0|\Y| \D_\0)}$ to two kinds of relative invariants of $ {(\D_0|\Y| \D_\0)}$, which are
\begin{enumerate}
\item[(a)] fiber class invariants of the form
\begin{align}\label{E fiber-class-invaraint-3.3}
\bl\mu\bb \tau_1([\D_0]\cdot \theta)\cdot\varpi\bb \nu\br^{ (\D_0\mid\Y\mid \D_\0)}_{g,\beta}
\end{align}
with $\beta$ being a fiber class, $\theta\in H^*(\D)$, and
\item[(b)] general class invariants of the form
\begin{align}\label{E distinguished-type-II-3.3}
\bl\mu\bb \tau_0([\D_0]\cdot \theta)\cdot\varpi \bb \nu\br^{ (\D_0\mid\Y\mid \D_\0)}_{g,\beta}
\end{align}
with $\theta\in H^{>0}(\D)$.
\end{enumerate}
Here cohomology classes in $\varpi$ come from the chosen basis \eqref{eq basis-of-Y}, $[\D_0]$ is the Poincar\'e dual of $\D_0$ in $\Y$, and $\mu$ and $\nu$ denote the relative weights corresponding to $\D_0$ and $\D_\0$ with cohomological weights coming from the chosen basis \eqref{eq basis-of-D}. The invariants in \eqref{E distinguished-type-II-3.3} are orbifold case analogues of {\em Distinguished Type II invariants} in \cite{Maulik-Pandharipande2006}. Here we also call them Distinguished Type II invariants. In the following we omit the superscript $(\D_0|\Y|\D_\0)$ to simplify the notations.

The fiber class invariants in (a) have been determined in \S \ref{sec fiber-class-inv}. We will determine Distinguished Type II invariants in (b) in \S \ref{subsec dertermine-dis-type-ii}.

Since we will encounter disconnected rubber invariants, we will also consider disconnected distinguished Type II invariants. We use a superscript ``$\bullet$'' to decorate disconnected invariants. However, as noted in \cite{Maulik-Pandharipande2006}, there is no product rule for rubber invariants.

\subsubsection{Rigidification}

Given a (possibly disconnected) rubber invariant of $(\D_0|\Y|\D_\0)$
\[
\bl\mu\bb\theta\cdot\varpi \bb\nu\br^{\bullet,\sim}_{g,\beta}
\]
with $\theta\in H^*(\Y)$ being an insertion form the non-twisted sector $\Y$ itself, we denote by
\[
\M^{\bullet,\sim}_\Gamma :=\M^{\bullet,\sim}_{g,(\bar\h),\beta,\vec\mu,\vec\nu}(\D_0|\Y|\D_\0)
\]
the moduli space of stable maps to the rubber target with $\Gamma$ denote the topological data. We also denote by
\[
\M^\bullet_\Gamma :=\M^\bullet_{g,(\bar\h),\beta,\vec\mu,\vec\nu}(\D_0|\Y|\D_\0)
\]
the moduli space of stable maps to $\Y$ relative to both $\D_0$ and $\D_\0$ with the same topological data. Here $g$ is the arithmetic genus. Without loss of generality we assume $(h_1)=(1)$ in $(\bar\h)$ is the index of the non-twisted sector $\Y$ itself, which corresponds to $\theta$.

We have a canonical forgetful map
\begin{align}\label{Eq forgetful-map}
\epsilon\co \M_\Gamma^\bullet\rto\M_\Gamma^{\bullet,\sim},
\end{align}
which is $\cplane^*$-equivariant with respect to the canonical $\cplane^*$-action on $\M^\bullet_\Gamma$ induced from the fiber-wise $\cplane^*$-action on $\Y$ and the trivial $\cplane^*$-action on $\M^{\bullet,\sim}_\Gamma$.

\begin{lemma}\label{lem rigidification}
Let $q$ be an absolute marked point corresponding to $(h_1)$, and the corresponding evaluation map be
\[
ev_q\co \M_\Gamma^\bullet\rto\Y.
\]
Then
\begin{align}\label{E rigidification-1}
[\M^{\bullet,\sim}_\Gamma]^{\text{\em vir}}= \epsilon_*(ev_q^*[\D_0]\cap [\M_\Gamma^\bullet]^{\text{\em vir}})=\epsilon_*(ev_q^*[\D_\0]\cap [\M_\Gamma^\bullet]^{\text{\em vir}}).
\end{align}
Therefore
\[\begin{split}
\bl\mu\bb\theta\cdot\varpi\bb\nu\br^{\bullet,\sim}_{g,\beta}
&=\bl\mu\bb([\D_0]\cdot\theta)\cdot\varpi \bb\nu\br^{\bullet}_{g,\beta}
=\bl\mu\bb([\D_\0]\cdot\theta)\cdot\varpi \bb\nu\br^{\bullet}_{g,\beta}
\end{split}\]
\end{lemma}

\begin{proof}
The proof is similar to the proof of \cite[Lemma 2]{Maulik-Pandharipande2006}. We use the localization formula \eqref{eq vir-local-formula} for relative orbifold Gromov--Witten theory. We only prove the first equality in \eqref{E rigidification-1}. The proof of the second one is similarly.

A $\cplane^\ast$--fixed  stable relative map in $\M_\Gamma^\bullet$ is a union of:
\begin{enumerate}
\item[(i)] a nonrigid stable map to the degeneration of $\Y$ over $\D_0$,
\item[(ii)] a nonrigid stable map to the degeneration of $\Y$ over $\D_\0$,
\item[(iii)] a collection of $\cplane^\ast$-invariant, fiber class,
            rational Galois covers joining (i), (ii), i.e. they are maps from orbifold Riemann spheres with two orbifold points and are total ramified over $\D_0$ and $\D_\0$.
\end{enumerate}
On the fixed locus, the forgetful map $\epsilon$ simply contracts the intermediate rational curves (iii).

Obviously, simple fixed loci contributes nothing since there is at least one absolute marked point. If we have a proper degeneration on both sides of $\Y$, the (complex) virtual dimension of the $\cplane^\ast$-fixed  locus is $2$ less than the virtual dimension of $\M_\Gamma^\bullet$. However, the dimension of
\[
\epsilon_*(ev^*_p([\D_0])\cap[\M^\bullet_\Gamma]^{\text{vir}})
\]
is only 1 less than the virtual dimension of $\M^\bullet_\Gamma$. Therefore we only have to consider fixed loci whose target degenerates on only one side of $\Y$.

Since the absolute marked point $q$ is constrained by an insertion of $[\D_0]$, we need only consider degenerations along $\D_0$. Then we see that there is a unique composite $\cplane^\ast$-fixed  locus
\[
\F_{\Gamma',\Gamma}=gl( \M_\Gamma^{\bullet,\sim}
\times_{(\I\D)^{\ell(\nu)}}
\M_{\Gamma'}^{\text{simple}}).
\]
where $\M_{\Gamma'}^{\text{simple}}$ consists of curves (iii), hence has no absolute marked points, and the relative marked points relative to $\D_\0$ are constrained by $\nu$, the relative marked point relative to $\D_0$ are constrained by $\vec{\check\nu}$.

Then by \eqref{eq vir-local-formula} we get
\begin{align*}
&\epsilon_*(ev^*_p([\D_0])\cap[\M^\bullet_\Gamma]^{\text{vir}})\\
=&\frac{\prod_i \nu_i}{\aut(\vec\nu)} \Big((ev^*_p(c_1(\L))+t)\cap
\frac{gl_*[ \M_\Gamma^{\bullet,\sim}
\times_{(\I\D)^{\ell(\nu)}}
\M_{\Gamma'}^{\text{simple}}]^{\text{vir}}}
{e(\cN_{\Gamma'})(t+\Psi_\0)}\Big)
\\=&
\prod_i \nu_i\cdot
\sum_{
\substack{
\rho =(\rho_1,\ldots,\rho_{\ell(\nu)}) \text{ in \eqref{eq basis-of-D}}
\\ \text{constrained by } \vec\nu}}\left\{
\Big(\frac{(ev^*_p(c_1(\L))+t) ev_{\D_\0}^*\rho}{(t+\Psi_\0)}\cap
[\M_\Gamma^{\bullet,\sim}]^{\text{vir}}\Big)\cdot \Big(\frac{ev_{\D_0}^*\check\rho
}{e(\cN_{\Gamma'})}\cap [\M_{\Gamma'}^{\text{simple}}]^{\text{vir}}\Big)
\right\}
\end{align*}
where $ev^*_p(c_1(\L))+t$ is the restriction of the equivariant lifting of the class $[\D_0]$ to the zero section $\D_0$. The last term in previous equation corresponds to the following relative invariant of $(\D_0|\Y|\D_\0)$
\begin{align*}
&|\aut(\rho)| \cdot|\aut(\nu)|\cdot
\bl\check\rho\bb\varnothing\bb\nu\br_{0,|\nu|,[F]}
=\int_{[\M_{\Gamma'}]^\vir} ev_{\D_0}^*\check\rho\wedge ev_{\D_\0}^*\nu \\&
=\Big(\frac{ev_{\D_0}^*\check\rho\cup ev_{\D_0}^*\nu}{e(\cN_{\Gamma'})}\cap
[\M_{\Gamma'}^{\text{simple}}]^{\text{vir}}\Big)=
\left\{\begin{array}{ll}
\frac{1}{\prod_i \nu_i} &\text{if $\rho=\nu$;}\\
0                       &\text{if $\rho\neq\nu$.}
\end{array}\right.
\end{align*}
Therefore by dimension counting
\begin{align*}
\epsilon_*(ev^*_p([\D_0])\cap[\M^\bullet_\Gamma]^{\text{vir}})
=\Big(\frac{(ev^*_p(c_1(\L))+t)}{(t+\Psi_\0)}\cap
[\M_\Gamma^{\bullet,\sim}]^{\text{vir}}\Big)
= [\M_\Gamma^{\bullet,\sim}]^{\text{vir}}.
\end{align*}
This finishes the proof.
\end{proof}

\subsubsection{Dilaton and divisor equations}\label{subsubsec dilaton-divisor-eqn}

As the smooth case, the rubber invariants for orbifolds also satisfy the dilaton equation and divisor equation as in \cite{Maulik-Pandharipande2006}\footnote{As noticed by Tseng--You \cite{Tseng-You2016a}, there is a typo in \cite[\S 1.5.4]{Maulik-Pandharipande2006}: $\tau_1(\one)$ should not appear on the right side of the dilaton equation.}.
The dilaton equation is
\[\begin{split}
\bl \mu\bb \tau_1(\one)\cdot &\prod_{i=1}^m
\tau_{k_i}(\gamma_{l_i})\cdot\Psi_\0^k\bb \nu\br^{\bullet,\sim}_{g,\beta}
=(2g-2+m+\ell(\mu)+\ell(\nu))\bl \mu\bb\prod_{i=1}^{m}
\tau_{k_i}(\gamma_{l_i})\cdot\Psi_\0^k\bb \nu\br^{\bullet,\sim}_{g,\beta}.
\end{split}\]
The divisor equation for $H\in H^2(\D)$ is
\[
\begin{split}
\bl\mu\bb \tau_0(H)\cdot\prod_{i=1}^m\tau_{k_i}(\gamma_{l_i})\cdot\Psi_\0^k \bb\nu\br^{\bullet,\sim}_{g,\beta} =(\int_{\pi_*\beta}^{\mathrm{orb}}H) \bl \mu\bb \prod_{i=1}^m \tau_{k_i}(\gamma_{l_i})\cdot\Psi_\0^k\bb \nu\br^{\bullet,\sim}_{g,\beta}\\
+\sum_{i=1}^m\bl\mu\bb\ldots \tau_{k_i-1}(\gamma_{l_i}\cup_{\text{CR}} H)\ldots \Psi^k_\0\bb\nu\br^{\bullet,\sim}_{g,\beta}\\
-\sum_{j=1}^{\ell(\nu)}\bl \mu\bb\prod_{i=1}^m\tau_{k_i}(\gamma_{l_i})\cdot\Psi^{k-1}_\0\bb\{\ldots (\nu_j,\delta_{s_j}\cup_{\text{CR}} H)\ldots\}\br^{\bullet,\sim}_{g,\beta}\cdot\nu_j.
\end{split}
\]
Since here the marked points corresponding to the two insertions $\tau_1(\mathbf 1)$ and $\tau_0(H)$ are both smooth marked points, i.e. without orbifold structure, the standard cotangent line comparison method proves the dilaton and divisor equations above.

\subsubsection{Rubber calculus I: fiber class invariants}
\label{subsubsec rubber-cal-fiber}

We first consider a fiber class rubber invariant in \eqref{eq notation-of-rubber-invariants} with descendant insertion $\varpi$
\begin{align}\label{eq fiber-rubber}
\bl \mu\bb\varpi\cdot\Psi^k_\0\bb\nu\br^{\bullet,\sim}_{g,\beta},
\end{align}
i.e. $\beta$ is a multiple of the fiber class of $\Y\rto\D$.

Note that by stability, a contracted genus zero irreducible component of the domain curve must carry at least three absolute marked points, a contracted genus one irreducible component of the domain curve must carry at least one absolute marked point. A non-contracted component of the domain curve must carry at least two relative marked points, i.e. the intersection points with $\D_0$ and $\D_\0$. Finally, by target stability, not all components of domain curve can be genus 0 and fully ramified over $\D_0$ and $\D_\0$. Hence we conclude
\[
2g-2+m+\ell(\mu)+\ell(\nu)>0,
\]
where $m=\|\varpi\|$ is the length of $\varpi$. Therefore the fiber class rubber invariants \eqref{eq fiber-rubber} is determined by
\begin{align}\label{eq fiber-rubber-1}
\bl \mu\bb\tau_1(\one)\cdot\varpi\cdot \Psi^k_\0\bb\nu\br^{\bullet,\sim}_{g,\beta}
\end{align}
and the dilaton equation in \S \ref{subsubsec dilaton-divisor-eqn}. Let $q$ denote the smooth absolute marked point that carries the insertion $\tau_1(\one)$ in the rubber invariants \eqref{eq fiber-rubber-1}. Denote the twisted sectors of the $1+||\varpi||$ marked points by $(\bar\h)$ with $(h_1)=(1)$ corresponding to $q$.

Denote the topological data of the moduli space of the invariant \eqref{eq fiber-rubber-1} by $\Gamma$ and the corresponding rubber moduli space by $\M_\Gamma^{\bullet,\sim}(\D_0|\Y|\D_\0)$. Consider a splitting $\Gamma_1\wedge_\D\Gamma_2$ of $\Gamma$ with
\begin{align}\label{E spliting-Gamma-fiber-class}
\Gamma_1=(g_1,\beta_1,(\bar\h_1),\vec\mu,\vec\eta),\qq\mbox{and}\qq \Gamma_2=(g_2,\beta_2,(\bar\h_2),\check{\vec\eta},\vec\nu),
\end{align}
satisfying
\[
(h_1)=(1)\in(\bar\h_1),
\]
i.e. the marked point corresponding to $\tau(\one)$ is distributed to the $\Gamma_1$ side. Then we could glue two stable maps in the rubber moduli spaces
\[
\M_{\Gamma_1}^{\bullet,\sim}(\D_0|\Y|\D_\0) ,\qq\mbox{and}\qq
\M_{\Gamma_2}^{\bullet,\sim}(\D_0|\Y|\D_\0)
\]
together to get a stable map in the rubber moduli space $\M_\Gamma^{\bullet,\sim}(\D_0|\Y|\D_\0)$. In this way we get a boundary component of $\M_\Gamma^{\bullet,\sim}(\D_0|\Y|\D_\0)$, that is we have the following map
\[
gl\co \M_{\Gamma_1}^{\bullet,\sim}(\D_0|\Y|\D_\0)
\times_{(\I\D)^{\ell(\vec\eta)}}
\M_{\Gamma_2}^{\bullet,\sim}(\D_0|\Y|\D_\0)
\rto   \M_\Gamma^{\bullet,\sim}(\D_0|\Y|\D_\0),
\]
which is a $|\aut(\vec\eta)|$ cover to its image, i.e. a boundary component of $\M_\Gamma^{\bullet,\sim}(\D_0|\Y|\D_\0)$ corresponding to the splitting \eqref{E spliting-Gamma-fiber-class}. Denote the image by $\scr B_{\Gamma_1,\Gamma_2}$, and the normal bundle of this boundary component by $\cN_{\Gamma_1,\Gamma_2}$. There is a line bundle $\cN$ (see for example \cite{Katz2007}) over $\M_\Gamma^{\bullet,\sim}$ whose zero set consists of disjoint union of all these boundary components $\scr B_{\Gamma_1,\Gamma_2}$, and its restriction on $\scr B_{\Gamma_1,\Gamma_2}$ is $\cN_{\Gamma_1,\Gamma_2}^{\otimes \prod_i\eta_i}$. The first Chern class of $\cN$ is
\[
c_1(\cN)=-\Psi_\0+ev_q^*(c_1(\L)).
\]
Therefore for the invariant \eqref{eq fiber-rubber-1} we have
\begin{align}\label{E fiber-class-rubber-spliting-eqn}
&\bl \mu\bb \tau_1(\one)\cdot\varpi\cdot \Psi_\0^k\bb\nu\br^{\bullet,\sim}_{g,\beta}
\nonumber\\
&=-\bl \mu\bb \tau_1(\one)\cdot\varpi\cdot \Psi_\0^{k-1} \cdot
c_1(\cN) \bb\nu\br^{\bullet,\sim}_{g,\beta} +\bl \mu\bb \tau_1(c_1(\L))\cdot\varpi\cdot \Psi_\0^{k-1}
\bb\nu\br^{\bullet,\sim}_{g,\beta}\nonumber\\
&=-\sum \bl\mu \bb \tau_1(\one)\cdot \varpi_1\bb \eta \br^{\bullet,\sim}_{g_1,\beta_1} \cdot\fk z(\eta)\cdot \bl \check\eta \bb  \varpi_2 \cdot \Psi_\0^{k-1} \bb \nu \br^{\bullet,\sim}_{g_2,\beta_2}
+\bl \mu \bb \tau_1(c_1(\L))\cdot\varpi \cdot \Psi_\0^{k-1} \bb
\nu \br^{\bullet,\sim}_{g,\beta},
\end{align}
where the sum is taken over all splittings $\Gamma_1\wedge_\D\Gamma_2$ of $\Gamma$ of the form in \eqref{E spliting-Gamma-fiber-class}, and all intermediate cohomology weighted partitions $\eta$, and $(\varpi_1,\varpi_2)$ is a distribution of $\varpi$ according the splittings of absolute marked points in $\Gamma_1$ and $\Gamma_2$.

The firt term $\bl \mu \bb \tau_1(\one)\cdot \varpi_1\bb  \eta \br^{\bullet,\sim}_{g_1,\beta_1}$ in the summation in \eqref{E fiber-class-rubber-spliting-eqn} can be expressed as a fiber class invariants of $(\D_0|\Y|\D_\0)$ by Lemma \ref{lem rigidification}:
\[
\bl \mu \bb \tau_1(\one)\cdot \varpi_1\bb  \eta \br^{\bullet,\sim}_{g_1,\beta_1}
=\bl \mu \bb \tau_1([\D_0])\cdot \varpi_1 \bb \eta \br^{\bullet}_{g_1,\beta_1}.
\]
Therefore we have reduced the original fiber class rubber invariant \eqref{eq fiber-rubber} to rubber invariants of the same type with {\em strictly fewer} $\Psi_\0$ insertions and fiber class invariants of $(\D_0|\Y|\D_\0)$. Repeating this cycle we reduce the original fiber class rubber invariant \eqref{eq fiber-rubber} to fiber class rubber invariants without $\Psi_\0$ insertions
\begin{align}\label{eq fiber-rubber-without-Psi}
\bl \mu\bb\tau_1(c_1(\L)^k)\cdot
\varpi'\bb\eta\br^{\bullet,\sim}_{g,\beta},\qq k\geq 0,
\end{align}
and fiber class invariants of $(\D_0|\Y|\D_\0)$. By Lemma \ref{lem rigidification}, the rubber invariants in \eqref{eq fiber-rubber-without-Psi} are determined by fiber class invariants of $(\D_0|\Y|\D_\0)$. Therefore the fiber class rubber invariant \eqref{eq fiber-rubber} is determined by fiber class invariants of $(\D_0|\Y|\D_\0)$ of the forms
\begin{align*}
\bl \mu\bb\tau_1([\D_0]\cdot c_1(\L)^k)\cdot
\varpi'\bb\eta\br^{\bullet}_{g,\beta},\qq k\geq 0,
\end{align*}
i.e. of the form of \eqref{E fiber-class-invaraint-3.3}. We have computed all fiber class invariants of $(\D_0|\Y|\D_\0)$ in \S \ref{sec fiber-class-inv}, so all fiber class rubber invariants are determined.

\subsubsection{Rubber calculus II: non-fiber class invariants}
\label{subsubsec rubber-cal-non-fiber}

Now suppose that the homology class $\beta$ in \eqref{eq notation-of-rubber-invariants} is not a fiber class, i.e. $\pi_*(\beta)\neq 0$. Consider a non-fiber class rubber invariants:
\begin{align}\label{eq non-fiber-rubber}
\bl \mu\bb \prod_{i=1}^m \tau_{k_i}(\gamma_{l_i})\cdot\Psi^k_\0\bb\nu\br^{\bullet,\sim}_{g,\beta}.
\end{align}
There exists an $H\in H^2(\D)$ such that
\[
\int_{\pi_*\beta}^{\mathrm{orb}}H>0.
\]
Now consider the rubber invariant
\begin{align}\label{eq non-fiber-rubber-1}
\bl \mu\bb \tau_0(H) \cdot\prod_{i=1}^m \tau_{k_i}(\gamma_{l_i})\cdot\Psi^k_\0\bb\nu\br^{\bullet,\sim}_{g,\beta}.
\end{align}
By the divisor equation in \S \ref{subsubsec dilaton-divisor-eqn}, modulo rubber invariants with strictly fewer $\Psi_\0$ insertions, the rubber invariant \eqref{eq non-fiber-rubber} is determined by the rubber invariant \eqref{eq non-fiber-rubber-1} and
\begin{align}\label{eq non-fiber-rubber-2}
\bl \mu\bb \ldots \tau_{k_i-1}(H\cup_{\mathrm{CR}}\gamma_{l_i})\ldots\Psi^k_\0
\bb\nu\br^{\bullet,\sim}_{g,\beta},\qq i=1,\cdots m.
\end{align}
For invariants in \eqref{eq non-fiber-rubber-2} consider rubber invariants of the form
\begin{align}\label{eq non-fiber-rubber-3}
\bl \mu\bb  \tau_0(H)\ldots \tau_{k_i-1}(H\cup_{\mathrm{CR}}\gamma_{l_i})\ldots\Psi^k_\0\bb\nu
\br^{\bullet,\sim}_{g,\beta},\qq i=1,\cdots m,
\end{align}
by adding an insertion $\tau_0(H)$. Again by the divisor equation in \S \ref{subsubsec dilaton-divisor-eqn}, modulo rubber invariants with strictly fewer $\Psi_\0$ insertion or strictly fewer $\overline\psi_i$ insertions, the rubber invariants in \eqref{eq non-fiber-rubber-2} are determined by rubber invariants in \eqref{eq non-fiber-rubber-3}.

Finally, combining \eqref{eq non-fiber-rubber-1} and \eqref{eq non-fiber-rubber-3} we see that, modulo rubber invariants with strictly
fewer $\Psi_\0$ insertions, the rubber invariants \eqref{eq non-fiber-rubber} is determined by rubber invariants of the form
\begin{align}\label{eq non-fiber-rubber-4}
\bl \mu\bb  \tau_0(H)\cdot\prod_{i=1}^m
\tau_{k_i-n_i}(H^{n_i}\cup_{\mathrm{CR}}\gamma_{l_i})\cdot\Psi^k_\0
\bb\nu\br^{\bullet,\sim}_{g,\beta},
\end{align}
where $1\leq n_i\leq k_i, 1\leq i\leq m$. We then apply the boundary relation in previous subsection \S \ref{subsubsec rubber-cal-fiber} to \eqref{eq non-fiber-rubber-4}.
By repeating this cycle we could reduce the rubber invariant
\eqref{eq non-fiber-rubber} to
\begin{align}\label{eq non-fiber-rubber-inv-outof-rubber}
\bl \mu\bb\tau_0(H\cdot c_1(\L)^k)\cdot
\varpi'\bb\eta\br^{\bullet,\sim}_{g,\beta},\qq k\geq 0.
\end{align}
Finally, we apply Lemma \ref{lem rigidification} to \eqref{eq non-fiber-rubber-inv-outof-rubber}. So we can express the original rubber invariant \eqref{eq non-fiber-rubber} in terms of distinguished type II invariants of $(\D_0|\Y|\D_\0)$ of the forms
\begin{align*}
\bl \mu\bb\tau_0([\D_0]\cdot H\cdot c_1(\L)^k)\cdot
\varpi'\bb\eta\br^{\bullet}_{g,\beta},\qq k\geq 0,
\end{align*}
i.e. distinguished type II invariants in \eqref{E distinguished-type-II-3.3}. We will determines these invariants in next subsection
 \S \ref{subsec dertermine-dis-type-ii}.

\subsection{Distinguished type II invariants}
\label{subsec dertermine-dis-type-ii}

In this subsection we determine all distinguished Type II invariants via an induction algorithm with initial datum being fiber class invariants determined in \S \ref{sec fiber-class-inv}. We give a partial order ``$\stackrel{\circ}{\prec}$'' over all distinguished Type II invariants. The partial order is different from the partial order used in \cite{Maulik-Pandharipande2006}, since here in general we can not compare cohomology weights coming from different twisted sectors. The partial order we give here is a modification of the partial order given in \cite{Chen-Du-Hu2019}. The partial order given in \cite{Chen-Du-Hu2019} is more geometric, and follows from the degeneration formula in \cite{Chen-Li-Sun-Zhao2011,Abramovich-Fantechi2016} directly.

\subsubsection{Partial order}\label{subsubsec partial order}

Consider a distinguished type II invariant
\begin{align}\label{eq type-ii-inv-partial-order}
\bl\mu\bb \tau_0([\D_0]\cdot \theta)\cdot\varpi\bb \nu\br_{g,\beta},
\end{align}
where $\theta\in H^{>0}(\D)$. We also assume that $\mu$ and $\nu$ have cohomological weights from the chosen basis \eqref{eq basis-of-D} and $\varpi$ has cohomology weights from the chosen basis \eqref{eq basis-of-Y}. We have a degeneration of $(\D_0|\Y|\D_\0)$ along $\D_\0$
\begin{align}\label{eq degenerate Y}
(\D_0|\Y|\D_\0)\xrightarrow{\text{degenerate}}\Y_1\wedge_\D \Y_2=(\D_0| \Y|\D_\0)\wedge_\D(\D_0| \Y|\D_\0)
\end{align}
where we glue $\D_\0$ in $\Y_1$ with $\D_0$ in $\Y_2$ via $\D_0\cong\D_\0\cong \D$.

For the invariant $\bl\mu \bb \tau_0([\D_0]\cdot \theta)\cdot\varpi \bb \nu \br_{g,\beta}$ we distribute $\tau_0([\D_0]\cdot \theta)$ and insertions in $\varpi$ with cohomology classes being divisible by one of $\{[\D(h)_0]|(h)\in\T_\D\}$ to $\Y_1$. Then by the degeneration formula (cf. \cite{Chen-Li-Sun-Zhao2011,Abramovich-Fantechi2016}) we have
\begin{align}\label{eq degenerate formula}
\bl\mu\bb \tau_0([\D_0]\cdot \theta)\cdot\varpi\bb \nu\br_{g,\beta}
=\sum\bl\mu\bb \tau_0([\D_0]\cdot \theta)\cdot \varpi_1\bb \eta
\br_{g_1,\beta_1}^\bullet\fk z(\eta)
\bl \check\eta\bb   \varpi_2\bb \nu\br_{g_2,\beta_2}^\bullet
\end{align}
with summation taking over all splittings of $(g,\beta)$, all distribution of $\varpi$, all immediate cohomology weights and all configurations of connected components that yield a connected total domain.

There are some special summands in \eqref{eq degenerate formula} with $g_2=0$, $\varpi_2=\varnothing$, $\beta_2$ being a fiber class and $\vec\eta=\vec\nu$. For such a summand, the invariant
\[
\bl \check\eta \bb \varnothing \bb\nu
\br_{0,d[F]}^\bullet=\left\{\begin{array}{ll}
\frac{1}{\nu_1\cdot\ldots\cdot \nu_{\ell(\nu)}} &\textrm{if $\eta=\nu$,}\\
0                    &\textrm{if $\eta\neq\nu$.}
\end{array}\right.\]
Therefore there is no summands, of such type and with $\eta\neq \nu$, on the right side of \eqref{eq degenerate formula}. And when $\eta=\nu$ we get the original invariant \eqref{eq type-ii-inv-partial-order} on the right side of \eqref{eq degenerate formula}. From the summands in \eqref{eq degenerate formula} and $\Y_1\cong\Y$ we get a lot of distinguished type II invariants
\[
\bl \mu \bb \tau_0([\D_0]\cdot \theta)\cdot\varpi'\bb \nu'\br_{g',\beta'}.
\]
We say that these distinguished type II invariants are all lower than $\big\langle\mu \big| \tau_0([\D_0]\cdot \theta)\cdot\varpi \big| \nu \big\rangle_{g,\beta}$ and denote this relation by
\begin{align}\label{E partial-order}
\bl \mu \bb \tau_0([\D_0]\cdot \theta)\cdot\varpi'\bb \nu'\br_{g',\beta'}
\prec \bl\mu\bb \tau_0([\D_0]\cdot \theta)\cdot\varpi\bb \nu\br_{g,\beta}.
\end{align}
The following theorem is a special case of \cite[Theorem 6.5]{Chen-Du-Hu2019}.
\begin{theorem}\label{thm general-partial-order}
The relation \eqref{E partial-order} is a partial order among all distinguished type II invariants of $(\D_0|\Y|\D_\infty)$.
\end{theorem}

We will also discuss this partial order in \S \ref{sec rel-GW-of-weit-blp} for relative orbifold Gromov--Witten invariants of $(\uX_\wa|\D_\wa)$.

Given a distinguished type II invariant, by the Gromov compactness there are only finite distinguished type II invariants lower than it according to the ordering ``$\prec$''. However, according to this partial order, two comparable distinguished type II invariants have the same relative datum over $\D_0$. Next we modify the partial order ``$\prec$'' so that we could compare distinguished Type II invariants with different relative datum over $\D_0$.
\begin{defn}\label{def partial-order}
We say that
\[
\bl\mu'\bb \tau_0([\D_0]\cdot \theta')\cdot\varpi'\bb \nu'\br_{g',\beta'} \po
\bl\mu\bb \tau_0([\D_0]\cdot \theta)\cdot\varpi\bb \nu\br_{g,\beta},
\]
if
\begin{enumerate}
\item[(1)] $\beta'<\beta$, i.e. $\beta-\beta'$ is an effective class in
       $H_2(|\Y|;\integer)$,
\item[(2)]  equality in (1) and $g'<g$,
\item[(3)]  equality in (1)-(2) and $\|\varpi'\|<\|\varpi\|$,
\item[(4)]  equality in (1)-(3) and $\deg_{\text{CR}}\mu'>\deg_{\text{CR}}\mu$,
\item[(5)]  equality in (1)-(4) and $\deg_{\text{CR}}\nu'>\deg_{\text{CR}}\nu$,
\item[(6)]  equality in (1)-(5) and $\deg\theta'>\deg\theta$,
\item[(7)]  equality in (1)-(6), $\mu'=\mu$, $\theta'=\theta$ and
\[
\bl\mu\bb \tau_0([\D_0]\cdot \theta)\cdot\varpi'
\bb \nu'\br_{g',\beta'}\prec
\bl\mu\bb \tau_0([\D_0]\cdot \theta)\cdot\varpi\bb \nu\br_{g,\beta}.
\]
\end{enumerate}
This partial order could be generalized to invariants with disconnected domains directly.
\end{defn}

\begin{prop}
The relation ``$\po$'' is an partial order over all distinguished type II invariants. Moreover, given a distinguished type II invariant, there are only finite distinguished type II invariants are lower than it under ``$\po$''.
\end{prop}
The first assertion follows from Theorem \ref{thm general-partial-order} and the explicit inequalities in Definition \ref{def partial-order}. The second assertion follows from the Gromov compactness.

We next use the weighted-blowup correspondence in \cite{Chen-Du-Hu2019} to find three relations to determine all distinguished type II invariants.

\subsubsection{First relation}
By the weighted-blowup correspondence in \cite{Chen-Du-Hu2019}, the relative data $\nu$ at $\D_\0$ determines a sequence of absolute insertions relative to $\D_\0$, which we denote by $\nu_\0$. See \S \ref{sec rel-GW-of-weit-blp} for a brief introduction of weighted-blowup correspondence and the procedure that determines $\mu_\os$ via $\mu$ for a weight-$\wa$ blowup $\Xa$ of $\X$ along $\os$, in particular \eqref{eq cj}.

Now consider a distinguished type II invariant
\begin{align}\label{E a-distinguished-type-ii-inv}
\bl\mu\bb \tau_0([\D_0]\cdot \theta)\cdot\varpi\bb \nu\br_{g,\beta},
\end{align}
where $\theta\in H^{>0}(\D)$ and $\pi_*\beta\neq 0$. Suppose $\nu$ is of the form
\begin{align}\label{E nu}
\nu=\left( (\nu_1,\theta_{(h_1)}^{s_1}) ,\ldots, (\nu_{\ell(\nu)},\theta_{(h_{\ell(\nu)})}^{s_{\ell(\nu)}}) \right).
\end{align}
We could view $\Y$ as the weight-$(1)$ blowup of $\Y$ along $\D_\0$, i.e. the trivial weight blowup. Then following the weighted-blowup correspondence in \cite{Chen-Du-Hu2019} (see also \S \ref{sec rel-GW-of-weit-blp}) we set
\begin{align}
\label{eq rel-insertion-from-t-II}
\nu_\0:=\left(
\tau_{[\nu_1]-\fk l_{(h_1)}}([\D(h_1)_\0] \cdot \theta_{(h_1)}^{s_1}), \ldots,
\tau_{[\nu_{\ell(\nu)}]-\fk l_{(h_{\ell_{\nu}})}}
   ([\D(h_{\ell(\nu)})_\0]\cdot
    \theta_{(h_{\ell(\nu)})}^{s_{\ell(\nu)}})
\right),
\end{align}
where the integer $[\nu_j]-\fk l_{(h_j)}$, which indicates the power of the psi-class, is obtained by applying the formula \eqref{eq cj} of $c_j$ to the relative insertions in $\nu$ at $\D_\0\subseteq \Y$, and $[\nu_j]$ means the integral part of $\nu_j$. See for Definition \ref{def L} for the definition of $\fk l_{(h_j)}$.

For instance when applying the formula \eqref{eq cj}
\[
c_j=\sum_{i=1}^n\left[-\frac{\fk b(g_j\inv)_\N^i}{\fk o(g_j)}+a_i\mu_j\right] +n-1-m_j
\]
to $(\nu_j,\theta_{(h_j)}^{s_j})$ of $\nu$ in \eqref{E nu} we have $\N=\L^*$, $(g_j)=(h_j\inv)$, $\mu_j=\nu_j$, $n=\text{codim}_\cplane \D_\0=1$, $\wa=(a_1,\ldots,a_n)=(1)$, $m_j=0$. As $\nu_j$ is the contact order of the $j$-th marked point mapped to $\D_\0(h_j)$ we have
\[
\nu_j=\left\{\begin{array}{ll}
[\nu_j]+\frac{\fk o(h_j)-\fk b(h_j)_\L}{\fk o(h_j)}
&\textrm{if $1\leq \fk b(h_j)_\L<\fk o(h_j)$, i.e. $\fk l_{(h_j)}=0$,}\\
{[}\nu_j{]}
&\textrm{if $\fk b(h_j)_\L=\fk o(h_j)$, i.e. $\fk l_{(h_j)}=1$,}
\end{array}\right.
\]
where $\fk o(h_j)$ is the order of $(h_j)$ and $\fk b(h_j)_\L$ is the action weight of $h_j$ on $\L$. On the other hand, note that $\fk b(g_j\inv)_\N^u$ is the action weight of $(g_j\inv)$ on $\N$. So for the $\N=\L^*$ and $(g_j\inv)=(h_j)$ we have
\[
\fk b(h_j)_{\L^*}=\left\{\begin{array}{ll}
\fk o(h_j)-\fk b(h_j)_\L &\textrm{if $1\leq \fk b(h_j)_\L< \fk o(h_j)$, i.e. $\fk l_{(h_j)}=0$,}\\
\fk o(h_j)  &\textrm{if $\fk b(h_j)_\L=\fk o(h_j)$, i.e. $\fk l_{(h_j)}=1$,.}
\end{array}\right.
\]
Therefore by noticing that $\fk o(h_j)=\fk o(h_j\inv)$ and $\fk b(g_j\inv)_\N^1=\fk b(h_j)_{\L^*}$ we have
\begin{align*}
\sum_{i=1}^1\left[-\frac{\fk b(g_j\inv)_\N^i}{\fk o(g_j)}+a_i\nu_j\right]
=\left[-\frac{\fk b(h_j)_{\L^*}}{\fk o(h_j\inv)}+\nu_j\right]
=[\nu_j]-\fk l_{(h_j)}.
\end{align*}

Now via replacing the relative insertions in $\nu$ by the absolute insertions in $\nu_\0$ in \eqref{eq rel-insertion-from-t-II} we get a relative invariant of $(\Y|\D_0)$
\begin{align}\label{eq rel-inv-from-II}
\bl\mu \bb
\tau_0([\D_0]\cdot \theta)\cdot\varpi \cdot\nu_\0
\br_{g,\beta}.
\end{align}

We denote the distinguished type II invariant \eqref{E a-distinguished-type-ii-inv} by $\mathbf R$. In the following we view relative invariants of $(\Y|\D_0), (\Y|\D_\0)$ and $(\D_0|\Y|\D_\0)$ with genus $g$ and class $\beta$ as {\em principle terms}, and relative invariants of $(\Y|\D_0), (\Y|\D_\0)$ and $(\D_0|\Y|\D_\0)$ with
\[
\beta'<\beta,
\]
or
\[
\beta'=\beta \qq\mbox{and}\qq g'<g
\]
as {\em non-principle terms}.

\begin{relation}\label{relation 1}
We have
\begin{eqnarray*}
C\cdot \fk z(\nu)\cdot
\bl
\mu\bb \tau_0([\D_0]\cdot \theta) \cdot\varpi\bb \nu
\br_{g,\beta}
&=&\bl\mu\bb  \tau_0([\D_0]\cdot \theta) \cdot\varpi \cdot \nu_\0 \br_{g,\beta}
\\&&-
\sum_{\substack{\mathbf R'_{g,\beta}\text{distinguished type II}\\ \mathbf R'\stackrel{\circ}{\prec} \mathbf R}}C_{\mathbf R,\mathbf R'} \cdot \mathbf R' \\
&&-\sum_{\substack{\|\varpi'\|\leq \|\varpi\|\\
\deg_{\text{CR}}\mu'\geq \deg_{\text{CR}}\mu+1}}
C_{\mu',\varpi'}\cdot \bl \mu'\bb \varpi' \cdot \nu_\0\br_{g,\beta}\\&&
-\cdots,
\end{eqnarray*}
where\footnote{The relative invariants of this form were computed by Hu and the first two authors in \cite{Chen-Du-Hu2019}. See for \eqref{eq 1+1-point-rel-inv} the expression of these relative invariants.}
\[
C=\bl \check\nu\bb\nu_\0 \br_{0,d[F]}^\bullet =
\prod_{j=1}^{\ell(\nu)}  \frac{1}{\text{{ord}}(h_j)\cdot \nu_j!}\neq 0,
\]
with
\[
\nu_j!=\left\{\begin{array}{ll}
\nu_j\cdot\ldots\cdot(\nu_j-[\nu_j])&\textrm {if $[\nu_j]<\nu_j$,}\\
\nu_j\cdot\ldots\cdot 1 &\textrm {if $[\nu_j]=\nu_j$,}
\end{array}\right.
\]
$C_{\ast,\ast}$ are fiber class invariants of $(\Y|\D_0)$ or $(\D_0|\Y|\D_\0)$ and ``$\cdots$'' stands for combinations of non-principle relative invariants of $(\Y|\D_0)$ and non-principle distinguished Type II invariants.
\end{relation}

\begin{proof}
We degenerate $\Y$ as \eqref{eq degenerate Y}. Then the degeneration formula express the invariant \eqref{eq rel-inv-from-II} in terms of relative invariants of $(\D_0|\Y_1|\D_\0)$ and relative invariants of $(\Y_2|\D_0)$:
\[\begin{split}
\bl\mu \bb \tau_0([\D_0] \cdot \theta)\cdot\varpi\cdot \nu_\0 \br_{g,\beta}
=&\sum \bl \mu \bb \tau_0([\D_0] \cdot \theta)\cdot \varpi_1 \bb\eta \br^\bullet_{g_1,\beta_1} \cdot \fk z(\eta) \cdot \bl \check\eta \bb \varpi_2 \cdot \nu_\0 \br^\bullet_{g_2,\beta_2},
\end{split}\]
where the summation is over all splittings of $g$ and $\beta$, all distributions of the insertions of $\varpi$, all intermediate cohomology weighted partitions $\eta$, and all configurations of connected components that yield a connected total domain. The invariants on the right side are possible disconnected, indicated by the superscript $\bullet$. The subscript $g_i$ denotes the arithmetic genus of the total map to $\Y_i$.

The invariants of $(\D_0|\Y_1|\D_\0)$ are all distinguished type II invariants. Since we only concern principle terms. We could assume that $\beta_1=\beta$ or $\beta_2=\beta$.
\v\n
{\bf Case I: $\beta_1=\beta$.}
Let $\f_i\co \C_i\rto\Y_i$ be the elements of the relative moduli spaces for a fixed splitting. $\beta_1=\beta$ forces $\beta_2$ to be a fiber class. For the arithmetic genus of the glued stable map we have
\[
g=g_1+g_2+\ell(\eta)-1.
\]
Since $\beta_2$ is a fiber class, every connected component of $\C_2$ contains at least one relative marked point. Therefore the arithmetic genus satisfies
\[
g_2\geq 1-\ell(\eta).
\]
We conclude $g\geq g_1$ with equality if and only if $C_2$ consists of rational components, each totally ramified over $\D_0$ and every component contains exactly one relative marked point. If $g>g_1$, we get non-principle terms. Hence we only consider extremal configurations.

If any insertions of $\varpi$ is distributed into $\Y_2$, we get distinguished type II invariants coming from $\Y_1$ which are lower than $\mathbf R$ in the order ``$\po$''. Therefore we assume that all insertions in $\varpi$ are distributed to $\Y_1$. Hence the absolute insertions for $\Y_2$ all come from $\nu_\0$.

Now every connected component of $\C_2$ has exactly one relative marked point. Therefore the decomposition of $\check\eta$ (hence $\eta$) decomposes $\nu_\0$, hence $\nu$, into $\ell(\eta)$ components
\[
\nu=\coprod_{k=1}^{\ell(\eta)} \pi^{(k)}
\]
where empty partition is allowed as $\nu_\0$ are absolute insertions.

Write
\[
\eta^{(k)}=\left((\eta_k,\rho_k)\right),\qq \pi^{(k)}=\left((\pi_1^{(k)}, \theta_{(h_1^{(k)})}^{s_1^{(k)}}),\ldots, (\pi_{\ell(\pi^{(k)})}^{(k)}, \theta_{(h_{\ell(\pi^{(k)})}^{(k)})}^{s_{\ell(\pi^{(k)})}^{(k)}}) \right),
\]
where for a pair $(\ast,\star)$, ``$\ast$'' stands for the contact order and ``$\star$'' stands for the relative insertion, i.e. cohomological weight.

For $1\leq k\leq \ell(\eta)$, each component of $\C_2$ gives us a fiber class relative invariant of $(\Y|\D_0)$
\begin{align}\label{E fiber-class-inv-Relation-1}
\bl \check\eta^{(k)}\bb \pi^{(k)}_\0 \br_{0,d_k[F]}
\end{align}
where the homology class is determined by contact order $\eta_k$, the relative data is $\check\eta^{(k)}:=((\eta_k,\check\rho_k))$, and the absolute insertions $\pi^{(k)}_\0$ are determined by $\pi^{(k)}$ via \eqref{eq rel-insertion-from-t-II}. Denote the topological data of this fiber class invariant by $\Gamma^{(k)}$. Following \S \ref{sec fiber-class-inv}, the fiber class invariants are obtained by integrating insertion coming from fiber, descendent classes and Hodge class over the relative moduli space of $([\P^1\rtimes K_{\underline{\Gamma^{(k)}}}]|[0\rtimes K_{\underline{\Gamma^{(k)}}}])$ of topological type $\overline{\Gamma^{(k)}}$ (determined by $\Gamma^{(k)}$), and then integrating the results together with $e(E_{\underline{\Gamma^{(k)}}})$ and insertions from $\D$ (in fact $\I\D$) over $\D_{\underline{\Gamma^{(k)}}}'$.

Recall that in \S \ref{subsec notation-for-rel-inv-DYD} we have introduced the notation
\[
\deg_{\text{CR}}(\mu):=\sum_{j=1}^{\ell(\mu)}(\deg \theta_{(h_j)}+2\iota^\D(h_i))
\]
for a relative insertion $\mu=((\mu_1,\theta_{(h_1)}),\ldots, (\mu_{\ell(\mu)},\theta_{(h_{\ell(\mu)})}))$. Furthermore we set $\deg \mu=\sum\limits_{j=1}^{\ell(\mu)}\deg \theta_{(h_j)}$. Then for the fiber class relative invariant \eqref{E fiber-class-inv-Relation-1}, to get a non-zero invariants we must have
\[
\deg \check \eta^{(k)}+\deg \pi^{(k)}+\mbox{rank}\,E_{\underline{\Gamma^{(k)}}}\leq \dim_\rone\D_{\underline{\Gamma^{(k)}}}' =\dimr\D_{\underline{\Gamma^{(k)}}}.
\]
By the same proof of \cite[Theorem 3.2]{Hu-Wang2013} of Hu--Wang we have
\[
\dimc\D-\dimc\D_{\underline{\Gamma^{(k)}}} +\mbox{rank}_\cplane\,E_{\underline{\Gamma^{(k)}}}
=\iota^\D(\check\eta^{(k)})+\iota^\D(\pi^{(k)})
\]
where $\iota^\D(\check\eta^{(k)})$ is the degree shifting number of the twisted sector to which $\check\rho_k$ belongs and $\iota^\D(\pi^{(k)})$ is the sum of degree shifting numbers of the twisted sectors $\D(h^{(k)}_i),1\leq i\leq \ell(\pi^{(k)})$. As a consequence we get
\[
\deg \check\eta^{(k)}+\deg\pi^{(k)}\leq\dimr\D
-2\big(\iota^\D(\check\eta^{(k)})+\iota^\D(\pi^{(k)})\big),
\]
which is equivalent to
\[
\deg_{\text{CR}} \check\eta^{(k)}+\deg_{\text{CR}}\pi^{(k)}\leq \dimr\D.
\]
Hence by the orbifold Poincar\'e duality we get
\begin{align}\label{eq deg+deg<dimD--c}
\deg_{\text{CR}}\eta^{(k)}\geq \deg_{\text{CR}}\pi^{(k)}.
\end{align}
Summing over all components for $1\leq k\leq \ell(\eta)$ we get
\[
\deg_{\text{CR}}\eta \geq \deg_{\text{CR}}\nu.
\]

When there is at least one strictly inequality in \eqref{eq deg+deg<dimD--c}, we get a strictly lower distinguished Type II invariant from $\Y_1$. So we consider the case that equality holds in \eqref{eq deg+deg<dimD--c} for all $1\leq k\leq \ell(\eta)$. Then all equalities in (1)-(6) in Definition \ref{def partial-order} of the partial order ``$\po$'' holds, and we only have to consider the following summands in the degeneration formula
\begin{align}\label{eq degene-2}
\sum_{ \deg_{\text{CR}}\eta=\deg_{\text{CR}}\nu}
\bl
\mu\bb \tau_0([\D_0]\cdot\theta)\cdot\varpi \bb \eta \br^\bullet_{g,\beta}
\cdot\fk z(\eta)\cdot
\bl
\check\eta\bb \nu_\0
\br^\bullet_{0,d[F]}.
\end{align}
There is exactly one summand with $\eta=\nu$, for which (cf. \cite[Theorem 5.29]{Chen-Du-Hu2019})
\[
C=\bl \check\nu \bb \nu_\0\br^\bullet_{0,d[F]}
=\prod_{j=1}^{\ell(\nu)}\frac{1}{\text{ord}(h_j)\cdot \nu_j!}\neq 0.
\]
This gives us the left side of Relation \ref{relation 1}. For the rest summands in \eqref{eq degene-2}, we must have $\vec\eta\neq \vec\nu$, and then by definition of $\po$, we have
\[
\bl \mu \bb \tau_0([\D_0]\cdot\theta)\cdot\varpi \bb \eta \br^\bullet_{g,\beta} \po \bl \mu \bb \tau_0([\D_0]\cdot\theta)\cdot\varpi\bb \nu \br_{g,\beta}
\]
and $\bl\eta\bb \nu_\0 \br^\bullet_{0,d[F]}$ are fiber class invariants of $(\Y|\D_0)$, which are determined in \S \ref{sec fiber-class-inv}. This gives us the second term on the right side of the Relation \ref{relation 1}.

\v\n{\bf Case II: $\beta_2=\beta$.} The principal terms from $\Y_2$ will be shown to be of the form of the third term on the right side of Relation \ref{relation 1}.

Let $\f_i\co \C_i\rto\Y_i$ be the elements of the relative moduli spaces for a fixed splitting. The condition $\beta_2=\beta$ forces $\beta_1$ to be a multiple of the fiber class $[F]$. After ignoring lower terms, we may assume $\C_1$ consists of $\ell(\eta)$ rational components, each totally ramified over $\D_\0$ and contains exactly one relative making. Then the decomposition of $\eta$ induces a decomposition of $\mu$ into components
\[
\mu=\coprod_k\pi^{(k)}
\]
where empty weighted partitions are {\em not} allowed, since insertions in $\mu$ are relative insertions.

For $1\leq k\leq \ell(\eta)$, each component of $\f_1$ gives us a fiber class relative invariant of $(\D_0|\Y|\D_\0)$
\[
\bl \pi^{(k)}\bb\cdots\varpi_1^{(k)}\bb \eta^{(k)}\br_{0,d_k[F]}.
\]
Here ``$\cdots$'' is empty or $\tau_0([\D_0]\cdot\theta)$. We also denote the topological data of this fiber class invariant by $\Gamma^{(k)}$.

By the analysis in \S \ref{sec fiber-class-inv} and {\bf Case I}, to get a nonzero invariant we must have
\[
\deg \eta^{(k)}+\deg\pi^{(k)}+\mbox{rank}\,E_{\underline{\Gamma^{(k)}}} \leq \dim_\rone\D_{\underline{\Gamma^{(k)}}}' =\dimr\D_{\underline{\Gamma^{(k)}}}
\]
for those components which do not contain the insertion $\tau_0([\D_0]\cdot\theta)$, and
\[
\deg \eta^{(k)}+\deg\pi^{(k)}+\deg \theta+\mbox{rank}\,E_{\underline{\Gamma^{(k)}}} \leq \dim_\rone\D_{\underline{\Gamma^{(k)}}}'= \dimr\D_{\underline{\Gamma^{(k)}}}
\]
for the component which contains the insertion $\tau_0([\D_0]\cdot\theta)$.

Then since by the assumption $\deg_{\text{CR}}\theta=\deg\theta\geq 1$ we get
\[
\deg_{\text{CR}} \check\eta\geq\deg_{\text{CR}}\mu+1.
\]
This gives us the third term on the right side of Relation \ref{relation 1}.
\end{proof}

\subsubsection{Second relation}

Next we consider the following non-fiber relative invariant of $(\Y|\D_\0)$
\[
\bl \mu
\bb \tau_0([\D_0]\cdot\theta)\cdot\varpi\cdot\nu_\0 \br_{g,\beta},
\]
which is the first term on the right side of Relation \ref{relation 1}.

\begin{relation}\label{relation 2}
We have
\begin{eqnarray*}
&&
\bl \mu
\bb \tau_0([\D_0]\cdot\theta)\cdot\varpi\cdot\nu_\0 \br_{g,\beta}%
\\
&=&\pm\sum_{\substack{\|\varpi'\|\leq \|\varpi\|\\
             \deg_{\text{CR}}\mu'\geq \deg_{\text{CR}}\mu\\
             \deg_{\text{CR}}\nu'\geq \deg_{\text{CR}}\nu+1\\
                  m\geq 0}}
                  C_{\mu',\varpi',\nu'}
                  \cdot \bl\mu'\bb \tau_0([\D_0]\cdot
                  H\cdot c_1(\L)^m)\cdot\varpi' \bb \nu' \br_{g,\beta}\\
&&\pm\sum_{\substack{\|\varpi'\|\leq \|\varpi\|\\
             \deg_{\text{CR}}\mu'\geq \deg_{\text{CR}}\mu+1\\
                  \deg_{\text{CR}}\nu'\geq \deg_{\text{CR}}\nu\\
                  m\geq 0}}
                  C_{\mu',\varpi',\nu'}\cdot
                  \bl\mu'\bb \tau_0([\D_0]\cdot H\cdot
                  c_1(\L)^m)\cdot\varpi' \bb \nu' \br_{g,\beta}\\
&&\pm\sum_{\substack{\|\varpi'\|\leq \|\varpi\|\\
                  \deg_{\text{CR}}\mu'\geq \deg_{\text{CR}}\mu\\
                  \deg_{\text{CR}}\nu'\geq \deg_{\text{CR}}\nu\\
                  m\geq 0}}
                  C_{\mu',\varpi',\nu'}\cdot
                  \bl\mu'\bb \tau_0([\D_0]\cdot \theta
                  \cdot c_1(\L)^{m+1})\cdot\varpi' \bb \nu' \br_{g,\beta}\\
&&+\cdots
\end{eqnarray*}
where $H\in H^2(\D)$ is a class such that $\int_{\pi_*(\beta)}^{\mathrm{orb}}H>0$, $C_{\ast,\ast,\ast}$ are fiber class invariants of $(\Y|\D_0),(\Y|\D_\0)$, and $(\D_0|\Y|\D_\0)$, and ``$\cdots$" stands for combinations of non-principle distinguished type II invariants, non-principle relative invariants of $(\Y|\D_0),(\Y|\D_\0)$, and descendent orbifold Gromov--Witten invariants of $\D$.
\end{relation}

\begin{proof}
We use
\[
[\D_0]=[\D_\0]+c_1(\L)
\]
to write the invariant $\bl \mu\bb \tau_0([\D_0]\cdot\theta)\cdot\varpi\cdot\nu_\0 \br_{g,\beta}$
into
\[
I_1+I_2:=\bl \mu
\bb \tau_0([\D_\0]\cdot\theta)\cdot\varpi\cdot\nu_\0 \br_{g,\beta}+
\bl \mu
\bb \tau_0(c_1(\L)\cdot\theta)\cdot\varpi\cdot\nu_\0 \br_{g,\beta}.
\]

We first compute the invariant $I_1$ via virtual localization with respect to the canonical $\cplane^*$-action on $\Y$. The insertions of $I_1$ all have canonical equivariant lifts. The moduli space have two kind of fixed loci, the simple type $\M^\simple_\Gamma$ and the composite type $\F_{\Gamma',\Gamma''}$.

The simple type fixed locus $\M^\simple_\Gamma$ consists of those stable maps $\f\co \C_1\cup\C_2\rto\Y$ such that $\C_1$ is a disjoint union of rational components and the restriction $\f_1\co \C_1\rto(\D_0|\Y|\D_\0)$ are totally ramified over $\D_0$ and $\D_\0$, and the restriction $\f_2\co \C_2\rto\Y$ maps into $\D_\0$. The insertions $\tau_0([\D_\0]\cdot \theta)$ and $\nu_\0$ are all distributed to $\D_\0$. Therefore by the localization analysis in \S \ref{subsec localization} the contribution of $\M^\simple_\Gamma$ is Hodge integrals in the twisted Gromov--Witten theory of $\D_\0\cong\D$ with twisting coming from $\N_{\D_\0|\Y}=\L$. Then by the orbifold quantum Riemann--Roch \cite{Tseng2010} of Tseng, these twisted invariants reduced to descendent Gromov--Witten invariants of $\D$.

We next consider composite fixed loci. Since we only consider principle terms we only need to consider $\F_{\Gamma',\Gamma''} =gl(\M_{\Gamma'}^\simple \times_{\I\D^{\ell(\rho)}}\M_{\Gamma''}^\sim)$ with maps in $\M_{\Gamma''}^\sim$ having genus $g$ and degree $\beta$. Let $\C_0$ be the sub-curve of the domain mapped to rubber, and $\C_\0$ be the sub-curve of the domain mapped to $\D_\0$. Then $\C_0$ and $\C_\0$ is connected by a disjoint union of rational components $\C_1$ which are totally ramified over $\D_0$ and $\D_\0$.

$\C_\0\cup\C_1$ gives a morphism to $(\Y|\D_0)$ of fiber class. Note that the insertion $\tau_0([\D_\0]\cdot\theta)$ and insertions in $\nu_\0$ must be distributed to $\C_\0$. There would be some insertions $\varpi''$ of $\varpi$ are distributed to $\C_\0$ too. The rest $\varpi'$ of $\varpi$ are distributed to $\C_0$. We next insert relative insertions to the connecting nodal points $\{\msf n_1,\dots, \msf n_m\}$ of $\C_0$ and rational components $\C_1$ by assigning $\eta$ to $\C_0$ and $\check\eta$ to $\C_1$. Then the argument used for fiber class invariants in the proof of Relation \ref{relation 1} and the analysis for fiber class invariants in \S \ref{sec fiber-class-inv} shows
\[
\deg_{\text{CR}}\eta \geq \deg_{\text{CR}}^\D\varpi''+\deg_{\text{CR}}\nu +\deg_{\text{CR}}\theta \geq\deg_{\text{CR}}\nu+\deg\theta
\]
where $\deg_{\text{CR}}^\D \varpi''$ means the CR-degree of non-fiber parts of $\varpi''$, i.e. those classes pulled back from $H^*_{\text{CR}}(\D)$. Since $\deg \theta>0$, the principle term of the localization formula of $I_1$ is of the form
\[
\bl\mu\bb\varpi'\cdot\Psi_\0^k\bb \eta\br_{g,\beta}^\sim,\qq \mbox{where}\qq
\|\varpi'\|\leq \|\varpi\|,\,\, \deg_{\text{CR}}(\eta)\geq \deg_{\text{CR}}(\nu)+1,
\]
with coefficient being a fiber class relative invariant of $(\Y|\D_0)$. Then by repeating the rubber calculus for non-fiber class rubber invariants in \S \ref{subsubsec rubber-cal-non-fiber} we obtain the first term on the right hand side of Relation \ref{relation 2} modulo non-principle terms and Gromov--Witten invariants of $\D$.

We next compute $I_2$ by virtual localization too. As above, the contribution from simple fixed loci reduces to descendent orbifold Gromov--Witten invariants of $\D$. We next consider contribution from composite fixed loci. As above since we only consider principle terms we only need to consider $\F_{\Gamma',\Gamma''} =gl(\M_{\Gamma'}^\simple \wedge_{\I\D^{\ell(\rho)}}\M_{\Gamma''}^\sim)$ with maps in $\M_{\Gamma''}^\sim$ having genus $g$ and degree $\beta$. We also denote by $\C_0$  the sub-curve of the domain mapped to rubber, and $\C_\0$ the sub-curve of the domain mapped to $\D_\0$. Then $\C_0$ and $\C_\0$ is connected by a disjoint union of rational components $\C_1$ which are totally ramified over $\D_0$ and $\D_\0$. Then $\C_1\cup\C_\0$ also gives a morphism to $(\Y|\D_0)$ with fiber class. However, although the insertions in $\nu_\0$ must be distributed to $\C_\0$, the insertion $\tau_0(c_1(\L)\cdot\theta)$ may be distributed to $\C_\0$ or $\C_0$. Suppose there are some insertions $\varpi''$ of $\varpi$ are distributed to $\C_\0$ too. The rest $\varpi'$ of $\varpi$ are distributed to $\C_0$. We next insert relative insertions to the connecting nodes $\{\msf n_1,\dots, \msf n_m\}$ of $\C_0$ and rational components $\C_1$ by assigning $\eta$ to $\C_0$ and $\check\eta$ to $\C_1$.

If the insertion $\tau_0(c_1(\L)\cdot\theta)$ is distributed to $\C_\0$, by the argument used for fiber class invariants in the proof of Relation \ref{relation 1} and the analysis for fiber class invariants in \S \ref{sec fiber-class-inv} we have
\[
\deg_{\text{CR}}(\eta) \geq \deg_{\text{CR}}^\D(\varpi'')+\deg_{\text{CR}}(\nu) +\deg_{\text{CR}}(\theta) \geq\deg_{\text{CR}}(\nu)+\deg(\theta).
\]
Therefore the principle term coming from this case is of the form
\[
\bl\mu\bb\varpi'\cdot\Psi_\0^k\bb \eta\br_{g,\beta}^\sim,\qq \mbox{with}\qq
\|\varpi'\|\leq \|\varpi\|,\,\, \deg_{\text{CR}}(\eta)\geq \deg_{\text{CR}}(\nu)+1,
\]
whose coefficient are fiber class relative invariants of $(\Y|\D_0)$. Then by repeating the rubber calculus for non-fiber class we obtain the first term on the right side of Relation \ref{relation 2} modulo non-principle terms and descendent orbifold Gromov--Witten invariants of $\D$.

If the insertion $\tau_0(c_1(\L)\cdot\theta)$ is distribute to $\D_0$ we have
\[
\deg_{\text{CR}}(\eta) \geq \deg_{\text{CR}}^\D(\varpi'')+\deg_{\text{CR}}(\nu) \geq\deg_{\text{CR}}(\nu).
\]
Therefore the principle term coming from this case is of the form
\begin{align}\label{E rubber-inv-relation-II-I-2}
\bl\mu\bb\tau_0(c_1(\L)\cdot\theta)\cdot\varpi'\cdot\Psi_\0^k
\bb \eta\br_{g,\beta}^\sim,
\end{align}
with $\|\varpi'\|\leq \|\varpi\|,\,\, \deg_{\text{CR}}(\eta)\geq \deg_{\text{CR}}(\nu), k\geq 0$, whose coefficient are fiber class relative invariants of $(\Y|\D_0)$. By the boundary relation in \S \ref{subsubsec rubber-cal-fiber}, after modulo non-principle rubber invariants, the rubber invariants in \eqref{E rubber-inv-relation-II-I-2} reduce to rubber invariants of the form
\begin{enumerate}
\item[(i)]
    $\bl\mu\bb \tau_0(c_1(\L)\cdot\theta)\cdot\varpi''\bb \eta\br_{g,\beta}^\sim$, with
    $\|\varpi''\|\leq \|\varpi'\|\leq \|\varpi\|$ and $\deg_{\text{CR}}(\eta)\geq \deg_{\text{CR}}(\nu)$,

\item[(ii)] $\bl\mu'\bb\varpi''\cdot\Psi_\0^{k-1}
     \bb\eta\br_{g,\beta}^\sim$,
    with $\|\varpi''\|\leq \|\varpi'\|\leq \|\varpi\|$,
    $\deg_{\text{CR}}(\eta)\geq \deg_{\text{CR}}(\nu)$ and
    $\deg_{\text{CR}}(\mu')\geq \deg_{\text{CR}}(\mu)+1$.

\item[(iii)]
    $\bl\mu\bb\tau_0(c_1(\L)^2\cdot\theta)
    \cdot\varpi'\cdot\Psi_\0^{k-1}
    \bb\eta\br_{g,\beta}^\sim$, with
    $\|\varpi'\|\leq \|\varpi\|$ and
    $\deg_{\text{CR}}(\eta)\geq \deg_{\text{CR}}(\nu)$,

\end{enumerate}
These three kinds of rubber invariants reduce to the second and third summands of the right hand side of Relation \ref{relation 2} and non-principle terms as follows:
\begin{enumerate}
\item[(1)] By rigidification, i.e. Lemma \ref{lem rigidification}, the invariants in (i) give rise to
\[
\bl\mu\bb\tau_0([\D_0]\cdot c_1(\L)\cdot\theta)\cdot\varpi'
    \bb \eta\br_{g,\beta},
\]
with $\|\varpi'\|\leq \|\varpi\|$, and $\deg_{\text{CR}}(\eta)\geq \deg_{\text{CR}}(\nu)$. These are part of the third summand of the right hand side of Relation \ref{relation 2}.

\item[(2)] By repeating the rubber calculus for non-fiber class, we reduce the invariants in (ii) to rubber invariants of the form
\[
\bl\mu'\bb\tau_0(H)\cdot \varpi''\cdot \Psi_\0^{k-1}\bb\eta\br^\sim_{g,\beta},
\]
with $\|\varpi''\|\leq \|\varpi'\|\leq \|\varpi\|$, $\deg_{\text{CR}}\mu'\geq\deg_{\text{CR}}\mu+1$, and $\deg_{\text{CR}}\eta\geq\deg_{\text{CR}}\nu$. By using the boundary relation, modulo non-principle terms, these rubber invariants reduce to rubber invariants of the form
\begin{itemize}
\item $\bl\mu\bb \tau_0(H)\cdot\varpi'''\bb\eta\br_{g,\beta}^\sim$, with
    $\|\varpi'''\|\leq \|\varpi''\|$ and $\deg_{\text{CR}}(\eta)\geq \deg_{\text{CR}}(\nu)$,

\item $\bl\mu'\bb\varpi'''\cdot \Psi_\0^{k-2}\bb\eta\br_{g,\beta}^\sim$, with
    $\|\varpi'''\|\leq \|\varpi''\|$,
    $\deg_{\text{CR}}(\eta)\geq \deg_{\text{CR}}(\nu)$ and
    $\deg_{\text{CR}}(\mu')\geq \deg_{\text{CR}}(\mu)+1$.

\item $\bl\mu\bb\tau_0(H\cdot c_1(\L))\cdot\varpi''\cdot \Psi_\0^{k-2}\bb\eta\br_{g,\beta}^\sim$, with
    $\|\varpi''\|\leq \|\varpi'\|$ and
    $\deg_{\text{CR}}(\eta)\geq \deg_{\text{CR}}(\nu)$,
\end{itemize}
Then by applying the boundary relation and rigidification to these resulting rubber invariants, modulo non-principle terms we obtain the second and third summands on the right of Relation \ref{relation 2}.

\item[(3)] At last, applying the boundary relation to the invariants in (iii), modulo non-principle terms we get
\begin{itemize}
\item $\bl\mu\bb\tau_0(c_1(\L)^2\cdot\theta)\cdot\varpi''
    \bb\eta\br_{g,\beta}^\sim$, with
    $\|\varpi''\|\leq \|\varpi'\|$ and $\deg_{\text{CR}}(\eta)\geq \deg_{\text{CR}}(\nu)$,

\item $\bl\mu'\bb\varpi''\cdot \Psi_\0^{k-2}\bb\eta\br_{g,\beta}^\sim$,
    with $\|\varpi''\|\leq \|\varpi'\|$,
    $\deg_{\text{CR}}(\eta)\geq \deg_{\text{CR}}(\nu)$ and
    $\deg_{\text{CR}}(\mu')\geq \deg_{\text{CR}}(\mu)+1$.

\item $\bl\mu\bb\tau_0(c_1(\L)^2\cdot \theta)
    \cdot\varpi'\cdot\Psi_\0^{k-2}
    \bb\eta\br_{g,\beta}^\sim$, with
    $\|\varpi'\|\leq \|\varpi\|$ and
    $\deg_{\text{CR}}(\eta)\geq \deg_{\text{CR}}(\nu)$,
\end{itemize}
Similarly, by applying the boundary relation and rigidification to these resulting rubber invariants, modulo non-principle terms we obtain the second and third summands on the right of Relation \ref{relation 2}.
\end{enumerate}

This finishes the proof of Relation \ref{relation 2}.
\end{proof}

\subsubsection{Third relation}

Now we compute the relative invariants of $(\Y|\D_0)$ in the third term of the right side of Relation \ref{relation 1}, i.e. relative invariants of the form
\[
\bl\mu'\bb \varpi' \cdot \nu_\0 \br_{g,\beta}
\]
with $\deg_{\text{CR}}\mu'\geq \deg_{\text{CR}}\mu+1$ and $\|\varpi'\|\leq \|\varpi\|$. By the proof of Relation \ref{relation 2} we have
\begin{relation}\label{relation 3}
\[
\bl\mu'\bb \varpi' \cdot \nu_\0 \br_{g,\beta}
=\pm\sum_{\substack{\|\varpi''\|\leq \|\varpi'\|\\
            \deg_{\text{CR}}\mu''\geq \deg_{\text{CR}}\mu'\\
            \deg_{\text{CR}}\nu''\geq \deg_{\text{CR}}\nu\\
                  m\geq 0}}
                  C_{\mu'',\varpi'',\nu''}\cdot
                  \bl\mu''\bb \tau_0([\D_0]\cdot
                  H\cdot c_1(\L)^m)\cdot\varpi'' \bb \nu''\br_{g,\beta}
+\cdots
\]
where $H\in H^2(\D)$ satisfying $\int_{\pi_*(\beta)}^{\mathrm{orb}}H>0$, $C_{\mu'',\varpi'',\nu''}$ are fiber class invariants of $(\Y|\D_0),(\Y|\D_\0)$ and $(\D_0|\Y|\D_\0)$, and ``$\cdots$'' stands for non-principle terms and descendent orbifold Gromov--Witten invariants of $\D$.
\end{relation}

\subsection{Proof of Theorem \ref{thm rel-GW-of-P(E)} and Theorem \ref{thm orb-LH}}\label{subsec proof-thm-1.1}

In this subsection we give a proof of Theorem \ref{thm rel-GW-of-P(E)} and Theorem \ref{thm orb-LH}. We restate them here for reader's convenience.

\begin{theorem}[Theorem \ref{thm rel-GW-of-P(E)}]
The relative descendent orbifold Gromov--Witten theory of the pair
$(\oE_\wa|\PE_\wa)$ can be effectively and uniquely reconstructed from the absolute descendent orbifold Gromov--Witten theories of $\os$ and $\PE_\wa$, the Chern classes of $\E$ and $\mc O_{\PE_\wa}(-1)$.
\end{theorem}
\begin{proof}
By localization calculation in \S \ref{subsec localization}, relative invariants of $(\oE_\wa|\PE_\wa)$ are determined by descendent absolute invariants of $\os$ and rubber invariants of $(\D_0|\Y|\D_\0)=(\D_0|\P(\L\oplus\mc O_{\PE_\wa})|\D_\0)$ with $\L=\N_{\PE_\wa|\oE_\wa}=\mc O_{\PE_\wa}(1)$ and $\D=\PE_\wa\cong \D_0\cong \D_\0$. By the rubber calculus in \S \ref{subsec rubber-calculus}, these rubber invariants are reduced to fiber class invariants of $(\D_0|\Y|\D_\0)$ and distinguished type II invariants of $(\D_0|\Y|\D_\0)$. The fiber class invariants of $(\D_0|\Y|\D_\0)$ are determined in \S \ref{sec fiber-class-inv}. On the other hand, by Relation \ref{relation 1}, Relation \ref{relation 2} and Relation \ref{relation 3}, we determined distinguished type II invariants of $(\D_0|\Y|\D_\0)$ by induction with respect to the partial order $\po$ with initial values being fiber class invariants determined in \S \ref{sec fiber-class-inv}. This finishes the proof of Theorem \ref{thm rel-GW-of-P(E)}.
\end{proof}

\begin{theorem}[Theorem \ref{thm orb-LH}]
All four theories, i.e. the absolute descendent orbifold
Gromov--Witten theory of $\Y=\P(\L\oplus\mc O_\D)$ and the relative descendent orbifold Gromov--Witten theories of the three pairs
\[
(\Y|\D_0),\qq (\Y|\D_\0),\qq \text{and}\qq (\D_0|\Y|\D_\0),
\]
can be uniquely and effectively reconstructed from the absolute descendent orbifold Gromov--Witten theory of $\D$ and the first Chern class of the line bundle $\L$.
\end{theorem}
\begin{proof}
There is a $\cplane^*$-action on $\Y$ induced from the $\cplane^*$-dilation on $\L$.

First of all, for an absolute invariant of $\Y$, apply the virtual localization with respect to the $\cplane^*$-action. Then fixed loci of the corresponding moduli space is obtained by gluing moduli space of $\D_0$ and $\D_\0$ via disjoint union of genus zero 2-marked fiber class relative moduli space of $(\D_0|\Y|\D_\0)$. Then as in \S \ref{subsec localization},
the absolute invariants are determined by Hodge integrals of Gromov--Witten theories of $\D_0,\D_\0$, hence $\D$. By orbifold quantum Riemann--Roch, Hodge integrals are removed.

Next consider relative invariants. By relative virtual localization
\begin{enumerate}
\item[(a)] relative invariants of $(\Y|\D_0)$ are determined by Hodge integrals in twisted invariants of $\D_0\cong\D$ twisted by $\L$ and rubber invariants of $(\D_0|\Y|\D_\0)$.

\item[(b)] relative invariants of $(\Y|\D_\0)$ are determined by Hodge integrals in twisted invariants of $\D_\0\cong\D$ twisted by $\L^*$ and rubber invariants of $(\D_0|\Y|\D_\0)$.

\item[(c)] relative invariants of $(\D_0|\Y|\D_\0)$ are determined by Hodge integrals in twisted invariants of $\D_0\cong\D$ twisted by $\L$, Hodge integrals in twisted invariants of $\D_\0\cong\D$ twisted by $\L^*$ and rubber invariants of $(\D_0|\Y|\D_\0)$.
\end{enumerate}

By the rubber calculus in \S \ref{subsec rubber-calculus}, rubber invariants of $(\D_0|\Y|\D_\0)$ are determined by fiber class invariants of $(\D_0|\Y|\D_\0)$ and distinguished Type II invariants of $(\D_0|\Y|\D_\0)$. It is proved in \S \ref{sec fiber-class-inv} and
\S \ref{subsec dertermine-dis-type-ii} that fiber class invariants and
distinguished Type II invariants of $(\D_0|\Y|\D_\0)$ are all determined by Gromov--Witten theory of $\D$ and $c_1(\L)$. Again Hodge integrals are removed by orbifold quantum Riemann--Roch. This finishes the proof of Theorem \ref{thm orb-LH}.
\end{proof}

\section{Relative orbifold Gromov--Witten theory of weighted blowups}
\label{sec rel-GW-of-weit-blp}

In this section we determine all admissible relative descendent orbifold Gromov--Witten invariants of weighted blowups, hence prove Theorem \ref{thm rel-GW-of-weighted-blp}.

We first recall the notations. Let $\X$ be a compact symplectic orbifold groupoids with $\os$ being a symplectic suborbifold groupoid with codimension $\text{codim}\,\os=2n$, and normal bundle $\N$. Recall from \eqref{E degenerate-X-in-sec1}, the weight-$\wa$ blowup of $\X$ along $\os$ gives rise to a degeneration of $\X$ into
\[
\X\xrightarrow{\text{degenerate}} (\uX_\wa|\D_\wa)\wedge_{\D_\wa}(\oN_\wa|\D_\wa).
\]
Here $\uX_\wa$ is the weight-$\wa$ blowup of $\X$ along $\os$, $\oN_\wa$ is the weight-$\wa$ projectification of $\N$, $\D_\wa$ is the exceptional divisor in $\uX_\wa$ and also the infinity section $\PN_\wa$ of $\oN_\wa$, i.e. $\D_\wa=\PN_\wa$. So we have
\[
\N_{\D_\wa|\uX_\wa}=\N^\ast_{\D_\wa|\oN_\wa}=\mc O_{\PN_\wa}(-1),\qq\mbox{ and}\qq
c_1(\N_{\D_\wa|\uX_\wa})=-c_1(\N_{\D_\wa|\oN_\wa}).
\]

We have natural maps
\[
\kappa\co (\uX_\wa|\D_\wa)\rto(\X,\os),\qq \pi\co \oN_\wa\rto\os.
\]
Recall that in \S \ref{sec weitblp-and-chom} we have fixed a basis $\Sigma_\star$, i.e. \eqref{eq basis-of-PE}, of $H^*_{\text{CR}}(\D_\wa)=H^*_{\text{CR}}(\PN_\wa)$ via fixing a basis $\sigma_\star$ of $H^*_{\text{CR}}(\os)$. There is also a dual basis which we denoted by $\Sigma^\star$ in \S \ref{sec weitblp-and-chom}.

Let $\iota\co\os\rto\X$ being the inclusion map. Then we have an induced restriction map
\[
\iota^*\co H^*_{\text{CR}}(\X)\rto H^*_{\text{CR}}(\os).
\]
By orbifold Poincar\'e dual, $\iota^*$ determines the push-forward
\[
\iota_*\co H^*_{\text{CR}}(\os)\rto H^*_{\text{CR}}(\X).
\]
The total Chern class of $\N$ is also determined by $\iota^*$ via
\[
c(\N)\cup c(T\os)=\iota^*(c(T\X)).
\]
In particular,
\[
c_1(\N)=\iota^*(c_1(T\X))-c_1(T\os).
\]

We next prove Theorem \ref{thm rel-GW-of-weighted-blp} by using Theorem \ref{thm rel-GW-of-P(E)} and the weighted-blowup correspondence result in \cite{Chen-Du-Hu2019}. The main tool is the degeneration formula for orbifold Gromov--Witten theory in \cite{Chen-Li-Sun-Zhao2011,Abramovich-Fantechi2016}.

Take a relative invariant of $(\uX_\wa|\D_\wa)$
\begin{align}\label{eq rel-of-X/D}
\bl \omega\bb\mu\br_{g,\beta}^{(\uX_\wa|\D_\wa)}
\end{align}
with $\mu$ being weighted by the basis $\Sigma^\star$, the dual of the basis $\Sigma_\star$ of $H^*_{\text{CR}}(\D_\wa)$. The degeneration \eqref{E degenerate-Xa-in-sec1} is also a degeneration of $(\uX_\wa|\D_\wa)$
\[
(\uX_\wa|\D_\wa)\xrightarrow{\text{degenerate}} (\uX_\wa|\D_\wa)\wedge_{\D_\wa}(\D_{\wa,\0}|\overline{\mc O_{\D_\wa}(-1)}|\D_{\wa,0}).
\]
Here $\overline{\mc O_{\D_\wa}(-1)}:=\P(\mc O_{\D_\wa}(-1)\oplus
\mc O_{\D_\wa})$, $\D_{\wa,0}$ and $\D_{\wa,\0}$ are the zero and infinity sections of $\P(\mc O_{\D_\wa}(-1)\oplus
\mc O_{\D_\wa})$, and both are isomorphic to $\D_\wa$. The gluing is along the $\D_\wa\subseteq \uX_\wa$ and $\D_{\wa,\0}\subseteq \overline{\mc O_{\D_\wa}(-1)}$. Then by degeneration formula we get
\begin{align}\label{eq degene-3}
\bl \omega\bb\mu\br_{g,\beta}^{(\uX_\wa|\D_\wa)}
=\sum \bl \omega_1\bb\eta
\br_{g_1,\beta_1}^{\bullet, (\uX_\wa|\D_\wa)}
\cdot\fk z(\eta)\cdot
\bl \check \eta\bb\omega_2\bb\mu
\br_{g_2,\beta_2}^{\bullet,(\D_{\wa,\0}|\overline{\mc O_{\D_\wa}(-1)}|\D_{\wa,0})}
\end{align}
with summation taking over all splittings of $(g,\beta)$, and distributions of insertions $\omega$, and all immediate cohomology weights $\eta$ coming from the basis $\Sigma^\star$ of $H^*_{\text{CR}}(\D_\wa)$. Therefore the cohomology weights $\check\eta$ come from the basis $\Sigma_\star$. As in \S \ref{subsubsec partial order}, for a summand on the right hand side of\eqref{eq degene-3} of the form
\[
\bl \omega\bb\eta\br_{g,\beta}^{\bullet, (\uX_\wa|\D_\wa)}
\cdot\fk z(\eta)\cdot
\bl \check \eta\bb\varnothing
\bb\mu\br_{0,d[F]}^{\bullet,(\D_{\wa,\0}|\overline{\mc O_{\D_\wa}(-1)}|\D_{\wa,0})}
\]
such that every component of $\bl \check \eta\bb\varnothing \bb\mu\br_{0,d[F]}^{\bullet,(\D_{\wa,\0}|\overline{\mc O_{\D_\wa}(-1)}|\D_{\wa,0})}$ is a fiber class invariant and $\vec\eta=\vec\mu$ but $\eta\neq \mu$, we have
\[
\bl \check \eta\bb\varnothing
\bb\mu\br_{0,d[F]}^{\bullet,(\D_{\wa,\0}|\overline{\mc O_{\D_\wa}(-1)}|\D_{\wa,0})}=0.
\]
Therefore, such summand would not appear on the right hand side of \eqref{eq degene-3}. For the rest summands on the right side of \eqref{eq degene-3}, we define (cf. \cite[Definition 6.3]{Chen-Du-Hu2019})
\begin{align}\label{eq def-prec-sec4}
\bl \omega_1\bb\eta\br_{g_1,\beta_1}^{\bullet,(\uX_\wa|\D_\wa)} \prec
\bl \omega\bb\mu\br_{g,\beta}^{(\uX_\wa|\D_\wa)}.
\end{align}

The following theorem is \cite[Theorem 6.5]{Chen-Du-Hu2019}.
\begin{theorem}
The ``$\prec$'' in \eqref{eq def-prec-sec4} is a partial order on the set of (possibly disconnected) relative invariants of $(\uX_\wa|\D_\wa)$ of the form \eqref{eq rel-of-X/D}, i.e. the relative insertions coming from the chosen basis \eqref{eq basis-of-PE} of $H^*_{\mathrm{CR}}(\D_\wa)$. For a fixed relative invariants, there are only finite relative invariants are lower than it with respect to ``$\prec$''.
\end{theorem}

We restrict this partial order on the set of all admissible relative invariants of $(\uX_\wa|\D_\wa)$.

Now take an admissible relative invariant
\begin{align}\label{E an-admissible-reltive-inv-of-(Xa,Da)}
\bl \gamma  \bb\mu\br_{g,\beta}^{(\uX_\wa|\D_\wa)}
\end{align}
of $(\uX_\wa|\D_\wa)$ with $\gamma=\prod_i\tau_{k_i}\gamma_i$, $\gamma_i\in\mc K$ and $\mu$ weighted by the chosen basis $\Sigma^\star$. Then $\check\mu$ is weighted by the chosen basis $\Sigma_\star$ in \eqref{eq basis-of-PE}. Suppose
\[
\check\mu=\Big((\mu_1,\delta_{(g_1)}^{s_1}\cup H_{(h_1)}^{m_1}),\ldots,
(\mu_{\ell(\mu)},\delta_{(g_{\ell_{\mu}})}^{s_{\ell(\mu)}}
\cup H^{m_{\ell(\mu)}}_{(h_{\ell(\mu)})}) \Big).
\]
For each $1\leq j\leq \ell(\mu)$, since $\I\pi(h_j)=(g_j)$ and the contact order at $j$-th marked points is $\mu_j$, $(h_j)$ has a representative of the form $(g_j,\exp^{2\pi \sqrt{-1} \mu_j})$ in local group. We next assign to each $(\mu_j,\delta_{(g_j)}^{s_j}\cup H_{(h_j)}^{m_j})$ in $\check\mu$ a fiber class relative invariant of $(\oN_\wa|\D_\wa)$ as follows (cf. \cite[\S 6.2.1]{Chen-Du-Hu2019}).
\begin{enumerate}
\item[(a)] The topological data is $\Gamma_j=\big(g=0,A=\mu_j[F],(g_j\inv),(\mu_j,(h_j))\big)$, where $(g_j\inv)$ indicates the twisted sector of the unique absolute marked point and $(\mu_j,(h_j))$ indicates the contact order and twisted sector of the unique relative marked point.
\item[(b)] The relative insertion is $\delta_{(g_j)}^{s_j} \cup H_{(h_j)}^{m_j}$.
\item[(c)] The absolute insertion is $\tau_{c_j}(\iota_*(\check\delta_{(g_j)}^{s_j}))$, where $\check\delta_{(g_j)}^{s_j}\in H^*(\os(g_j\inv))$ is the orbifold Poincar\'e dual of $\delta_{(g_j)}^{s_j}$ in $H^*_{\text{CR}}(\os)$ (cf. \eqref{eq dual-basis-of-S}) and $c_j\in \integer_{\geq 0}$ is determined by the equation \eqref{eq cj} in the following.
\end{enumerate}
Denote the corresponding moduli space by $\M_{\Gamma_j}(\oN_\wa|\D_\wa)$. The number $c_j$ is determined by dimension constraint, i.e. we must have
\begin{align*}
&\text{virdim}_\cplane\M_{\Gamma_j}(\oN_\wa|\D_\wa)\\
&=\deg \tau_{c_j}(\iota_*(\check\delta_{(g_j)}^{s_j}))+\deg \delta_{(g_j)}^{s_j} \cup H_{(h_j)}^{m_j}
\\
&=c_j+\codimc(\os(g_j\inv),\X(g_j\inv ))
+\frac{1}2(\deg \delta_{(g_j)}^{s_j}+
\deg \check\delta_{(g_j)}^{s_j})+m_j\\
&=c_j+\codimc(\os(g_j),\N(g_j))
+\dimc \os(g_j)+m_j\\
&=c_j+\dimc \N(g_j)+m_j.
\end{align*}
The virtual dimension of such a moduli space $\M_{\Gamma_j}(\oN_\wa|\D_\wa)$ was computed in \cite[Proposition 5.11]{Chen-Du-Hu2019}. Suppose $g_j\inv$ acts on the fiber of $\N$ via
\[
g_j\inv\cdot(z_1,\ldots,z_n)
=(\exp^{2\pi \sqrt{-1} \frac{\fk b(g_j\inv)_\N^1}{\fk o(g_j)}} z_1,\ldots,\exp^{2\pi \sqrt{-1} \frac{\fk b(g_j\inv)_\N^n}{\fk o(g_j)}} z_n)
\]
where $\fk o(g_j)$ is the order of $g_j$ and $1\leq \fk b(g_j\inv)_\N^i\leq \fk o(g_j), 1\leq i\leq n$ are the action weights of $g_j\inv$ on $\N$. Then we have
\begin{align*}
\text{virdim}_\cplane\M_{\Gamma_j}(\oN_\wa|\D_\wa)
&= \sum_{i=1}^n\left[-\frac{\fk b(g_j\inv)_\N^i}{\fk o(g_j)}+a_i\mu_j\right]+ n-1+\dim_\cplane \N(g_j),
\end{align*}
where $[\cdot]$ means the integer part of a real number.
Therefore $c_j$ is determined by
\begin{align}\label{eq cj}
c_j=\sum_{i=1}^n\left[-\frac{\fk b(g_j\inv)_\N^i}{\fk o(g_j)}+a_i\mu_j\right] +n-1-m_j.
\end{align}
So to each $(\mu_j,\delta_{(g_j)}^{s_j}\cup H_{(h_j)}^{m_j})$ in $\check \mu$, the assigned fiber class relative invariant is
\[
\bl \tau_{c_j}(\iota_*(\check\delta_{(g_j)}^{s_j}))  \bb (\mu_j,\delta_{(g_j)}^{s_j}\cup H_{(h_j)}^{m_j})\br^{\oN_\wa|\D_\wa}_{\Gamma_j}
\]
with $c_j$ determined by \eqref{eq cj}. Now we glue these relative invariants to the relative invariant \eqref{E an-admissible-reltive-inv-of-(Xa,Da)} to get an absolute invariant of $\X$:
\begin{align}\label{eq abs-to-rel}
\bl \bar\gamma\cdot \mu_\os \br_{g,\beta}^{\X}
:=\bl \prod_i\tau_{k_i}\bar\gamma_i\cdot
\prod_j \tau_{c_j}(\iota_*(\check\delta_{(g_j)}^{s_j})) \br_{g,\beta}^{\X}
\end{align}
where $\bar\gamma_i$ is the pre-image of $\gamma_i$ in $H^*_{\text{CR}}(\X)$, and
\[
\mu_\os:=(\tau_{c_1}(\iota_*(\check\delta_{(g_1)}^{s_1})) ,\ldots,\tau_{c_{\ell(\mu)}} (\iota_*(\check\delta_{(g_{\ell(\mu)})}^{s_{\ell(\mu)}}))).
\]

Hu and the first two authors proved in \cite[Theorem 5.29, 6.10, 6.12]{Chen-Du-Hu2019} the following weighted-blowup correspondence results.
\begin{theorem}\label{thm 1-to-1 correspondence}
The correspondence
\[
\eqref{E an-admissible-reltive-inv-of-(Xa,Da)}\rto\eqref{eq abs-to-rel}
\]
is a one-to-one map from the set of admissible relative invariants of $(\uX_\wa|\D_\wa)$ to the set of absolute invariants of $\X$ of the form \eqref{eq abs-to-rel}. This means that $c_j$ is uniquely determined by
$(\mu_j,\delta_{(g_j)}^{s_j}\cup H_{(h_j)}^{m_j})$, i.e. when we fix $(g_j)$, hence $\fk o(g_j)$ and $\fk b(g_j\inv)_\N^u, 1\leq u\leq n$ are fixed, then $c_j$ is determined by these datum via \eqref{eq cj}. Conversely, given $c_j$ and $(g_j)$ we could get $(h_j), \mu_j$ and $m_j$ uniquely.

Moreover

\begin{align}\label{eq lower-triangle}
\bl&\bar\gamma\cdot\mu_\os \br_{g,\beta}^{\X}
=\bl  \gamma\bb\mu\br_{g,\beta}^{(\uX_\wa|\D_\wa)}\cdot \fk z(\mu)\cdot C+
\\&
\sum_{\bl \gamma^-\bb\eta\br_{g_1,\beta_1}^{\bullet, \uX_\wa,\D_\wa}\prec
\bl \gamma\bb\mu\br_{g,\beta}^{(\uX_\wa|\D_\wa)}}
\bl \gamma^-\bb\eta
\br_{g_1,\beta_1}^{\bullet, (\uX_\wa|\D_\wa)}
\cdot\fk z(\eta)\cdot
\bl \check \eta\bb\gamma^+\cdot
\prod_j \tau_{c_j}(\iota_*(\check \delta_{(g_j)}^{s_j}))
\br_{g_2,\beta_2}^{\bullet,(\oN_\wa|\D_\wa)}\nonumber
\end{align}
with
\begin{align}\label{eq 1+1-point-rel-inv}
C&=\bl \check \mu \bb \prod_j \tau_{c_j}(\iota_*(\check\delta_{(g_j)}^{s_j}))
\br_{0,d[F]}^{\bullet,(\oN_\wa|\D_\wa)}
=\prod_j\Big\{\frac{1}{\fk o(g_j)} \mu_j^{m_j}
\prod_{i=1}^n\frac{1}{(\frac{\fk b(g_j\inv)_\N^i}{\fk o(g_j)}+[-\frac{\fk b(g_j\inv)_\N^i}{\fk o(g_j)}+a_i\mu_j])!}\Big\}
\neq 0.
\end{align}
Here the definition of $a!$ for a rational number $a$ is the same as the one in Relation \ref{relation 1}, and $\gamma^-$ and $\gamma^+$ are extensions of parts of $\gamma$ over $\uX_\wa$ and $\oN_\wa$ respectively. For $\bar\gamma_i$ we have $\gamma_i^-=\gamma_i$. Therefore $\big\langle \gamma^-\big|\eta \big\rangle_{g_1,\beta_1}^{\bullet, (\uX_\wa|\D_\wa)}$ are admissible relative invariants of $(\uX_\wa|\D_\wa)$.
\end{theorem}

Now we can prove Theorem \ref{thm rel-GW-of-weighted-blp}. For reader's convenience we restate it here.
\begin{theorem}[Theorem \ref{thm rel-GW-of-weighted-blp}]
The admissible relative descendent orbifold Gromov--Witten theory of $ {(\uX_\wa| \D_\wa)}$ can be uniquely and effectively reconstructed from the orbifold Gromov--Witten theories of $\X$, $\os$ and $\D_\wa$, the restriction map $H^*_{\text{\em CR}}(\X)\rto H^*_{\text{\em CR}}(\os)$ and the first Chern class of $\mc O_{\D_\wa}(-1)$.
\end{theorem}
\begin{proof}
The equation \eqref{eq lower-triangle} gives us a lower triangular system determining the admissible relative invariants of $(\uX_\wa|\D_\wa)$ in terms of the Gromov--Witten invariants of $\X$ and the relative Gromov--Witten invariants of $(\oN_\wa|\D_\wa)$. By Theorem \ref{thm rel-GW-of-P(E)}, the relative Gromov--Witten invariants of $(\oN_\wa|\D_\wa)$ are determined by the Gromov--Witten invariants of $\os$ and $\D_\wa$, $c(\N)$ and $c_1(\N_{\D_\wa|\oN_\wa})=-c_1(\N_{\D_\wa|\uX_\wa})=-c_1(\mc O_{\D_\wa}(-1))$. Therefore, the relative Gromov--Witten theory of $(\uX_\wa|\D_\wa)$ can be uniquely and effectively reconstructed from the Gromov--Witten theories of $\X$, $\os$ and $\D_\wa$, the first chern class $c_1(\N_{\D_\wa|\uX_\wa})$ and the restriction map $H^*_{\text{CR}}(\X)\rto H^*_{\text{CR}}(\os)$. This finishes the proof.
\end{proof}

\bibliographystyle{abbrv}

\bibliography{chengyong}

\end{document}